\documentclass[11pt]{article}
\usepackage{amssymb,amsmath,amsthm,epsfig}
\usepackage{latexsym, enumerate}
\usepackage{booktabs}
\usepackage{eepic}
\usepackage{epic}
\usepackage{graphicx}
\usepackage{color}
\usepackage{ifpdf}
\usepackage{subfigure}
\usepackage{tikz}
\usepackage{dsfont}
\usepackage{multirow}
\usepackage{makecell}
\usepackage{algorithm}
\usepackage{bm}
\usepackage{multirow}
\usepackage{color}
\usepackage[colorlinks, linkcolor=blue,anchorcolor=blue,citecolor=blue,urlcolor=blue]{hyperref}
\usepackage{appendix}
\usepackage{comment}
\graphicspath{{figs/}}
\theoremstyle{plain}
\newtheorem{lem}{Lemma}[section]
\newtheorem{thm}[lem]{Theorem}

\theoremstyle{definition}
\newtheorem{defn}{Definition}[section]

\theoremstyle{remark}
\newtheorem{rem}{Remark}[section]

\newcommand{\p}{\partial}
\newcommand{\ds}{\displaystyle}

\newcommand{\ms}{\medskip}
\newcommand{\R}{ \mathbb{R}}

\def \e{\ensuremath{\mathrm{e}}}
\def \i{\ensuremath{\mathrm{i}}}
\def \d{\ensuremath{\mathrm{d}}}
\begin{document}

\title{ \large\bf  Enhanced Microscale Hydrodynamic Near-cloaking using Electro-osmosis
}
\author{
Hongyu Liu\thanks{Department of Mathematics, City University of Hong Kong, Kowloon, Hong Kong, China. Email: hongyu.liuip@gmail.com; hongyliu@cityu.edu.hk}
\and
Zhi-Qiang Miao\thanks{School of Mathematics, Hunan University, Changsha 410082, Hunan Province, China. Email: zhiqiang\_miao@hnu.edu.cn}
\and
Guang-Hui Zheng\thanks{School of Mathematics, Hunan University, Changsha 410082, Hunan Province, China. Email: zhenggh2012@hnu.edu.cn}
}
\date{}%
\maketitle
% ----------------------------------------------------------------
\begin{abstract}
In this paper, we develop a general mathematical framework for enhanced hydrodynamic near-cloaking of electro-osmotic flow for more complex shapes, which is  obtained by simultaneously perturbing the inner and outer boundaries of the perfect cloaking structure.
We first derive the asymptotic expansions of perturbed fields and obtain a first-order coupled system. We then establish the representation
formula of the solution to the first-order coupled system using the layer potential techniques. Based on the asymptotic analysis, the enhanced
hydrodynamic near-cloaking conditions are derived for the control region with general cross-sectional shape. The conditions reveal the inner relationship between the shapes of the object and the control region. Especially, for the shape of  a deformed annulus or confocal ellipses cylinder, the cloaking conditions and relationship of shapes are quantified more accurately. Our
theoretical findings are validated and supplemented by a variety of numerical results. The results in this paper also provide a mathematical foundation for more complex
hydrodynamic cloaking.
\end{abstract}

\section{Introduction}
Over the years, near-cloaking has been developed all the time along with perfect cloaking, although the latter is what people want most.
Many studies about near-cloaking have focused on regularized versions of a singular change-of-variables approach (transformation optics) in the literature \cite{Greenleaf2003, Pendry2006}. This singular transformation effectively blows up a point to a region in space that needs to be cloaked, which yields perfect cloaking; that is, the target region is rendered completely invisible to boundary measurements.  Later, in \cite{Kohn2008} Kohn et al. presented a regularized approximation by blowing up a tiny ball to a hidden region and studied the asymptotic behavior when the radius of the small ball tends to zero, therefore recovering the singular transform of \cite{Greenleaf2003, Pendry2006}. The proposed near-cloaking for the steady conduction problem is estimated to be $\epsilon^2$-close to the perfect cloaking in two-dimensional space.
The method is also extended to the Helmholtz equation \cite{Liu2009, Kohn2010}.
For the purpose of providing an approximation scheme for the singular transform in \cite{Greenleaf2003, Pendry2006}, Greenleaf et al. \cite{Greenleaf2008} used an alternative strategy that involved truncating singularities. It is worth noting that the structures in \cite{Kohn2008, Greenleaf2008} are proved to be  equivalent in  \cite{Kocyigit2013}. We direct the interested reader to the review papers \cite{Greenleaf2009-1, Greenleaf2009-2} for more information on cloaking via a change-of-variables method with a focus on the previously presented singular transform and briefly address some related studies on near-cloaking in acoustics and electromagnetic \cite{Li2013, Li2015, Liu2009, Liu2013-1, Liu2011, Deng2017-1, Deng2017-2}.

The enhancement of near-cloaking is another topic that has been addressed in the literature by the applied mathematics community working on metamaterials. In 2013, Ammari et al. proposed an enhancement technique that involves covering a small ball of radius $\epsilon$ with multiple coatings and then applying the push-forward maps defined in \cite{Kohn2008}. These multiple coatings which result in the vanishing of certain polarization tensors allow us to improve the $\epsilon^2$-closeness of \cite{Kohn2008} to $\epsilon^{2N}$-closeness in two-dimensional space, where $N$ denotes the number of coatings in the aforementioned structure. For further details, we refer the reader to \cite{Ammari2012-1, Ammari2012-2} in the mathematics literature. The numerical experiments also confirmed their results \cite{Ammari2012-3}. These enhancement techniques are a combination of scattering-cancellation technology and regularized change-of-variables approach.
Here we would like to briefly introduce scattering-cancellation technology. Scattering-cancellation technology has been created and successfully applied in the physics literature, for electromagnetism \cite{Alu2005, Alu2007} and other fields \cite{Chen2012}. This method can realize a similar function to transformation optics, while it only needs bilayer or monolayer structures and homogeneous isotropic bulk materials.
Furthermore, the enhancement method can also be extended to electromagnetism wave \cite{Ammari2013} which is akin to the Maxwell equation, and to Elastic wave \cite{Abbas2017, Liu2021} which is linked to the Lam\'{e} system. An alternate approach involving covering a small ball with a lossy layer with well-chosen parameters was employed by Liu et al. to enhance the near-cloaking in acoustics \cite{Li2012, Liu2013-2}. The lossy-layer cloaking scheme can help us improve the $|\ln\epsilon|^{-1}$ closeness of \cite{Kohn2010} to $\epsilon$-closeness in two-dimensional space.

Recently, there has been rapid progress in microscale hydrodynamic cloaking. The hydrodynamic model has been used to control fluid motion,
i.e., the creeping flow or Stokes flow inside two parallel plates, and a series of experimental works have been reported \cite{Park2019a, Park2019b, Park2021, Boyko2021}. The gap between the two plates is much smaller than the characteristic length of the other two spatial dimensions, so the model is also called the Hele-Shaw flow or Hele-Shaw cell \cite{Hele1898}. By using these microfluidic structures, Park et al. \cite{Park2019a, Park2019b} have demonstrated by simulation that such anisotropic fluid media can be mimicked within the cloak, thereby producing the desired hydrodynamic cloaking effect.
As we know, the cloaking devices designed by transformation optics are difficult to fabricate, which limits their application. Hence, there has been a growing interest in realizing metamaterial-less hydrodynamic cloaks. In particular, in \cite{Boyko2021} Boyko et al. present a new theoretical approach and an experimental demonstration of hydrodynamic cloaking and shielding in a Hele-Shaw cell that does not rely on metamaterials. The method has attracted our attention. We then develop a general mathematical framework \cite{Liu2023} for perfect and approximate hydrodynamic cloaking and shielding of electro-osmotic flow
in the spirit of Boyko's work.

This paper is a follow-up study of our earlier work \cite{Liu2023}, in which we studied perfect cloaking for concentric disks and confocal ellipses structures using analytic solution and approximate cloaking for general shapes by optimal method.  In this present work, we address the concept of enhanced near-cloaking in the context of microscale hydrodynamics using electro-osmosis by the perturbation theory. Our study is motivated by the physics literature \cite{Boyko2021}, in which authors studied the enhanced near-cloaking for annulus under a linear background field. The purpose of this paper is to extend the technique to a more general background field and study the case based on the perturbation of confocal ellipses simultaneously under this general background field. In order to
achieve enhanced invisibility, our construction of the near-cloaking structure is exactly different from the construction in \cite{Kohn2008}, which is linked closely to the study of a Poisson problem with a small volume defect. However, the near-cloaking in this paper is related to a small boundary defect. To the best of our knowledge, this is the first work to consider near-cloaking strategies by boundary perturbation in mathematics.  One could employ our constructions to the conductivity problem and scattering problem to obtain enhanced near-cloaking structures. This is left for future investigations. To
provide a global view of our study, the major contributions of this work can be summarised as
follows.
\begin{itemize}
  \item Based on the physics literature \cite{Boyko2021}, we give a rigorous mathematical definition of hydrodynamic near-cloaking. Especially we establish a unified mathematical framework for enhanced hydrodynamic near-cloaking with general geometry by utilizing asymptotic analysis theory.
  \item We rigorously derive the asymptotic expansion of the perturbed electric and pressure fields for the general domain. The representation formula of the solution to the first-order coupled system is obtained, which gives a quantitative analysis of the perturbed hydrodynamic model. Furthermore, the general conditions for enhanced hydrodynamic near-cloaking are derived, which reveal the inner relationship between the shapes of the core (object) and shell (cloaking region).
  \item By using the uniform approach---layer potential theory, we establish sharp conditions that can ensure the occurrence of the enhanced hydrodynamic near-cloaking for annulus (radial case) and confocal ellipses
  (non-radial case). Especially, for the confocal ellipses case which is not considered in \cite{Boyko2021}, we introduce an additional elliptic coordinates technique to overcome the difficulty caused by non-radial geometry.
\end{itemize}

The paper is organized as follows. We begin with the mathematical setting of the problem and briefly recall some known results in Section \ref{sec-setting-problem}. This section also makes precise the notion of near-cloaking and its connection to perfect cloaking followed by the construction of cloaking zeta potential. In Section \ref{Sec-Sensitivity-analysis} we rigorously derive the asymptotic expansion of the perturbed electric and pressure fields by two different methods. Section \ref{sec-hnc} is devoted to the study of the enhanced near-cloaking conditions
by the analytical method.  In Section \ref{sec-num-sim}, we present some numerical examples to illustrate our theoretical results. The paper is concluded in Section \ref{sec-conclusion} with some relevant discussions.

\section{Mathematical setting of the problem and preliminaries}\label{sec-setting-problem}
We consider a pillar-shaped object with arbitrary cross-sectional shape confined between the walls of a Hele-Shaw
cell and subjected to a non-uniform electro-osmotic flow with an externally imposed mean velocity $u_{ext}$ and electric field $\textbf{E}$ along the $x$-axis.
Applying the lubrication approximation, we average over the depth of the cell and reduce the analysis to a two-dimensional problem. The governing equations for the depth-averaged
velocity $u_{aver}$, the pressure $p$, and the electrostatic potential $\varphi$ are (see \cite{Liu2023})
\begin{equation}\label{dimensionless-equations}
 \bm u_{aver}  =- \frac{1}{12} \nabla p + \bm{u}_{slip},\quad \Delta p = -12 \nabla \varphi \cdot \nabla \zeta_{mean}  \quad \mbox{and} \quad \Delta \varphi = 0,
\end{equation}
where $\bm{u}_{slip}=-\zeta_{mean} \nabla \varphi$ is the depth-averaged Helmholtz-Smoluchowski slip velocity \cite{Boyko2021}. In addition, we assume that no penetration and insulation
occur at the object's surface and that the velocity and electric fields far from the object tend to a uniform externally applied velocity and an electric field.

%In this paper, we mainly consider the problem in which the perfect cloaking is perturbed due to the presence of a small boundary perturbation.
To mathematically state the problem, let $\Omega$ be a bounded domain in $ \mathbb{R}^2$ and let $D$ (object) be a domain whose closure is contained in $\Omega$. Throughout this paper, we assume that $\Omega$ and $D$ are of class $C^2$.
Let $H(x)$ and $P(x)$ be the harmonic function  in $\mathbb{R}^2$, denoting the background electrostatic potential and pressure field, and $\overline{D}\subset\Omega$. For a given constant parameter $\zeta_0$, the  zeta potential distribution in $ \mathbb{R}^2\setminus\overline{D}$ is given by
\begin{equation*}
 \zeta_{mean}=
  \begin{cases}
\ds \zeta_0, \quad \mbox{in } \Omega\setminus\overline{D},\\
\ds 0, \quad\mbox{in } \mathbb{R}^2\setminus\overline{\Omega}.
  \end{cases}
\end{equation*}
We may consider the configuration as an insulation and no-penetration core coated by the shell (control region) $\Omega \setminus\overline{D}$ with zeta potential $ \zeta_0$. Note that the continuity of the pressure and normal velocity is satisfied on $\p \Omega$.
From the equations \eqref{dimensionless-equations} and the assumption of boundary conditions, the governing equations for non-uniform electro-osmotic flow via a Hele-Shaw configuration are modeled as follows:
\begin{align}\label{electro-osmotic equation-original}
\begin{cases}
\ds \Delta \varphi = 0  & \mbox{in } \mathbb{R}^2\setminus\overline{D}, \ms \\
\ds \frac{\partial \varphi}{\partial \nu} = 0  & \mbox{on } \partial D, \ms \\
\ds  \varphi = H(x) + O(|x|^{-1}) & \mbox{as } |x|\rightarrow + \infty,\ms\\
\ds \Delta p= 0 & \mbox{in }  \mathbb{R}^2\setminus\overline{D}, \ms\\
\ds \frac{\partial p}{\partial \nu} = 0  & \mbox{on } \partial D, \ms\\
\ds p|_{+}=p|_{-} & \mbox{on }  \partial \Omega,\ms\\
\ds \frac{\partial p}{\partial \nu} \Big|_{+} - \frac{\partial p}{\partial \nu} \Big|_{-} = 12 \zeta_0\frac{\partial \varphi}{\partial \nu} & \mbox{on } \partial \Omega,\ms \\
\ds  p = P(x) + O(|x|^{-1}) & \mbox{as } |x|\rightarrow +\infty,
\end{cases}
\end{align}
where $\frac{\partial }{\partial \nu_{D}}$ and $ \frac{\partial }{\partial \nu_{\Omega}}$ denote the outward normal derivative on the boundary $\partial D$ and $\partial \Omega$, and the notations $p|\pm$ and $\frac{\partial p}{\partial \nu}\big|_{\pm}$ denote the traces on $\p \Omega$ from the outside and inside of $\Omega$, respectively.

In this paper, we consider an enhanced near-cloaking scheme of the hydrodynamic pressure field by perturbing the inner and outer boundaries of the perfect hydrodynamic cloaking structure. As discussed in the introduction,
this scheme was considered in the physics literature \cite{Boyko2021} for deformed cylinder under a linear background field. For self-containedness, we briefly discuss the perfect hydrodynamic cloaking
for the proposed enhanced near-cloaking scheme in the sequel, which can be found in our earlier work \cite{Liu2023}.

We are now in a position to introduce the definition of perfect hydrodynamic cloaking.
\begin{defn}\label{def-cloaking}
The triples $\{D, \Omega; \zeta_0\}$ is said to be a perfect hydrodynamic cloaking if
\begin{equation}\label{cond-cloaking}
\bm u_{aver}  = U \quad  \mbox{in }  \mathbb{R}^2\setminus\overline{\Omega},
\end{equation}
where $U= - \nabla P / 12$ denotes a uniform externally applied velocity.
\end{defn}
Outside the cloaking region, because the zeta potential is zero, from \eqref{dimensionless-equations} the pressure is related to the velocity field through $\bm{u}_{aver} = -\nabla p/12 $ subjected to the
boundary condition $p(x)=P(x)$ as $|x|\rightarrow\infty$. Therefore, according to the Definition \ref{def-cloaking}, the condition (\ref{cond-cloaking}) can be expressed in terms of the
pressure as
\begin{align*}
%\label{cond-cloaking-p}
p(x)=P(x), \quad x\in \mathbb{R}^2\setminus\overline{\Omega}.
\end{align*}

Here and throughout this paper, we assume that $\{D, \Omega; \zeta_0\}$ is a perfect hydrodynamic cloaking. Hence $p(x)=P(x), \ x\in \mathbb{R}^2\setminus\overline{\Omega}$.  We next consider the inner and outer boundaries perturbations of this perfect cloaking structure. For small $\epsilon \in \mathbb{R}_{+}$, we let $\partial D_\epsilon$ and $\partial \Omega_\epsilon$ be an $\epsilon$ -perturbation of $D$ and $\Omega$, respectively, i.e.,
\begin{align}
\label{pertur-D}
\partial D_\epsilon:=\{\tilde{x}=x+\epsilon f(x)\nu_D(x),\ \ &x\in \partial D\},\\
\partial \Omega_\epsilon:=\{\tilde{x}=x+\epsilon g(x)\nu_{\Omega}(x), \ \ &x\in \partial \Omega\},\label{pertur-Omega}
\end{align}
where $\nu_D$ and $\nu_{\Omega}$ are the outward unit normal vector to $\p D$ and $\p \Omega$; $f\in \mathcal{C}^1(\partial D)$ and $g\in \mathcal{C}^1(\partial \Omega)$ are called shape function of $D$ and $\Omega$ respectively. The $\epsilon$ -perturbation of $D$ can be treated as that is formed by tailoring delicately the basic shape of object $D$, and so is $\partial \Omega_\epsilon$.

Let $\varphi_\epsilon$ and $p_\epsilon$  be the solution to
\begin{align}\label{electro-osmotic equation}
\begin{cases}
\ds \Delta \varphi_\epsilon = 0  & \mbox{in } \mathbb{R}^2\setminus\overline{D_\epsilon}, \ms \\
\ds \frac{\partial \varphi_\epsilon}{\partial \nu_{D_\epsilon}} = 0  & \mbox{on } \partial D_\epsilon, \ms \\
\ds  \varphi_\epsilon = H(x) + O(|x|^{-1}) & \mbox{as } |x|\rightarrow + \infty,\ms\\
\ds \Delta p_\epsilon= 0 & \mbox{in }  \mathbb{R}^2\setminus\overline{D_\epsilon}, \ms\\
\ds \frac{\partial p_\epsilon}{\partial \nu_{D_\epsilon}} = 0  & \mbox{on } \partial D_\epsilon, \ms\\
\ds p_\epsilon|_{+}=p_\epsilon|_{-} & \mbox{on }  \partial \Omega_\epsilon,\ms\\
\ds \frac{\partial p_\epsilon}{\partial \nu_{\Omega_\epsilon}} \Big|_{+} - \frac{\partial p_\epsilon}{\partial \nu_{\Omega_\epsilon}} \Big|_{-} = 12 \zeta_0\frac{\partial \varphi_\epsilon}{\partial \nu_{\Omega_\epsilon}} & \mbox{on } \partial \Omega_\epsilon,\ms \\
\ds  p_\epsilon = P(x) + O(|x|^{-1}) & \mbox{as } |x|\rightarrow +\infty,
\end{cases}
\end{align}
where the zeta potential value remains the same and is given by
\begin{equation*}
 \zeta_{mean}=
  \begin{cases}
\ds \zeta_0, \quad \mbox{in } \Omega_\epsilon \setminus\overline{D_\epsilon},\\
\ds 0, \quad\mbox{in } \mathbb{R}^2\setminus\overline{\Omega_\epsilon}.
  \end{cases}
\end{equation*}
Then the hydrodynamic near-cloaking design (HNCD) problem is considered as follows.
\begin{defn}[HNCD]\label{HCDP}
Assume that the shape function $f\in \mathcal{C}^1(\partial D)$ is a priori known, find the shape function $g\in \mathcal{C}^1(\partial \Omega)$ such that
\begin{align}
\label{pcloak}
p_\epsilon(x)-P(x)=\mathcal{E}(x,\epsilon),\ \ \ \text{for}\ x\in \mathbb{R}^2\setminus\overline{\Omega},
\end{align}
where the error term (or scattering) $\mathcal{E}(x,\epsilon)$ satisfies $\mathcal{E}(x,\epsilon)\rightarrow 0$ as $\epsilon\rightarrow 0$. In particular, for $\mathcal{E}(x,\epsilon)=q_0(x)+q_1(x)\epsilon+q_2(x)\epsilon^2+\cdots$, if $q_0(x)=\cdots=q_{N-1}(x)=0$, and $q_j(x)$ is uniformly bounded for $j>N-1$, we call such a design scheme
hydrodynamic near-cloaking design (HNCD) of order $N$ or $N$-order HNCD is given. The $\infty$-order HNCD ($N=\infty$) is called perfect hydrodynamic cloaking design (PHCD). Furthermore, assume that $N-1$-order HNCD is given, and
$|q_{N-1}(x)|\leq |Q_{N-1}(x)|$, where $|Q_{N-1}(x)| = |Q_{N-1}^{(0)}(x)|+|Q_{N-1}^{(1)}(x)|$. If $Q_{N-1}^{(0)}(x)=0$ or $Q_{N-1}^{(1)}(x)=0$, then we call it weak $N$-order HNCD.
\end{defn}

\begin{rem}
(i) In fact, $q_0(x)\equiv 0$ in (\ref{pcloak}), since basic structure $\{D, \Omega; \zeta_0\}$ satisfies perfect cloaking, i.e., $1$-order HNCD always holds;\\
(ii) From Definition \ref{HCDP}, it is easy to see the weak $N$-order HNCD is an intermedium between $N-1$-order HNCD and $N$-order HNCD, i.e., weak $N$-order HNCD must be $N-1$-order HNCD, but may not be $N$-order HNCD;\\
(iii) Throughout this paper, since the lower order terms are vanishing in Definition \ref{HCDP}, we call the $N$-order and weak-$N$-order hydrodynamic near-cloaking  enhanced hydrodynamic near-cloaking for $N\geq2$.
\end{rem}

According to Definition \ref{HCDP}, we will utilize the asymptotic analysis
with respect to $\epsilon$ and find the shape function $g$ from the priori known shape function $f$, such that (\ref{pcloak}) holds for $N=2$.

We are now in a position to present the first main result of this paper on asymptotic expansions. The proofs are given in Subsections \ref{sec-FE} and \ref{Lptm}, respectively.
\begin{thm}\label{thm-expansion}
Let $\varphi_\epsilon$  and $p_\epsilon$ be the solutions to (\ref{electro-osmotic equation}). For $x\in  \mathbb{R}^2\setminus \overline{D}$, the following pointwise asymptotic expansions hold
\begin{align*}
  %\label{ep-expansion}
  \varphi_\epsilon(x)=\varphi(x)+\epsilon\varphi^{(1)}(x)+O(\epsilon^2),
\end{align*}
and
\begin{align}
\label{p-expansion}
  p_\epsilon(x)=p(x)+\epsilon p^{(1)}(x)+O(\epsilon^2),
\end{align}
  where the remainder $O(\epsilon^2)$ depends only on the $\mathcal{C}^2$-norm of $\partial D$, $\partial \Omega$ and $\mathcal{C}^1$-norm of $f$ and $g$. $\varphi$  and $p$ are the solutions to (\ref{electro-osmotic equation-original}), and the pair $(\varphi^{(1)}, p^{(1)})$ is the unique solution to the following first-order coupled system
\begin{align}\label{first-order-equation}
\begin{cases}
\ds \Delta \varphi ^{(1)}= 0  \quad \ &\mbox{in} \  \mathbb{R}^2\setminus\overline{D},\ms \\
\ds \frac{\partial \varphi^{(1)}}{\partial \nu_{D}} =E \quad &\mbox{on } \partial D,\ms \\
\ds  \varphi^{(1)} = O(|x|^{-1})\ \ &as\ |x|\rightarrow +\infty,\ms \\
\ds \Delta p^{(1)}= 0 \quad \ &\mbox{in} \  \mathbb{R}^2\setminus\overline{D},\ms \\
\ds \frac{\partial p^{(1)}}{\partial \nu_{D}} = A \quad &\mbox{on } \partial D,\ms \\
\ds p^{(1)}|_{+}-p^{(1)}|_{-}=B \quad &\mbox{on} \ \partial \Omega,\ms \\
\ds \frac{\partial p^{(1)}}{\partial \nu_{\Omega}} \big|_{+} - \frac{\partial p^{(1)}}{\partial \nu_{\Omega}} \big|_{-} = C \quad &\mbox{on } \partial \Omega, \ms \\
\ds  p^{(1)} = O(|x|^{-1})\ \ &as\ |x|\rightarrow +\infty,
\end{cases}
\end{align}
with
\begin{align}\label{boundary-term}
\begin{cases}
\ds E=f'\frac{\partial  \varphi}{\partial T_{D}}-  f \frac{\partial^2  \varphi}{\partial \nu^2_{D}} \quad &\mbox{on } \partial D,\ms \\
\ds A=f'\frac{\partial  p}{\partial T_{D}}-  f \frac{\partial^2  p}{\partial \nu^2_{D}}\quad &\mbox{on } \partial D, \ms \\
\ds B=g\Big(\frac{\partial p}{\partial\nu_{\Omega}}\Big|_{-}-\frac{\partial p}{\partial\nu_{\Omega}}\Big|_{+}\Big)\quad &\mbox{on} \ \partial \Omega,\ms \\
\ds C= g\Big(\frac{\partial^2  p}{\partial \nu^2_{\Omega}}\Big|_{-}- \frac{\partial^2  p}{\partial \nu^2_{\Omega}}\Big|_{+}\Big) + 12\zeta_0 \Big( \frac{\partial  \varphi^{(1)}}{\partial \nu_{\Omega}} +g \frac{\partial^2  \varphi}{\partial \nu^2_{\Omega}}-g'\frac{\partial  \varphi}{\partial T_{\Omega}}\Big)\quad &\mbox{on } \partial \Omega.
\end{cases}
\end{align}
\end{thm}

Note that $p=P$ in $\mathbb{R}^2\setminus\overline{D}$ since basic structure $\{D, \Omega; \zeta_0\}$ satisfies perfect cloaking. From the Definition \ref{HCDP} and asymptotic formula (\ref{p-expansion}), it is easy to obtain the following theorem which plays a central role in this paper.
\begin{thm}\label{thm-cloaking1}
Let $p$ and $p^{(1)}$ be defined in Theorem \ref{thm-expansion}. Given the shape function $f\in \mathcal{C}^1(\partial D)$, if there is a shape function $g\in \mathcal{C}^1(\partial \Omega)$ such that
\begin{align*}
p^{(1)}=0,\quad \ \mbox{in} \  \mathbb{R}^2\setminus\overline{\Omega},
\end{align*}
then 2-order HNCD is given.
\end{thm}

\begin{rem}
Notice that the shape functions $f$ and $g$ are implicit in the cloaking condition: $p^{(1)}=0$ ($x\in\mathbb{R}^2\setminus\overline{\Omega}$). Hence, it actually reveals the inner relationship between the shapes of the object and cloaking region.
\end{rem}

\begin{rem}
According to Definition \ref{HCDP}, when $p^{(1)}$ can be split into two parts $p^{(1,0)}$ and $p^{(1,1)}$ in Theorem \ref{thm-cloaking1}, and $p^{(1,0)}=0$ or $p^{(1,1)}=0$, the weak 2-order HNCD occurs.
\end{rem}

To that end, the rest of main results in this paper are given in the following theorems. The constructive proofs are given in Subsections \ref{subsec-annulus} and \ref{subsec-ellipse}, respectively.
\begin{thm}\label{thm-near-cloaking-circle}
Let the domains $D$ and $\Omega$ be concentric disks of radii $r_i$ and $r_e$, where $r_e > r_i$.
Let $H(x) = r^n\cos(n\theta)$ and $P(x)=12 r^n\cos(n\theta)$  (or $H(x) = r^n\sin(n\theta)$ and $P(x)=12 r^n\sin(n\theta)$) for $n\geq 1$. If the shape function $f\in \mathcal{C}^1(\partial D)$, then we can construct a shape function $g\in \mathcal{C}^1(\partial \Omega)$ such that 2-order HNCD can be achieved.
\end{thm}

\begin{thm}\label{thm-near-cloaking-ellipse}
Let the domains $D$ and $\Omega$ be confocal ellipses of elliptic radii $\xi_i$ and $\xi_e$, where $\xi_e > \xi_i$.
Let $H(x) = \cosh(n\xi)\cos(n\eta)$ and $P(x)=12 \cosh(n\xi)\sin(n\eta)$ (or $H(x) = \sinh(n\xi)\sin(n\eta)$ and $P(x)=12 \sinh(n\xi)\sin(n\eta)$) for $n\geq 1$. If the shape function $f\in \mathcal{C}^1(\partial D)$, then we can construct a shape function $g\in \mathcal{C}^1(\partial \Omega)$ such that weak 2-order HNCD can be achieved.
\end{thm}
\begin{rem}
In Theorem \ref{thm-near-cloaking-circle} and Theorem \ref{thm-near-cloaking-ellipse}, our proofs are constructive arguments. According to the proofs, the shape function $g$ can be constructed by recursive formulas.
\end{rem}

\subsection{Layer potentials formulation}
%\label{Layer Potentials}
In this section, we first collect some preliminary knowledge on boundary layer potentials and then recall
the representation formula of the solution to the governing equations.
For a bounded domain $\Gamma=D$ or $\Omega$ in $\R^2$, let us now introduce the single-layer and double-layer potential by
\begin{align*}
%\label{single layer}
\mathcal{S}_\Gamma[\vartheta](x):=\int_{\partial \Gamma}G(x,y)\vartheta(y)\d\sigma(y),\ \ \ &x\in \mathbb{R}^2,\\
\mathcal{D}_\Gamma[\vartheta](x):=\int_{\partial \Gamma}\frac{\partial G(x,y)}{\partial \nu_{\Gamma}(y)}\vartheta(y)\d\sigma(y),\ \ \ &x\in \mathbb{R}^2\setminus \partial \Gamma,
\end{align*}
where $\vartheta\in L^2(\p \Gamma)$ is the density function, and the Green function $G(x,y)$ to the Laplace in  $\mathbb{R}^2$ is given by
\begin{align*}
%\label{Gamma}
G(x,y)=\frac{1}{2\pi}\ln|x-y|.
\end{align*}
For a function $p$ defined on $\mathbb{R}^2\setminus\partial \Gamma$, we denote
\[
p|_{\pm}(x):=\lim_{t\rightarrow 0^+}p(x\pm t\nu_{\Gamma}(x)), \ \ x\in \partial\Gamma,
\]
and
\[
\frac{\partial p}{\partial \nu_{\Gamma}}\bigg|_{\pm}(x):=\lim_{t\rightarrow 0^+}\langle \nabla p(x\pm t\nu_{\Gamma}(x)),\nu_{\Gamma}(x) \rangle, \ \ x\in  \partial\Gamma,
\]
if the limits exist. Then the following jump relations hold :
\begin{align*}
%\label{Jump rela}
\frac{\partial\mathcal{S}_\Gamma[\vartheta]}{\partial \nu_{\Gamma}}\bigg|_{\pm}(x)=\Big(\pm\frac{1}{2}I+\mathcal{K}^*_\Gamma\Big)[\vartheta](x), \ \ &x\in \partial\Gamma,\\
\mathcal{D}_\Gamma[\vartheta]|_{\pm}(x)=\Big(\mp\frac{1}{2}I+\mathcal{K}_\Gamma\Big)[\vartheta](x), \ \ &x\in \partial\Gamma,
\end{align*}
where $\mathcal{K}^*_{\Gamma}$ is the $L^2$-adjoint of $\mathcal{K}_{\Gamma}$ and
\begin{align*}
%\label{NP}
\mathcal{K}_\Gamma^*[\vartheta](x)=\mathrm{p.v.} \int_{\partial \Gamma}\frac{\partial G(x,y)}{\partial \nu_{\Gamma} (x)}\vartheta(y)\d\sigma(y), \ \ x\in \partial\Gamma,
\end{align*}
where $\mathrm{p.v.}$ stands for the Cauchy principal value.

Next, from Theorem 3.2 in \cite{Liu2023}, we
know that the solution $\varphi_\epsilon$ to (\ref{electro-osmotic equation}) can be represented using the single layer potentials $S_{D_\epsilon}$ as follows:
\begin{equation*}
%\label{sol-varphi-epsilon}
\varphi_\epsilon (x)=
H(x) + \mathcal{S}_{D_\epsilon}[\phi_\epsilon](x),\quad x\in\mathbb{R}^2\setminus\overline{D_\epsilon},
\end{equation*}
where  density function $ \phi_\epsilon \in L_0^2(\p D_\epsilon)$ satisfies
\begin{align*}
%\label{e-potential-density}
\Big(\frac{1}{2} I+\mathcal{K}^*_{D_\epsilon}\Big)[\phi_\epsilon]=-\frac{\partial H}{\partial \nu_{D_\epsilon}} \quad \mbox{on }\p D_\epsilon,
\end{align*}
and $p_\epsilon$ can be represented using the single-layer potentials $S_{D_\epsilon}$ and $S_{\Omega_\epsilon}$ as follows:
\begin{equation*}
 %\label{sol-p-epsilon}
p_\epsilon(x) = P(x) +\mathcal{S}_{D_\epsilon}[\psi_{i,\epsilon}](x) + \mathcal{S}_{\Omega_\epsilon}[\psi_{e,\epsilon}](x),\quad  x\in  \mathbb{R}^2\setminus\overline{D_\epsilon},
\end{equation*}
where the pair $(\psi_{i,\epsilon}, \psi_{e,\epsilon})\in L_0^2(\p D_\epsilon)\times L_0^2(\p \Omega_\epsilon)$ satisfies
\begin{align}
\label{density-equation}
\begin{cases}
\ds \Big(\frac{1}{2} I+\mathcal{K}^*_{D_\epsilon}\Big)[\psi_{i,\epsilon}]
+\frac{\partial\mathcal{S}_{\Omega_\epsilon}[\psi_{e,\epsilon}]}{\partial\nu_{D\epsilon}}= -\frac{\p P}{\p \nu_{D\epsilon}} &\quad  \mbox{on }\p D_\epsilon, \ms\\
\ds \psi_{e,\epsilon} =12 \zeta_0\frac{\partial \varphi_\epsilon}{\partial \nu_{\Omega_\epsilon}} &\quad  \mbox{on }\p \Omega_\epsilon. \
\end{cases}
\end{align}
Furthermore, there exists a constant $C=C(\zeta_0,D_\epsilon,\Omega_\epsilon)$ such that
\begin{align*}
%\label{stability}
\|\psi_{i,\epsilon}\|_{L^2(\p D_\epsilon)}+\|\psi_{e,\epsilon}\|_{L^2(\p \Omega_\epsilon)}
\leq C \big(\| \nabla P \|_{L^2(\p D_\epsilon)}+\left\|\nabla H \right\|_{L^2(\p \Omega_\epsilon)}\big).
\end{align*}

\section{Enhanced hydrodynamic near-cloaking for general perturbed Geometry}\label{Sec-Sensitivity-analysis}
In this section, we rigorously derive the asymptotic expansions of the perturbed electric and pressure fields  and obtain a first-order coupled system by two different methods.  We first derive  formally the asymptotic expansions by  FE method \cite{Coifman1999} and then prove rigorously these results by using the layer-potential perturbation technique. The representation formulas of the solutions to the first-order coupled system are also obtained by the layer potential.

Let $\Psi_{D_\epsilon}$ and $\Psi_{\Omega_\epsilon}$ be the diffeomorphism from $\partial D$ to $\partial D_\epsilon$ and $\partial \Omega$ to $\partial \Omega_\epsilon$ given by
\begin{align*}
\Psi_{D_\epsilon}(x)=x+\epsilon f(x)\nu_{D}(x), \quad &x\in \partial D,\\
\Psi_{\Omega_\epsilon}(x)=x+\epsilon g(x)\nu_{\Omega}(x), \quad &x\in \partial \Omega.
\end{align*}
Moreover, we denote $\nu_{D_\epsilon}$ the outward unit normal vector field on  $\partial D_\epsilon$ and $\d\sigma_{D\epsilon}$ the line element of $\partial D_\epsilon$, the following expansions of $\nu_{D_\epsilon}$  and $\d\sigma_{D\epsilon} $ hold \cite{M. Lim2012}:
\begin{align}
\label{nu-D-eps}
\nu_{D_\epsilon} (\tilde{x})&=\nu_D(x)-\epsilon f'(x)T_D(x)+O(\epsilon^2),\\
\label{line-D-eps}\d\sigma_{D\epsilon}(\tilde{x})&=\d\sigma_D(x)-\epsilon\tau_D(x)f(x)\d\sigma_D(x)+O(\epsilon^2).
\end{align}
Similarly, we obtain
\begin{align}
\label{nu-Omega-eps}
\nu_{\Omega_\epsilon} (\tilde{x})&=\nu_\Omega(x)-\epsilon g'(x)T_\Omega(x)+O(\epsilon^2),\\
\label{line-Omega-eps}\d\sigma_{\Omega\epsilon}(\tilde{x})&=\d\sigma_\Omega(x)-\epsilon\tau_\Omega(x)g(x)\d\sigma_\Omega(x)+O(\epsilon^2).
\end{align}
Here and throughout this paper, $\tau_{D}(x)$ and $\tau_{\Omega}(x)$ denote the curvature of $\partial D$ and $\partial \Omega$ at $x$, $T_{D}$ and $T_{\Omega}$ are the unit tangential vector on $\partial D$ and $\partial \Omega$, respectively. $f'(x)$ is the tangential derivative of $f$ on $\partial D$, i.e., $f'=\frac{\partial f}{\partial T}$. So is $g'(x)$.

\subsection{Formal derivations: the FE method}\label{sec-FE}
In this subsection, we prove formally Theorem \ref{thm-expansion} based on the FE method.
We first derive formally the asymptotic expansion of $\varphi_\epsilon$, solution to (\ref{electro-osmotic equation}), as $\epsilon$ goes to zero. We start by expanding $\varphi_\epsilon$ in powers of  $\epsilon$, that is
\begin{align}
\label{ep-expansion-FE}
\varphi_\epsilon=\varphi^{(0)}+\epsilon\varphi^{(1)}+O(\epsilon^2),
\end{align}
where $\varphi^{(n)}$, $n = 0, 1$, are well defined in $ \mathbb{R}^2\setminus\overline{D}$, and satisfy
\begin{equation*}
\begin{cases}
\ds \Delta \varphi ^{(n)}= 0  \quad \ &\mbox{in} \  \mathbb{R}^2\setminus\overline{D},\\
\ds  \varphi^{(n)} =  \delta_{0,n} H(x) + O(|x|^{-1})\ \ &as\ |x|\rightarrow +\infty,\\
\end{cases}
\end{equation*}
Here, $\delta_{0,n}$ is the Kronecker symbol.

Let $\tilde{x}=x+\epsilon f(x)\nu_D(x) \in \partial D_\epsilon $ for $ x\in \partial D$. The normal derivative  $\frac{\partial \varphi_\epsilon}{\partial \nu_{D_\epsilon}}(\tilde{x})$ on $\partial D_\epsilon$ is given by
\begin{align}\label{normal-derivative-D-eps}
\frac{\partial  \varphi_\epsilon}{\partial \nu_{D_\epsilon}}(\tilde{x}) = \nabla \varphi_\epsilon (\tilde{x}) \cdot \nu_{D_\epsilon}(\tilde{x}),
\end{align}
where $\nu_{D_\epsilon}(\tilde{x})$ is the outward unit normal to $\partial D_\epsilon$ at $\tilde{x}$ defined by (\ref{nu-D-eps}).
To evaluate  $\nabla \varphi (\tilde{x})$  appearing in (\ref{normal-derivative-D-eps}),  we expand $\nabla \varphi_\epsilon$  around $\partial D$ and use (\ref{ep-expansion-FE}) to obtain
\begin{align}\label{varphi-gradient}
\nabla \varphi_\epsilon (\tilde{x}) = \nabla \varphi ^{(0)}(x) +\epsilon \nabla \varphi^{(1)} (x) + \epsilon f \nabla^2\varphi ^{(0)}\nu_{D}(x)+O(\epsilon^2), \quad x\in \partial D.
\end{align}
It then follows from (\ref{nu-D-eps}), (\ref{normal-derivative-D-eps}) and \eqref{varphi-gradient} that
\begin{align}\label{ep-normal-derivative-D-eps-expansion}
\frac{\partial  \varphi_\epsilon}{\partial \nu_{D_\epsilon}}(\tilde{x})
=&\frac{\partial  \varphi^{(0)}}{\partial \nu_{D}}(x)  + \epsilon \Big(\frac{\partial  \varphi^{(1)}}{\partial \nu_{D}}(x)   + f \frac{\partial^2  \varphi^{(0)}}{\partial \nu^2_{D}}(x)  -f'\frac{\partial  \varphi^{(0)}}{\partial T_{D}}(x) \Big) +O(\epsilon^2),  \quad x\in \partial D.
\end{align}
By using $\frac{\partial  \varphi_\epsilon}{\partial \nu_{D_\epsilon}}(\tilde{x}) =0$ on $\partial D_\epsilon$, we deduce from (\ref{ep-normal-derivative-D-eps-expansion})  that
\begin{align*}
\frac{\partial  \varphi^{(0)}}{\partial \nu_{D}}(x)=0, \quad &x\in \partial D,\\
\frac{\partial  \varphi^{(1)}}{\partial \nu_{D}}(x) = f'\frac{\partial  \varphi^{(0)}}{\partial T_{D}}(x)  - f \frac{\partial^2  \varphi^{(0)}}{\partial \nu^2_{D}}(x),  \quad &x\in \partial D.
\end{align*}
Note that $\varphi^{(0)} =\varphi$, which is the solution $\varphi$ to (\ref{electro-osmotic equation-original}).
In a similar way, we next expand $p_\epsilon$ in powers of  $\epsilon$, that is
\begin{align*}
%\label{p-expansion-FE}
p_\epsilon=p^{(0)}+\epsilon p^{(1)}+O(\epsilon^2),
\end{align*}
where $p^{(n)}$, $n = 0, 1$, are well defined in $ \mathbb{R}^2\setminus\overline{D}$, and satisfy
\begin{equation*}
\begin{cases}
\ds \Delta p ^{(n)}= 0  \quad \ &\mbox{in} \  \mathbb{R}^2\setminus\overline{D},\\
\ds  p^{(n)} =  \delta_{0,n} P(x) + O(|x|^{-1})\ \ &as\ |x|\rightarrow +\infty,\\
\end{cases}
\end{equation*}
Here, $\delta_{0,n}$ is the Kronecker symbol.

Let $\tilde{x}=x+\epsilon f(x)\nu_D(x) \in \partial D_\epsilon $, for $ x\in \partial D$. The normal derivative  $\frac{\partial p_\epsilon}{\partial \nu_{D_\epsilon}}(\tilde{x})$ on $\partial D_\epsilon$ is given by
\begin{align*}
%\label{p-normal-derivative-D-eps}
\frac{\partial  p_\epsilon}{\partial \nu_{D_\epsilon}}(\tilde{x}) = \nabla p_\epsilon (\tilde{x}) \cdot \nu_{D_\epsilon}(\tilde{x}),
\end{align*}
where $\nu_{D_\epsilon}(\tilde{x})$ is the outward unit normal to $\partial D_\epsilon$ at $\tilde{x}$ defined by  (\ref{nu-D-eps}).
Similarly, using \eqref{nu-Omega-eps} we obtain
\begin{align}\label{p-normal-derivative-D-eps-expansion}
\frac{\partial  p_\epsilon}{\partial \nu_{D_\epsilon}}(\tilde{x})
=&\frac{\partial  p^{(0)}}{\partial \nu_{D}}(x)  + \epsilon \Big(\frac{\partial  p^{(1)}}{\partial \nu_{D}}(x)   + f \frac{\partial^2  p^{(0)}}{\partial \nu^2_{D}}(x)  -f'\frac{\partial  p^{(0)}}{\partial T_{D}}(x) \Big) +O(\epsilon^2),  \quad x\in \partial D.
\end{align}
By using $\frac{\partial  p_\epsilon}{\partial \nu_{D_\epsilon}}(\tilde{x}) =0$ on $\partial D_\epsilon$, we deduce from (\ref{p-normal-derivative-D-eps-expansion}) that
\begin{align*}
\frac{\partial  p^{(0)}}{\partial \nu_{D}}(x)=0, \quad &x\in \partial D,\\
\frac{\partial  p^{(1)}}{\partial \nu_{D}}(x) = f'\frac{\partial  p^{(0)}}{\partial T_{D}}(x)  - f \frac{\partial^2  p^{(0)}}{\partial \nu^2_{D}}(x),  \quad &x\in \partial D.
\end{align*}

For $\tilde{x}=x+\epsilon g(x)\nu_\Omega(x) \in \partial \Omega_\epsilon $, we have the following Taylor expansions:
\begin{align*}
p_\epsilon|_{-}(\tilde{x}) =  p ^{(0)}\big|_{-}(x) + \epsilon  p^{(1)}\big|_{-}(x)+ \epsilon g \frac{\partial  p^{(0)}}{\partial \nu_{D}}\Big|_{-}(x) +O(\epsilon^2), \quad &x \in \partial \Omega,\\
p_\epsilon|_{+}(\tilde{x}) =  p ^{(0)}\big|_{+}(x) + \epsilon  p^{(1)}\big|_{+}(x)+ \epsilon g \frac{\partial  p^{(0)}}{\partial \nu_{D}}\Big|_{+}(x) +O(\epsilon^2), \quad &x \in \partial \Omega,
\end{align*}
and
\begin{align*}
\frac{\partial  p_\epsilon}{\partial \nu_{\Omega_\epsilon}}\Big|_{-}(\tilde{x})
=\frac{\partial  p^{(0)}}{\partial \nu_{\Omega}}\Big|_{-}(x) + \epsilon \Big(\frac{\partial  p^{(1)}}{\partial \nu_{\Omega}}\Big|_{-}(x) + g \frac{\partial^2  p^{(0)}}{\partial \nu^2_{\Omega}}\Big|_{-}(x)-g'\frac{\partial  p^{(0)}}{\partial T_{\Omega}}\Big|_{-}(x)\Big)+O(\epsilon^2),\quad &x\in \partial \Omega,\\
\frac{\partial  p_\epsilon}{\partial \nu_{\Omega_\epsilon}}\Big|_{+}(\tilde{x})
=\frac{\partial  p^{(0)}}{\partial \nu_{\Omega}}\Big|_{+}(x) + \epsilon \Big(\frac{\partial  p^{(1)}}{\partial \nu_{\Omega}}\Big|_{+}(x) + g \frac{\partial^2  p^{(0)}}{\partial \nu^2_{\Omega}}\Big|_{+}(x) -g'\frac{\partial  p^{(0)}}{\partial T_{\Omega}}\Big|_{+}(x)\Big)+O(\epsilon^2),\quad &x\in \partial \Omega,\\
\frac{\partial  \varphi_\epsilon}{\partial \nu_{\Omega_\epsilon}}(\tilde{x})
=\frac{\partial  \varphi^{(0)}}{\partial \nu_\Omega}(x)+ \epsilon \Big(\frac{\partial  \varphi^{(1)}}{\partial \nu_{\Omega}}(x) + g \frac{\partial^2  \varphi^{(0)}}{\partial \nu^2_{\Omega}}(x)-g'\frac{\partial  \varphi^{(0)}}{\partial T_{\Omega}}(x)\Big)+O(\epsilon^2),\quad &x\in \partial \Omega.
\end{align*}
The transmission conditions on $\Omega_\epsilon$ immediately yield
\begin{align*}
 p^{(0)}|_{-}=p^{(0)}|_{+} \quad  &\mbox{on }\partial\Omega,\\
 \frac{\partial p^{(0)}}{\partial \nu_{\Omega}} \Big|_{+} - \frac{\partial p^{(0)}}{\partial \nu_{\Omega}} \Big|_{-} = 12 \zeta_0
\frac{\partial \varphi^{(0)}}{\partial \nu_{\Omega}} \quad &\mbox{on }\partial\Omega,
\end{align*}
and
\begin{align*}
 p^{(1)}|_{+} - p^{(1)}|_{-}=g\Big(\frac{\partial p^{(0)}}{\partial\nu_{\Omega}}\Big|_{-}-\frac{\partial p^{(0)}}{\partial\nu_{\Omega}}\Big|_{+}\Big)   \quad & \mbox{on }\partial\Omega,\\
 \frac{\partial p^{(1)}}{\partial \nu_{\Omega}} \Big|_{+} - \frac{\partial p^{(1)}}{\partial \nu_{\Omega}} \Big|_{-} = g\Big(\frac{\partial^2  p^{(0)}}{\partial \nu^2_{\Omega}}\Big|_{-}- \frac{\partial^2  p^{(0)}}{\partial \nu^2_{\Omega}}\Big|_{+}\Big)
  +12\zeta_0 \Big( \frac{\partial  \varphi^{(1)}}{\partial \nu_{\Omega}} +g \frac{\partial^2  \varphi^{(0)}}{\partial \nu^2_{\Omega}}-g'\frac{\partial  \varphi^{(0)}}{\partial T_{\Omega}}\Big) \quad & \mbox{on }\partial\Omega,
\end{align*}
where we used $\frac{\partial  p^{(0)}}{\partial T_{\Omega}}\Big|_{+}=\frac{\partial  p^{(0)}}{\partial T_{\Omega}}\Big|_{-}$ on $\p\Omega$.
Note that $p^{(0)} =p$, which is the solution $p$ to (\ref{electro-osmotic equation-original}). Thus we formally obtain Theorem \ref{thm-expansion}, as desired. For a rigorous proof, see Section \ref{Lptm}.
\subsection{Layer potential and asymptotic analysis}\label{Lptm}
In this subsection, we prove rigorously Theorem \ref{thm-expansion} based on the layer potential techniques and establish
the representation formula of the solution to the first-order coupled system \eqref{first-order-equation}.
\subsubsection{Asymptotic expansion of electric potential}
For $\tilde{x}=x+\epsilon f(x)\nu_D(x) \in \partial D_\epsilon $, we have the following Taylor expansion
\begin{align*}
\frac{\partial H}{\partial \nu_{D_\epsilon}}(\tilde{x})
=\frac{\partial H}{\partial \nu_{D}}(x)+\epsilon \Big( f(x)\frac{\partial^2 H(x)}{\partial\nu_D^2(x)}-f'(x)\frac{\partial  H}{\partial T_D}(x)\Big)+O(\epsilon^2),\ x \in\partial D,
\end{align*}
where the remainder $O(\epsilon^2)$ depends only on the $\mathcal{C}^2$-norm of $\partial D$ and $\parallel f\parallel_{\mathcal{C}^1}$.

Let the solution $\varphi$ to (\ref{electro-osmotic equation-original}) be represented as
\begin{equation}\label{sol-varphi}
\varphi(x)=
H(x) + \mathcal{S}_{D}[\phi](x),\quad x\in\mathbb{R}^2\setminus\overline{D},
\end{equation}
where  density function $ \phi \in L_0^2(\p D)$ satisfies
\begin{align*}
%\label{e-potential-density}
\Big(\frac{1}{2} I+\mathcal{K}^*_{D}\Big)[\phi]=-\frac{\partial H}{\partial \nu_{D}} \quad \mbox{on }\p D.
\end{align*}

The following lemmas can be find in \cite{M. Lim2012}.
\begin{lem}\label{lem-KD-epsilon-expansion}
 For $\tilde{\phi}\in L^2(\partial D_\epsilon)$, let $\phi:=\tilde{\phi}\circ\Psi_{D_\epsilon}$. Then there exists a constant $C$ depending only on the $\mathcal{C}^2$-norm of $\partial D$ and $\parallel f\parallel_{\mathcal{C}^1}$ such that
\[
\parallel\big{(}\mathcal{K}_{ D_\epsilon}^{*}[\tilde{\phi}]\big{)}\circ\Psi_{D_\epsilon}-\mathcal{K}_{ D}^{*}[{\phi}]-\epsilon \mathcal{K}_{ D}^{(1)}[\phi]\parallel_{L^2(\partial D)}\leq C\epsilon^2\parallel \phi\parallel_{L^2(\partial D)},
\]
with the operator $\mathcal{K}_D^{(1)}$ defined for any $\phi\in \mathcal{C}(\partial D)$ by
\begin{align*}
%\label{k1}
\mathcal{K}_{D}^{(1)}[\phi](x)=\mathrm{p.v.}\int_{\partial D}K_1(x,y)\phi(y)d\sigma(y), \ \ \ x\in \partial D,
\end{align*}
where
\begin{align*}
K_1(x,y)=&-2\frac{\langle x-y,\nu(x)\rangle\langle x-y, f(x)\nu(x)-f(y)\nu(y)\rangle}{|x-y|^4}\\
&+\frac{\langle f(x)\nu(x)-f(y)\nu(y),\nu(x) \rangle}{|x-y|^2}-\frac{\langle x-y,\tau(x)f(x)\nu(x)+f'(x)T(x) \rangle}{|x-y|^2}\\
&+\frac{\langle x-y,\nu(x) \rangle}{|x-y|^2}(f(x)\tau(x)-f(y)\nu(y)).
\end{align*}
Here, $\mathrm{p.v.}$ denotes the Cauchy principal value.
\end{lem}
In fact, we can rewrite the operator $\mathcal{K}_{D}^{(1)}$ (see \cite{M. Lim2012}) in terms of more familiar operators as follows,
\begin{align}
\small
\label{KD1}
\mathcal{K}_{D}^{(1)}[\phi]=-\frac{\partial}{\partial T_D}\bigg{(} f\frac{\partial \mathcal{S}_D[\phi]}{\partial T_D}\bigg{)}+\frac{\partial \mathcal{D}_D[f\phi]}{\partial \nu_D}+\tau_Df\frac{\partial \mathcal{S}_{D}[\phi]}{\partial \nu_D}\Big|_{+}- \frac{\partial \mathcal{S}_D[\tau_Df \phi]}{\partial \nu_D}\Big|_{+}.
\end{align}
%The following lemma was also obtained in \cite{M. Lim2010}.
\begin{lem}\label{density asy}
 Let $\phi_\epsilon=-(\frac{1}{2} I+\mathcal{K}_{D_\epsilon}^*)^{-1}[\nabla H\cdot\nu_{D_\epsilon}]$ and $\phi=-(\frac{1}{2} I+\mathcal{K}^*_{D})^{-1}[\nabla H\cdot\nu_D]$. Then we have
\begin{align}\label{ep-density-expansion}
\parallel \phi_\epsilon\circ\Psi_{D_\epsilon}-\phi-\epsilon\phi^{(1)}\parallel_{L^2(\partial D)}\leq C\epsilon^2,
\end{align}
where $C$ is a constant depending only on the $\mathcal{C}^2$-norm of $\partial D$ and $\parallel f\parallel_{\mathcal{C}^1}$ and
\begin{align}
\label{asy density1}
\phi^{(1)}=-\Big(\frac{1}{2} I+\mathcal{K}_D^*\Big)^{-1}\bigg{(}f\frac{\partial^2 H}{\partial\nu_D^2}-f'\frac{\partial H}{\partial T_D}+\mathcal{K}_D^{(1)} \phi\bigg{)}.
\end{align}
\end{lem}
After the change of variables $\tilde{y}=\Psi_{D_\epsilon}(y)$, we obtain from \eqref{line-D-eps},  (\ref{ep-density-expansion}), and the Taylor
expansion of $G(x-\tilde{y})$ for $y \in \p D$, and $x\in \mathbb{R}^2\setminus\overline{D}$ fixed that
\begin{align*}
&\mathcal{S}_{D_\epsilon}[\phi_\epsilon](x)=\int_{\partial D_\epsilon}G(x,\tilde{y})\phi_\epsilon(\tilde{y})\d\sigma(\tilde{y})\\
 =&\int_{\partial D}\Big(G(x,y)+\epsilon f(y)\frac{\partial G(x,y)}{\partial\nu_D(y)}\Big)\Big(\phi(y)+\epsilon \phi^{(1)}(y)\Big)
 \Big(1-\epsilon\tau_D(y)f(y)\Big) \d\sigma(y) + O(\epsilon^2)\\
 =&\mathcal{S}_{D}[\phi](x)+\epsilon\Big(\mathcal{S}_{D}[\phi^{(1)}](x)-\mathcal{S}_{D}[\tau_{D}f\phi ](x)+\mathcal{D}_{D}[f\phi ](x)\Big)
+O(\epsilon^2).
\end{align*}
Hence from \eqref{sol-varphi} the following pointwise expansion holds for $x\in \mathbb{R}^2\setminus\overline{D}$:
\begin{align}\label{varphi-epsilon}
\varphi_\epsilon(x) =\varphi(x) +\epsilon\Big(\mathcal{S}_{D}[\phi^{(1)}](x)-\mathcal{S}_{D}[\tau_{D}f\phi ](x))+\mathcal{D}_{D}[f\phi ](x)\Big)+O(\epsilon^2).
\end{align}

We now prove the following representation theorem for the solution $\varphi^{(1)}$ to the first-order coupled system (\ref{first-order-equation}), which will be very helpful in
the proof of Theorem \ref{thm-expansion}.
\begin{thm}
%\label{thm-varphi1}
  The solution $\varphi^{(1)}$ to (\ref{first-order-equation}) is represented by
  \begin{align}\label{varphi1}
   \varphi^{(1)}=\mathcal{S}_{D}[\phi^{(1)}](x)-\mathcal{S}_{D}[\tau_{D}f\phi](x)+\mathcal{D}_{D}[f\phi](x), \quad x\in \mathbb{R}^2\setminus\overline{D},
  \end{align}
 where $\phi$ and $\phi^{(1)}$ are defined in Lemma \ref{density asy}.
\end{thm}
\begin{proof}
One can easily see that
  \[
  \Delta \varphi^{(1)}=0 \quad \mbox{in } \mathbb{R}^2\setminus\overline{D}.
  \]
Using (\ref{KD1}) and (\ref{asy density1}), we obtain
\begin{align*}
  \frac{ \partial \varphi^{(1)}}{\partial \nu_D}\Big|_{+}=&\frac{ \partial \mathcal{S}_{D}[\phi^{(1)}]}{\partial \nu_D}|_{+}- \frac{\partial \mathcal{S}_{D}[\tau_{D}f\phi]}{\partial \nu_D}\Big|_{+}+\frac{\partial \mathcal{D}_{D}[f\phi]}{\partial \nu_D}\\
  =&(\frac{1}{2} I+\mathcal{K}_D^*)[\phi^{(1)}] - \frac{\partial \mathcal{S}_{D}[\tau_{D}f\phi]}{\partial \nu_D}\Big|_{+}+\frac{\partial \mathcal{D}_{D}[f\phi]}{\partial \nu_D}\\
  =&\bigg{(}  f'\frac{\partial H}{\partial T_D}-f\frac{\partial^2 H}{\partial\nu_D^2}-\mathcal{K}_D^{(1)} \phi\bigg{)}- \frac{\partial \mathcal{S}_{D}[\tau_{D}f\phi]}{\partial \nu_D}\Big|_{+}+\frac{\partial \mathcal{D}_{D}[f\phi]}{\partial \nu_D}\\
  =& f'\frac{\partial H}{\partial T_D}-f\frac{\partial^2 H}{\partial\nu_D^2}+ \frac{\partial}{\partial T_D}\bigg{(} f\frac{\partial \mathcal{S}_D[\phi]}{\partial T_D}\bigg{)}-\tau_Df\frac{\partial \mathcal{S}_{D}[\phi]}{\partial \nu_D}\Big|_{+}\\
  =& f'\frac{\partial H}{\partial T_D}-f\frac{\partial^2 H}{\partial\nu_D^2} +f' \frac{\partial \mathcal{S}_D[\phi]}{\partial T_D} + f  \frac{\partial^2 \mathcal{S}_D[\phi]}{\partial  T^2_D}\\
=&f'\frac{\partial H}{\partial T_D} +f' \frac{\partial \mathcal{S}_D[\phi]}{\partial T_D} -f\frac{\partial^2 H}{\partial\nu_D^2} - f \frac{\partial^2 \mathcal{S}_D[\phi]}{\partial  \nu^2_D}.\\
=&f'\frac{\partial  \varphi}{\partial T_{D}}-  f \frac{\partial^2  \varphi}{\partial \nu^2_{D}}.
\end{align*}
Now, let us check the condition $ \mathcal{S}_{D}[\phi^{(1)}-\tau_{D}f\phi](x) \rightarrow 0 \ \mbox{as } |x|\rightarrow \infty$.
Since $\phi$ and $\phi^{(1)}\in L^2_0(\partial D)$, we have $\int_{\partial D}(\phi^{(1)}-\tau_{D}f\phi  )\d\sigma =0$.
Therefore,
\begin{align*}
 \mathcal{S}_{D}[\phi^{(1)}-\tau_{D}f\phi](x)= G(x)\int_{\partial D}(\phi^{(1)}-\tau_{D}f\phi  )\d\sigma  +O(|x|^{-1})=O(|x|^{-1}) \ \mbox{as } |x|\rightarrow \infty.
\end{align*}
Thus $\varphi^{(1)}$ defined by (\ref{varphi1}) satisfies $\varphi^{(1)}=O(|x|^{-1})$ as $|x|\rightarrow \infty$.

The proof is complete.
\end{proof}

\subsubsection{Asymptotic expansion of pressure}
For $\tilde{x}=x+\epsilon f(x)\nu_D(x) \in \partial D_\epsilon $, we have the following Taylor expansion
\begin{align}
%\label{expansion-P}
%P(\tilde{x})=&P(x)+\epsilon g(x)\frac{\partial P}{\partial\nu_{\Omega}}(x)+O(\epsilon^2)\ \ \ x\in \partial \Omega,\\
\label{expansion-P-normal-d}
\frac{\partial P}{\partial \nu_{D_\epsilon}}(\tilde{x})
=&\frac{\partial P}{\partial \nu_{D}}(x)+\epsilon \Big(f(x)\frac{\partial^2 P(x)}{\partial\nu_D^2(x)}-f'(x)\frac{\partial  P}{\partial T_D}(x)\Big)+O(\epsilon^2),\ \ \ x \in\partial D,
\end{align}
where the remainder $O(\epsilon^2)$ depends only on the $\mathcal{C}^2$-norm of $\partial D$ and $\parallel f\parallel_{\mathcal{C}^1}$.

Let the solution $p$ to (\ref{electro-osmotic equation-original}) be represented as
\begin{equation}\label{sol-p}
p=P(x)+
\mathcal{S}_{D}[\psi_{i}](x) + \mathcal{S}_{\Omega_\epsilon}[\psi_{e}](x),\quad  x\in \mathbb{R}^2\setminus\overline{D},
\end{equation}
where the pair $(\psi_{i}, \psi_{e})\in L_0^2(\partial D)\times L_0^2(\partial \Omega)$ satisfies
\begin{align}
\label{boundary-integral-equation}
\begin{cases}
\ds \Big(\frac{1}{2}I + \mathcal{K}_{D}^*\Big)[\psi_{i}]
+\frac{\partial\mathcal{S}_{\Omega}[\psi_{e}]}{\partial\nu_{D}}= - \frac{\partial P}{\partial\nu_{D}}\quad  &\mbox{on }\partial D,\\
\ds \psi_{e}= 12 \zeta_0
\frac{\partial \varphi}{\partial \nu_{\Omega}} \quad  &\mbox{on } \partial \Omega.
\end{cases}
\end{align}

Now, we first introduce an integral operator $\mathcal{A}_{\Omega}$, defined for any $ \psi_e \in L^2_0(\p \Omega)$, by
\begin{align*}
\mathcal{A}_{\Omega}[\psi_e](x)=-\frac{\partial \mathcal{S}_{\Omega}[\tau_{\Omega}g\psi_e]}{\partial \nu_{D}}(x)+ f(x)\frac{\partial^2 \mathcal{S}_{\Omega}[\psi_e] }{\partial \nu^2_{D}}(x)-f'(x) \frac{\partial\mathcal{S}_{\Omega}[\psi_e]}{\partial T_{D}}(x) + \frac{\partial\mathcal{D}_{\Omega}[g\psi_e]}{\partial \nu_{D}}(x), \quad x\in\p D.
\end{align*}
Next, introducing the pair $\big(\psi_i^{(1)},\psi_e^{(1)}\big)$ as a solution to the following system:
\begin{align}
\label{first-order-density-equation}
\begin{cases}
\ds \Big(\frac{1}{2}I+\mathcal{K}^*_{D}\Big)[\psi_i^{(1)}]
+\frac{\partial\mathcal{S}_{\Omega}[\psi_e^{(1)}]}{\partial\nu_{D}}=-\mathcal{K}^{(1)}_{D}[\psi_i] -\mathcal{A}_{\Omega}[\psi_e]-f\frac{\partial^2 P}{\partial \nu^2_D}+f'\frac{\partial P}{\partial T_D}\ \ &\mbox{on } \partial D,\\
\ds \psi_e^{(1)}=12\zeta_0\Big(\frac{\partial \varphi^{(1)}}{\partial \nu_\Omega}+ g\frac{\partial^2 \varphi}{\partial \nu^2_\Omega}-g'\frac{\partial \varphi}{\partial T_\Omega}\Big)\ \ &\mbox{on } \partial \Omega,
\end{cases}
\end{align}
where the pair $(\psi_i,\psi_e)$ is the solution to (\ref{boundary-integral-equation}).

It follows from \eqref{density-equation}, (\ref{first-order-density-equation}) and Lemma \ref{lem-KD-epsilon-expansion} that
\begin{align}
\label{diff}
\begin{cases}
\ds \Big(\frac{1}{2}I+\mathcal{K}^*_{D}+\epsilon \mathcal{K}^{(1)}_{D}\Big)\big[\tilde{\psi}_{i}-\psi_i- \epsilon\psi_i^{(1)}\big]
+\Big(\frac{\partial\mathcal{S}_{\Omega}}{\partial\nu_{D}}+\epsilon\mathcal{A}_{\Omega}\Big)\big[\tilde{\psi}_{e}-\psi_e-\epsilon\psi_e^{(1)}\big]\\
\qquad\qquad \qquad\qquad \qquad\qquad\quad\ \ = -\frac{\partial P}{\partial \nu_D}\circ \Phi_{\epsilon} + \frac{\partial P}{\partial \nu_D} + \epsilon\Big(f\frac{\partial^2 P}{\partial \nu^2_D}-f'\frac{\partial P}{\partial T_D}\Big) + O(\epsilon^2)\ \ &\mbox{on } \partial D,\\
\ds \tilde{\psi}_{e}-\psi_e-\epsilon\psi_e^{(1)} =12\zeta_0\Big(\frac{\partial \varphi_\epsilon}{\partial \nu_\Omega}\circ \Phi_{\epsilon} -\frac{\partial \varphi}{\partial \nu_\Omega}-\epsilon\Big( \frac{\partial \varphi^{(1)}}{\partial \nu_\Omega}+ g\frac{\partial^2 \varphi}{\partial \nu^2_\Omega}-g'\frac{\partial \varphi}{\partial T_\Omega}\Big)\Big)+O(\epsilon^2)\ \ &\mbox{on } \partial \Omega,
\end{cases}
\end{align}
where $\tilde{\psi}_{i} = \psi_{i,\epsilon} \circ \Psi_{D_\epsilon}$ and $\tilde{\psi}_{e} = \psi_{e,\epsilon} \circ \Psi_{\Omega_\epsilon}$.

The following lemma follows immediately from \eqref{expansion-P-normal-d} and \eqref{diff}. For detailed proof, we refer the reader to Zribi \cite{Zribi2016}.
\begin{lem}
%\label{lem-K}
Let $(\psi_{i,\epsilon}, \psi_{e,\epsilon})$, $(\psi_{i}, \psi_{e})$, $(\psi_{i}^{(1)}, \psi_{e}^{(1)})$ be the solutions to \eqref{density-equation}, \eqref{boundary-integral-equation}, \eqref{first-order-density-equation}, respectively. Then there exists a constant $C$ depending only on the $\mathcal{C}^2$-norm of $\partial D$, $\partial \Omega$ and $\parallel f\parallel_{\mathcal{C}^1}$, $\parallel g\parallel_{\mathcal{C}^1}$ such that
\begin{align}\label{p-density-expansion}
  \|\psi_{i,\epsilon}\circ\Psi_{D_\epsilon}-\psi_i -\psi_i^{(1)}\|_{L^2(\p D)}+   \|\psi_{e,\epsilon}\circ\Psi_{\Omega_\epsilon}-\psi_e -\psi_e^{(1)}\|_{L^2(\p \Omega)}\leq C \epsilon^2.
\end{align}
\end{lem}

After the change of variables $\tilde{y}=\Psi_\epsilon(y)$, we obtain from \eqref{line-D-eps}, \eqref{line-Omega-eps}, (\ref{p-density-expansion}) and the Taylor
expansion of $G(x-\tilde{y})$ in $y \in \partial D$ or $\partial \Omega$ that for $x\in \mathbb{R}^2\setminus\overline{D}$ fixed,
\begin{align}
&\mathcal{S}_{D_\epsilon}[\psi_{i,\epsilon}](x)=\int_{\partial D_\epsilon}G(x,\tilde{y})\psi_{i,\epsilon}(\tilde{y})\d\sigma(\tilde{y}) \nonumber\\ \nonumber
 =&\int_{\partial D}\Big(G(x,y)+\epsilon f(y)\frac{\partial G(x,y)}{\partial\nu_D(y)}\Big)\Big(\psi_i(y)+\epsilon \psi_i^{(1)}(y)\Big)
 \Big(1-\epsilon\tau_D(y)f(y)\Big) \d\sigma(y) + O(\epsilon^2)\\
 =&\mathcal{S}_{D}[\psi_i](x)+\epsilon \Big(\mathcal{S}_{D}[\psi_i^{(1)}](x)-\mathcal{S}_{D}[\tau_{D}f\psi_i](x)+\mathcal{D}_{D}[f\psi_i](x)\Big)
+O(\epsilon^2),\label{SD-expansion}
\end{align}
and
\begin{align}
&\mathcal{S}_{\Omega_\epsilon}[\psi_{e,\epsilon}](x)=\int_{\partial \Omega_\epsilon}G(x,\tilde{y})\psi_{e,\epsilon}(\tilde{y})\d\sigma(\tilde{y})\nonumber\\
 =&\int_{\partial \Omega}\Big(G(x,y)+\epsilon g(y)\frac{\partial G(x,y)}{\partial\nu_\Omega(y)}\Big)\Big(\psi_e(y)+\epsilon \psi_e^{(1)}(y)\Big)
 \Big(1-\epsilon\tau_\Omega(y)g(y)\Big) \d\sigma(y) + O(\epsilon^2) \nonumber\\
 =&\mathcal{S}_{\Omega}[\psi_e](x)+\epsilon\Big(\mathcal{S}_{\Omega}[\psi_e^{(1)}](x)-\mathcal{S}_{\Omega}[\tau_{\Omega}g\psi_e](x)+\mathcal{D}_{\Omega}[g\psi_e](x)\Big)
+O(\epsilon^2).\label{SOmega-expansion}
\end{align}
The following pointwise expansions follow immediately from \eqref{sol-p}, (\ref{SD-expansion}) and (\ref{SOmega-expansion}):
\begin{align}
  p_\epsilon(x)=&
   p(x)+ \epsilon\Big(\mathcal{S}_{D}[\psi_i^{(1)}](x)-\mathcal{S}_{D}[\tau_{D}f\psi_i](x)+\mathcal{D}_{D}[f\psi_i](x)\nonumber \\ &+\mathcal{S}_{\Omega}[\psi_e^{(1)}](x)-\mathcal{S}_{\Omega}[\tau_{\Omega}g\psi_e](x)+\mathcal{D}_{\Omega}[g\psi_e](x)\Big)+O(\epsilon^2), \quad x\in \mathbb{R}^2\setminus\overline{D}.\label{p-epsilon}
\end{align}
We now prove the following representation theorem for the solution $p^{(1)}$ to the first-order coupled system (\ref{first-order-equation}), which will be very helpful in
the proof of Theorem \ref{thm-expansion}.
\begin{thm}\label{thm-p1}
The solution $p^{(1)}$ to (\ref{first-order-equation}) is represented by
\begin{align}\label{p1}
  p^{(1)}=
  \mathcal{S}_{D}[\psi_i^{(1)}-\tau_{D}f\psi_i](x) +\mathcal{D}_{D}[f\psi_i](x)
  + \mathcal{S}_{\Omega}[\psi_e^{(1)}- \tau_{\Omega}g\psi_e](x)+\mathcal{D}_{\Omega}[g\psi_e](x), \quad x\in \mathbb{R}^2\setminus\overline{D},
\end{align}
where $\psi_i$, $\psi_e$, $\psi_i^{(1)}$ and $\psi_e^{(1)}$ are solutions to (\ref{density-equation}) and (\ref{first-order-density-equation}), respectively.
\end{thm}
\begin{proof}
  One can easily see that
  \[
  \Delta p^{(1)}=0  \quad \mbox{in } \Omega\setminus\overline{D}
  \mbox{\quad and \quad}
   \Delta p^{(1)}=0  \quad \mbox{in } \mathbb{R}^2\setminus\overline{\Omega}.
  \]
 It is clear that  $p^{(1)}$ defined by (\ref{p1}) satisfies the transmission conditions (the conditions on the fifth, sixth and seventh lines in (\ref{first-order-equation})).
  Using (\ref{first-order-density-equation}), we have
 \begin{align*}
   \frac{\p p^{(1)}}{\p \nu_D}\Big|_{+}=& \frac{\mathcal{S}_{D}[\psi_i^{(1)}]}{\p \nu_D}\Big|_{+} -  \frac{\mathcal{S}_{D}[\tau_{D}f\psi_i]}{\p \nu _D}\Big|_{+} + \frac{\mathcal{D}_{D}[f\psi_i]}{\p \nu_D}
  +\frac{ \mathcal{S}_{\Omega}[\psi_e^{(1)}]}{\p \nu_D} - \frac{\mathcal{S}_{\Omega}[\tau_{\Omega}g\psi_e]}{\p \nu_D } + \frac{\p \mathcal{D}_{\Omega}[g\psi_e]}{\p \nu_D }\\
  =& -\mathcal{K}^{(1)}_{D}[\psi_i] -\mathcal{A}_{\Omega}[\psi_e]-f\frac{\partial^2 P}{\partial \nu^2_D}+f'\frac{\partial P}{\partial T_D} -  \frac{\mathcal{S}_{D}[\tau_{D}f\psi_i]}{\p \nu _D}\Big|_{+} + \frac{\mathcal{D}_{D}[f\psi_i]}{\p \nu_D} - \frac{\mathcal{S}_{\Omega}[\tau_{\Omega}g\psi_e]}{\p \nu_D } + \frac{\p \mathcal{D}_{\Omega}[g\psi_e]}{\p \nu_D }\\
  =&\frac{\partial}{\partial T_D}\bigg{(} f\frac{\partial \mathcal{S}_D[\psi_i]}{\partial T_D}\bigg{)}-\tau_D f\frac{\partial \mathcal{S}_{D}[\psi_i]}{\partial \nu_D}\Big|_{+}-f\frac{\partial^2 P}{\partial \nu^2_D}+f'\frac{\partial P}{\partial T_D}- f(x)\frac{\partial^2 \mathcal{S}_{\Omega}[\psi_e] }{\partial \nu^2_{D}} + f' \frac{\partial\mathcal{S}_{\Omega}[\psi_e]}{\partial T_{D}}\\
  =&f' \frac{\partial \mathcal{S}_D[\psi_i]}{\partial T_D} - f \frac{\partial^2 \mathcal{S}_D[\psi_i]}{\partial  \nu^2_D}-f\frac{\partial^2 P}{\partial \nu^2_D}+f'\frac{\partial P}{\partial T_D}- f\frac{\partial^2 \mathcal{S}_{\Omega}[\psi_e] }{\partial \nu^2_{D}} + f' \frac{\partial\mathcal{S}_{\Omega}[\psi_e]}{\partial T_{D}}\\
  =&f'\frac{\partial  p}{\partial T_{D}}-  f \frac{\partial^2  p}{\partial \nu^2_{D}}.
 \end{align*}
 It follows from \eqref{first-order-equation} that
 \begin{align*}
p^{(1)}|_{+}-p^{(1)}|_{-} =-g \psi_e = g\Big(\frac{\partial p}{\partial\nu_{\Omega}}\Big|_{-}-\frac{\partial p}{\partial\nu_{\Omega}}\Big|_{+}\Big).
 \end{align*}
According to \eqref{first-order-equation} and \eqref{first-order-density-equation}, we obtain
  \begin{align*}
\frac{\p p^{(1)}}{\p \nu_\Omega}\Big|_{+} - \frac{p^{(1)}}{\p \nu_\Omega}\Big|_{-} =& \psi_e^{(1)}-\tau_{\Omega}g\psi_e\\
=&12\zeta_0\Big( \frac{\partial \varphi^{(1)}}{\partial \nu_\Omega}+ g\frac{\partial^2 \varphi}{\partial \nu^2_\Omega}-g'\frac{\partial \varphi}{\partial T_\Omega}\Big)
 +\tau_{\Omega}g\Big(\frac{\partial p}{\partial\nu_{\Omega}}\Big|_{-}-\frac{\partial p}{\partial\nu_{\Omega}}\Big|_{+}\Big) \\
 =&12\zeta_0\Big( \frac{\partial \varphi^{(1)}}{\partial \nu_\Omega}+ g\frac{\partial^2 \varphi}{\partial \nu^2_\Omega}-g'\frac{\partial \varphi}{\partial T_\Omega}\Big)
 +g\Big(\frac{\partial^2  p}{\partial \nu^2_{\Omega}}\Big|_{-}- \frac{\partial^2  p}{\partial \nu^2_{\Omega}}\Big|_{+}\Big).
 \end{align*}

Because $\psi_i - \tau_D f \psi_i^{(1)}\in L^2_0(\partial D)$ and $\psi_e - \tau_\Omega g \psi_e^{(1)}\in L^2_0(\partial \Omega)$, we have $\mathcal{S}_\Omega[\psi_i - \tau_D f \psi_i^{(1)}]=O(|x|^{-1})$ and $\mathcal{S}_\Omega[\psi_e - \tau_\Omega g \psi_e^{(1)}]=O(|x|^{-1})$ as $|x|\rightarrow \infty$.
Therefore $p^{(1)}$ defined by (\ref{p1}) satisfies $p^{(1)}=O(|x|^{-1})$ as $|x|\rightarrow \infty$. This finishes the proof of the theorem.
\end{proof}
Theorem \ref{thm-expansion} immediately follows from (\ref{varphi-epsilon}), (\ref{p-epsilon}) and the integral representations of $\varphi^{(1)}$ and $p^{(1)}$ in (\ref{varphi1}) and (\ref{p1}).

From Theorem \ref{thm-p1}, we can rewrite Theorem \ref{thm-cloaking1} as the following theorem.
\begin{thm}
%\label{thm-cloaking2}
Let $\psi_i$, $\psi_e$, $\psi_i^{(1)}$ and $\psi_e^{(1)}$ be given in Theorem \ref{thm-p1}. Given the shape function $f\in \mathcal{C}^1(\partial D)$, if there is a shape function $g\in \mathcal{C}^1(\partial \Omega)$,
such that
\begin{align*}
 p^{(1)}=0,\quad \ \mbox{in} \  \mathbb{R}^2\setminus\overline{\Omega},
\end{align*}
then 2-order HNCD occurs.
\end{thm}

\begin{rem}\label{rem-representation-formula}
In fact, the solution $\varphi^{(1)}$ to \eqref{first-order-equation} can also be the following representation formula
\begin{equation*}
\varphi^{(1)}=
\mathcal{S}_{D}[\phi_1](x),\quad x\in\mathbb{R}^2\setminus\overline{D},
\end{equation*}
where
 \begin{align*}
\Big(\frac{1}{2} I+\mathcal{K}^*_{D}\Big)[\phi_1]=f'\frac{\partial  \varphi}{\partial T_D}-  f \frac{\partial^2  \varphi}{\partial \nu^2}\quad \mbox{on } \p D.
\end{align*}
And the solution $p^{(1)}$ to \eqref{first-order-equation} can also be the following representation formula
\begin{equation*}
p^{(1)}=
\mathcal{S}_{D}[\psi_1](x) + \mathcal{S}_{\Omega}[\psi_2](x) + \mathcal{D}_{\Omega}[\psi_3](x),\quad x\in\mathbb{R}^2\setminus\overline{D},
\end{equation*}
where
\begin{align*}
\begin{cases}
%\label{desity-equation-v2}
  \ds \Big(\frac{1}{2} I+\mathcal{K}^*_{D}\Big)[\psi_1] = -\frac{\p \mathcal{S}_{\Omega}[\psi_2] }{\p \nu_D} -\frac{\p \mathcal{D}_{\Omega}[\psi_3] }{\p \nu_D} + A &\quad \mbox{on } \p D, \\
  \ds \psi_2 = C  &\quad \mbox{on } \p \Omega, \\
  \ds \psi_3 = - B  &\quad \mbox{on } \p \Omega.
\end{cases}
\end{align*}
These representation formulas are very useful for dealing with the deformed confocal ellipses case in Subsection \ref{subsec-ellipse}.
\end{rem}

Using the above representation formulas, we can discuss a special case that the shape function $f$ and $g$ are constants.
\begin{rem}
If the shape functions $f$ and $g$ are constants, then $f'=0$ and $g'=0$.  We first have
\begin{align*}
  \varphi^{(1)} = -f \mathcal{S}_D\Big(\frac{1}{2} I+\mathcal{K}^*_{D}\Big)^{-1}\Big[\frac{\partial^2  \varphi}{\partial \nu_D^2}\Big|_{\p D}\Big].
\end{align*}
Let $\bar{\varphi}^{(1)} =- \mathcal{S}_D\Big(\frac{1}{2} I+\mathcal{K}^*_{D}\Big)^{-1}\Big[\frac{\partial^2  \varphi}{\partial \nu_D^2}\big|_{\p D}\Big]$, then $ \varphi^{(1)} =f \bar{\varphi}^{(1)}$.
We next have
\begin{align*}
\begin{cases}
\ds A=-  f \frac{\partial^2  p}{\partial \nu^2_{D}}\quad &\mbox{on } \partial D, \ms \\
\ds B=g\Big(\frac{\partial p}{\partial\nu_{\Omega}}\Big|_{-}-\frac{\partial p}{\partial\nu_{\Omega}}\Big|_{+}\Big)\quad &\mbox{on} \ \partial \Omega,\ms \\
\ds C= g\Big(\frac{\partial^2  p}{\partial \nu^2_{\Omega}}\Big|_{-}- \frac{\partial^2  p}{\partial \nu^2_{\Omega}}\Big|_{+}\Big) + 12\zeta_0 \Big( f\frac{\partial \bar{\varphi}^{(1)}}{\partial \nu_{\Omega}} +g \frac{\partial^2  \varphi}{\partial \nu^2_{\Omega}}\Big)\quad &\mbox{on } \partial \Omega.
\end{cases}
\end{align*}
For variable separation, we denote
\begin{align*}
\bar{\psi}_1=\frac{\partial^2  p}{\partial \nu^2_{D}}, \
\bar{\psi}_3=\frac{\partial p}{\partial\nu_{\Omega}}\Big|_{+}-\frac{\partial p}{\partial\nu_{\Omega}}\Big|_{-},\
\bar{\psi}_{2,1}= \frac{\partial^2  p}{\partial \nu^2_{\Omega}}\Big|_{-}- \frac{\partial^2  p}{\partial \nu^2_{\Omega}}\Big|_{+}, \
\bar{\psi}_{2,2,1}=12 \zeta_0\frac{\partial  \bar{\varphi}^{(1)}}{\partial \nu_{\Omega}},\
\bar{\psi}_{2,2,2}=12 \zeta_0 \frac{\partial^2  \varphi}{\partial \nu^2_{\Omega}}.
\end{align*}
 Then $A=-f \bar{\psi}_1$, $\psi_3= g \bar{\psi}_3$ and $\psi_2= f \bar{\psi}_{2,2,1}+g(\bar{\psi}_{2,1}+\bar{\psi}_{2,2,2})$, it is straightforward to see that
 \begin{align*}
   p^{(1)}=&g\Big(\mathcal{S}_{\Omega}[\bar{\psi}_{2,1}+\bar{\psi}_{2,2,2}]+\mathcal{D}_{\Omega}[\bar{\psi}_3]-  \mathcal{S}_D \Big(\frac{1}{2} I+\mathcal{K}^*_{D}\Big)^{-1}\Big[\frac{\p \mathcal{S}_{\Omega}[\bar{\psi}_{2,1}+\bar{\psi}_{2,2,2}] }{\p \nu_D}+\frac{\p \mathcal{D}_{\Omega}[\bar{\psi}_3] }{\p \nu_D}\Big]\Big)\\
   &+ f \Big(\mathcal{S}_{\Omega}[\bar{\psi}_{2,2,1}]- \mathcal{S}_D \Big(\frac{1}{2} I+\mathcal{K}^*_{D}\Big)^{-1}\Big[\frac{\p \mathcal{S}_{\Omega}[\bar{\psi}_{2,2,1}]}{\p \nu_D}+\bar{\psi}_1\Big]\Big),\quad x\in \R^2\setminus\overline{D}.
 \end{align*}
 Thus if the following equation holds, then $p^{(1)}=0$ in $\R^2\setminus\overline{\Omega}$.
 \begin{align}\label{eq-f-g-constant}
   &g\Big(\mathcal{S}_{\Omega}[\bar{\psi}_{2,1}+\bar{\psi}_{2,2,2}]+\mathcal{D}_{\Omega}[\bar{\psi}_3]-  \mathcal{S}_D \Big(\frac{1}{2} I+\mathcal{K}^*_{D}\Big)^{-1}\Big[\frac{\p \mathcal{S}_{\Omega}[\bar{\psi}_{2,1}+\bar{\psi}_{2,2,2}] }{\p \nu_D}+\frac{\p \mathcal{D}_{\Omega}[\bar{\psi}_3] }{\p \nu_D}\Big]\Big)\nonumber\\
   &+ f \Big(\mathcal{S}_{\Omega}[\bar{\psi}_{2,2,1}]- \mathcal{S}_D \Big(\frac{1}{2} I+\mathcal{K}^*_{D}\Big)^{-1}\Big[\frac{\p \mathcal{S}_{\Omega}[\bar{\psi}_{2,2,1}]}{\p \nu_D}+\bar{\psi}_1\Big]\Big)=0.
 \end{align}
 \end{rem}
The remark would like to illustrate that one can consider solving equation \eqref{eq-f-g-constant} to find a special shape function $g$ when the boundaries $\p D$ and $\p \Omega$ are complex geometries such that the equation \eqref{first-order-equation} can not be solved analytically.

\section{Enhanced hydrodynamic near-cloaking for special perturbed Geometry}\label{sec-hnc}
This section is devoted to the proofs of Theorems \ref{thm-near-cloaking-circle} and \ref{thm-near-cloaking-ellipse}, which determine the conditions for enhanced hydrodynamic near-cloaking.
So far, in the literature \cite{Liu2023}, we know that perfect cloaking occurs on annulus and confocal ellipses.
Hence, in this section, we specifically focus on an object with the shape of a slightly deformed annulus or confocal ellipses cylinder. We show that, compared with the general geometry, the cloaking conditions and the relationship of shapes are quantified more precisely. Before dealing with these cases, we first recall some knowledge about the perfect cloaking in the following lemmas.

\begin{lem}
Let the domains $D$ and $\Omega$ be concentric disks of radii $r_i$ and $r_e$, where $r_e>r_i$. Let $H(x) = r^n\e^{\i n \theta}$ and $P(x) = 12 r^n\e^{\i n \theta}$ for $n\geq 1$. If
\begin{equation}\label{annulus-cloaking-zeta}
    \zeta_0=\frac{2r_i^{2n}r_e^{2n}}{r_e^{4n}-r_i^{4n}},
\end{equation}
then the perfect hydrodynamic cloaking occurs.
\end{lem}
Simplify \eqref{annulus-cloaking-zeta}, we can obtain
\begin{equation}\label{annulus-cloaking-zeta-v2}
    \zeta_0=\frac{2r_*^{2n}}{r_*^{4n}-1},
\end{equation}
where $r_*=r_e/r_i$. From \eqref{annulus-cloaking-zeta-v2}, it follows that the zeta potential depends only on the rate of inner and outer radii.

After introducing the elliptic coordinates $(\xi, \eta)$ so that $x=(x_1,x_2)$ in Cartesian coordinates are defined by
\begin{align}\label{elliptic-coordinates}
  x_1=l \cosh \xi \cdot \cos \eta, \quad x_2=l \sinh \xi \cdot \sin \eta,\quad \xi \geq 0, \quad 0\leq \eta \leq 2\pi,
\end{align}
where $2l$ is the focal distance, we have the following lemma.
\begin{lem}
Let the boundaries of the domains $D$ and $\Omega$ be confocal ellipses of elliptic radii $\xi_i$ and $\xi_e$, where $\xi_e>\xi_i$.
\begin{itemize}
    \item Let $H(x) = \cosh (n\xi) \cos (n\eta)$ and $P(x) = 12  \cosh (n\xi) \cos (n\eta)$ for $n\geq 1$. If
    \begin{equation}\label{ellipse-cloaking-zeta-x}
    \zeta_0=\frac{ \sinh (n\xi_i)}{\big(\sinh (n\xi_e) - \e^{n(\xi_i-\xi_e)}\sinh (n\xi_i)\big)\cosh (n(\xi_e -\xi_i))},
    \end{equation}
    then the perfect hydrodynamic cloaking occurs.
    \item Let $H(x) = \sinh (n\xi) \sin (n\eta)$ and $P(x) = 12  \sinh (n\xi) \sin (n\eta)$ for $n\geq 1$. If
    \begin{equation}\label{ellipse-cloaking-zeta-y}
    \zeta_0=\frac{ \cosh (n\xi_i)}{\big(\cosh (n\xi_e) - \e^{n(\xi_i-\xi_e)}\cosh (n\xi_i)\big)\cosh (n(\xi_e -\xi_i))},
    \end{equation}
    then the perfect hydrodynamic cloaking occurs.
  \end{itemize}
\end{lem}
The alternate form for \eqref{ellipse-cloaking-zeta-x} and \eqref{ellipse-cloaking-zeta-y} is
\begin{align}\label{ellipse-cloaking-zeta}
  \zeta_0 = \frac{2 \e^{2n\xi_i}\e^{2n\xi_e}}{\e^{4n\xi_e}-\e^{4n\xi_i}}(1\mp\e^{-2n\xi_i})
\end{align}
by using the definition of hyperbolic function, where $"\mp"$ corresponds to the \eqref{ellipse-cloaking-zeta-x} and \eqref{ellipse-cloaking-zeta-y}, respectively.

Here, we  briefly discuss the relationship between the annulus case and the confocal ellipses case. From \eqref{elliptic-coordinates}, it follows that the confocal ellipse is a circle of radius $\frac{1}{2}l\e^{\xi}$ as $\xi\rightarrow \infty$. Consider the boundaries $\p D$ and $\p \Omega$ are confocal ellipses of elliptic radii $\xi_i$ and $\xi_e$ with $\xi_e>\xi_i$ as $\xi_i, \xi_e \rightarrow \infty$. Then the boundaries $\p D$ and $\p \Omega$ are circles of radii $ \frac{1}{2}l\e^{\xi_i}$ and $\frac{1}{2}l\e^{\xi_e}$, respectively.  Let $r_i= \frac{1}{2}l\e^{\xi_i}$ and $r_e= \frac{1}{2}l\e^{\xi_e}$, we find that the equation \eqref{ellipse-cloaking-zeta} is equivalent to \eqref{annulus-cloaking-zeta}. Therefore, the cloaking on annulus is a special case of the cloaking on confocal ellipses in limit.

\subsection{Enhanced near-cloaking on the deformed annulus}\label{subsec-annulus}
In this subsection, we consider the enhanced microscale hydrodynamic near-cloaking when the domains $D$ and
$\Omega$ are concentric disks. We construct the shape function $g$ and derive the enhanced hydrodynamic near-cloaking conditions by calculating the explicit form of the solution. Throughout this subsection, we set
$D :=\{|x| <r_i\}$ and $\Omega :=\{|x| <r_e\}$, where $r_e > r_i$. Then, from \eqref{pertur-D} and \eqref{pertur-Omega}, $\p D_\epsilon$ and $\p \Omega_\epsilon$ in polar coordinates are written as
\begin{align*}
\begin{cases}
  \tilde{x}_{i,1}=(r_i+\epsilon f(\theta))\cos(\theta),\\
   \tilde{x}_{i,2}=(r_i+\epsilon f(\theta))\sin(\theta),
\end{cases}
\quad\mbox{and }\quad
\begin{cases}
  \tilde{x}_{e,1}=(r_e+\epsilon g(\theta))\cos(\theta),\\
   \tilde{x}_{e,2}=(r_e+\epsilon g(\theta))\sin(\theta),
\end{cases}
\end{align*}
where $\tilde{x}_i=(\tilde{x}_{i,1}, \tilde{x}_{i,2})\in \p D_\epsilon$ and  $\tilde{x}_e= (\tilde{x}_{e,1}, \tilde{x}_{e,2})\in \p \Omega_\epsilon$.

For each integer $m\neq 0$ and $a=i, e$, one can easily see that (cf. \cite{Ammari2013-1})
\begin{equation}\label{S-r-ra}
\mathcal{S}_{\Gamma}[\e^{\i m\theta}](x) = \begin{cases}
-\frac{r_a}{2m}\Big( \frac{r}{r_a}\Big)^m \e^{\i m\theta},\quad & |x|=r  < r_a,\\
-\frac{r_a}{2m}\Big( \frac{r_a}{r}\Big)^m \e^{\i m\theta},\quad & |x|=r > r_a,
\end{cases}
\end{equation}
and
\begin{equation}\label{K-r-ra}
\mathcal{K}_{\Gamma}^*[\e^{\i m\theta}](x)=0, \quad \forall m\neq 0.
\end{equation}

We begin with the proof of Theorem \ref{thm-near-cloaking-circle}, where the shape function $g$ is constructed by recursive formulas.
\renewcommand{\proofname}{\indent Proof of Theorem \ref{thm-near-cloaking-circle}}
\begin{proof}
Let $H(x) = r^n\cos (n\theta)$ and $P(x)=12 r^n\cos(n\theta)$ for $n\geq 1$. From Section 4 in \cite{Liu2023}, we have already known
\begin{align}\label{density-annulus-ep}
\phi = -2n r_i^{n-1}\cos(n\theta),
\end{align}
and
\begin{align}\label{disk-varphi}
\varphi = \Big(r^n+\frac{r_i^{2n}}{r^n}\Big)\cos(n\theta).
\end{align}
To determine $\phi^{(1)}$, we first write $f$ as a Fourier series expansion
\begin{align}\label{f-series}
f(\theta)=\frac{a_0}{2}+\sum_{m=1}^\infty a_m\cos(m\theta) + b_m\sin(m\theta),
\end{align}
where the coefficients $a_m$  and  $b_m$  are defined as
\begin{align*}
 a_0 = \frac{1}{\pi}\int_{-\pi}^{\pi} f(\theta) \d \theta, \quad
 a_m = \frac{1}{\pi}\int_{-\pi}^{\pi} f(\theta)  \cos(m\theta) \d \theta,\quad  b_m = \frac{1}{\pi}\int_{-\pi}^{\pi} f(\theta)  \sin(m\theta) \d \theta,  \quad m= 1,  2,\ldots\ .
\end{align*}
Substituting \eqref{density-annulus-ep} and (\ref{f-series}) into \eqref{asy density1} and using \eqref{K-r-ra}, leads to
\begin{align*}
\phi^{(1)}=&- 2 n r_i^{n-2} \sum_{m = n}^\infty m \big((a_{m-n}-a_{m+n}) \cos (m \theta) + (b_{m-n}-b_{m+n})\sin(m\theta)\big).
%&+ n r_i^{n-1} \sum_{n=1}^{\infty}m\big((a_{m-n}+a_{m+n}) \cos (m \theta) + (b_{m-n}+b_{m+n})\sin(m\theta)\big)
\end{align*}
Therefore by \eqref{varphi1} and \eqref{S-r-ra}, the first-order solution to the electrostatic potential is given by
\begin{align}\label{annulus-varphi1}
\varphi^{(1)}= n \sum_{m = n}^\infty r_i^{m+n-1} r^{-m} \big((a_{m-n}-a_{m+n}) \cos (m \theta) + (b_{m-n}-b_{m+n})\sin(m\theta)\big).
\end{align}

Before proceeding, we write $g$ as a Fourier series expansion
\begin{align}\label{g-series}
g(\theta)=\frac{d_0}{2}+\sum_{m=0}^\infty d_m\cos(m\theta) + h_m\sin(m\theta),
\end{align}
where the coefficients $d_m$  and  $h_m$ are defined as
\begin{align*}
 d_0 = \frac{1}{\pi}\int_{-\pi}^{\pi} g(\theta) \d \theta, \quad  d_m = \frac{1}{\pi}\int_{-\pi}^{\pi} g(\theta)  \cos(m\theta) \d \theta,\quad  h_m = \frac{1}{\pi}\int_{-\pi}^{\pi} g(\theta)  \sin(m\theta) \d \theta,  \quad m= 1,  2,\ldots\ .
\end{align*}
From Section 4 in \cite{Liu2023}, we also have already known
\begin{align}
\label{density-annulus}
\begin{cases}
\ds \psi_i = 12n\frac{r_i^{n-1}}{r_e^{2n}}\Big((r_e^{2n}-r_i^{2n})\zeta_0-2r_e^{2n}\Big)\cos(n\theta), \ms \\
\ds \psi_e = 12 n \zeta_0 \Big(r_e^{n-1}-\frac{r_i^{2n}}{r_e^{n+1}}\Big)\cos(n\theta).
\end{cases}
\end{align}
Substituting (\ref{disk-varphi})--(\ref{density-annulus}) into \eqref{first-order-density-equation} and using \eqref{K-r-ra}, by solving the equation \eqref{first-order-density-equation} we obtain
%{\small
\begin{align*}
 \psi_i^{(1)}= & 6n\zeta_0\Big(- 2\sum_{m = n}^\infty m r_i^{2m+n-2} r_e^{-2m} \big((a_{m-n}-a_{m+n}) \cos (m \theta) + (b_{m-n}-b_{m+n})\sin(m\theta)\big)\\
&+  \Big(r_e^{n-2}+ \frac{r_i^{2n}}{r_e^{n+2}}\Big)\sum_{m=n}^{\infty} m\Big(\frac{r_i}{r_e}\Big)^{m-1}\big((d_{m-n}-d_{m+n})\cos(m\theta)+(h_{m-n}-h_{m+n})\sin(m\theta)\big)\\
&-  \Big(r_e^{n-2}-\frac{r_i^{2n}}{r_e^{n+2}}\Big) \sum_{m=n}^{\infty} m \Big(\frac{r_i}{r_e}\Big)^{m-1} \big((d_{m-n}+d_{m+n})\cos(m\theta)+(h_{m-n}+h_{m+n})\sin(m\theta)\big)\Big) \\
&-12n\frac{r_i^{n-2}}{r_e^{2n}}\Big((r_e^{2n}-r_i^{2n})\zeta_0-2r_e^{2n}\Big) \sum_{m=n}^{\infty} m\big((a_{m-n}+a_{m+n})\cos(m\theta)+(b_{m-n}+b_{m+n})\sin(m\theta)\big)\\
&+\frac{12n r_i^{n-2}}{r_e^{2n}}\Big((r_e^{2n}-r_i^{2n})\zeta_0-2r_e^{2n}\Big) \sum_{m = n}^\infty m \big((a_{m-n}-a_{m+n}) \cos (m \theta) + (b_{m-n}-b_{m+n})\sin(m\theta)\big),
\end{align*}
and
\begin{align*}
\psi_e^{(1)} =&  6n\zeta_0\Big(-2 \sum_{m = n}^\infty m r_i^{m+n-1} r_e^{-m-1} \big((a_{m-n}-a_{m+n}) \cos (m \theta) + (b_{m-n}-b_{m+n})\sin(m\theta)\big)\\
&+  \Big(r_e^{n-2}+ \frac{r_i^{2n}}{r_e^{n+2}}\Big)\sum_{m=n}^{\infty} m\big((d_{m-n}-d_{m+n})\cos(m\theta)+(h_{m-n}-h_{m+n})\sin(m\theta)\big)\\
& -\Big(r_e^{n-2}-\frac{r_i^{2n}}{r_e^{n+2}}\Big)\sum_{m=n}^{\infty} (d_{m-n}+d_{m+n})\cos(m\theta)+(h_{m-n}+h_{m+n})\sin(m\theta)\Big).
\end{align*}
Using \eqref{p1} and \eqref{S-r-ra} we have
\begin{align*}
p^{(1)} = \sum_{m=n}^{\infty} r^{-m}\big(M_{m,n}^{1}\cos(m\theta)+ M_{m,n}^{2}\sin(m\theta)\big),\quad r > r_e,
\end{align*}
where
\begin{align*}
  M_{m,n}^{1}=  \frac{6n \zeta_0}{r_e^{2m}}\Big[\Big(r_i^{3m+n-1} + r_e^{2(m-n)}\Big(r_i^{3n+m-1}+\frac{2r_i^{n+m-1} r_e^{2n}}{\zeta_0}\Big)\Big)(a_{m-n} - a_{m+n})\\
  -r_e^{m-n-1}\big((r_i^{2(m+n)}+r_e^{2(m+n)})d_{m-n}-(r_i^{2m}r_e^{2n}+r_i^{2n}r_e^{2m})d_{m+n}\big)\Big],
\end{align*}
\begin{align*}
  M_{m,n}^{2}=  \frac{6n \zeta_0}{r_e^{2m}}\Big[\Big(r_i^{3m+n-1} + r_e^{2(m-n)}\Big(r_i^{3n+m-1}+\frac{2r_i^{n+m-1} r_e^{2n}}{\zeta_0}\Big)\Big)(b_{m-n} - b_{m+n})\\
  -r_e^{m-n-1}\big((r_i^{2(m+n)}+r_e^{2(m+n)})h_{m-n}-(r_i^{2m}r_e^{2n}+r_i^{2n}r_e^{2m})h_{m+n}\big)\Big].
\end{align*}
Here $\zeta_0=$ satisfies \eqref{annulus-cloaking-zeta}, and $ M_{m,n}^{1}$, $ M_{m,n}^{2}$ are called first-order scattering coefficients.

If one needs $p^{(1)}=0 $ for $ r>r_e$, then the following conditions should be satisfied:
{\small
\begin{align}\label{recursive equations-d-cos}
r_i^{m-n-1}&\big(r_i^{2(m+n)}+r_e^{2(m+n)}\big)(a_{m-n}-a_{m+n}) - r_e^{m-n-1}\big(r_i^{2(m+n)}+r_e^{2(m+n)}\big)d_{m-n}\nonumber\\
&+r_e^{m+n-1}\big(r_i^{2m}+r_i^{2n}r_e^{2(m-n)}\big)d_{m+n}=0,
\end{align}}
and
{\small
\begin{align}\label{recursive equations-h-cos}
r_i^{m-n-1}&\big(r_i^{2(m+n)}+r_e^{2(m+n)}\big)(b_{m-n}-b_{m+n})-r_e^{m-n-1}\big(r_i^{2(m+n)}+r_e^{2(m+n)}\big)h_{m-n}\nonumber\\
&+r_e^{m+n-1}\big(r_i^{2m}+r_i^{2n}r_e^{2(m-n)}\big)h_{m+n}=0.
\end{align}}
The recursive equations (\ref{recursive equations-d-cos}) and (\ref{recursive equations-h-cos}) define the shape of the cloaking region relating it to the shape of the object.
Rearranging and changing the subscripts according to $m \rightarrow m + n$ yields
\begin{align}\label{recursive equations-d-1-cos}
 d_m=\frac{r_i^{2(m+n)}r_e^{2n}+r_i^{2n}r_e^{2(m+n)}}{r_i^{2(m+2n)}+r_e^{2(m+2n)}}d_{m+2n}+\Big(\frac{r_i}{r_e}\Big)^{m-1}(a_{m}-a_{m+2n}), \\ h_m=\frac{r_i^{2(m+n)}r_e^{2n}+r_i^{2n}r_e^{2(m+n)}}{r_i^{2(m+2n)}+r_e^{2(m+2n)}}h_{m+2n}+\Big(\frac{r_i}{r_e}\Big)^{m-1}(b_{m}-b_{m+2n}),\label{recursive equations-h-1-cos}
\end{align}
or
\begin{align}\label{recursive equations-d-2-cos}
 d_{m+2n}=\frac{r_i^{2(m+2n)}+r_e^{2(m+2n)}}{r_i^{2(m+n)}r_e^{2n}+r_i^{2n}r_e^{2(m+n)}}\Big[d_{m}-\Big(\frac{r_i}{r_e}\Big)^{m-1}(a_{m}-a_{m+2n})\Big], \\ h_{m+2n}=\frac{r_i^{2(m+2n)}+r_e^{2(m+2n)}}{r_i^{2(m+n)}r_e^{2n}+r_i^{2n}r_e^{2(m+n)}}\Big[h_{m}-\Big(\frac{r_i}{r_e}\Big)^{m-1}(b_{m}-b_{m+2n})\Big].\label{recursive equations-h-2-cos}
\end{align}

Let $H(x) = r^n\sin (n\theta)$ and $P(x)=12 r^n\sin(n\theta)$ for $n\geq 1$. In a similar way, we have
\begin{align*}
%\label{annulus-varphi1-sin}
\varphi^{(1)}= n \sum_{m = n}^\infty r_i^{m+n-1} r^{-m} \big((a_{m-n}+a_{m+n}) \sin (m \theta) - (b_{m-n}+b_{m+n})\cos(m\theta)\big),
\end{align*}
and
\begin{align*}
p^{(1)} = \sum_{m=n}^{\infty} r^{-m}\big(M_{m,n}^{1}\sin(m\theta)+ M_{m,n}^{2}\cos(m\theta)\big),
\end{align*}
where
\begin{align*}
  M_{m,n}^{1}=  \frac{6n \zeta_0}{r_e^{2m}}\Big[\Big(r_i^{3m+n-1} + r_e^{2(m-n)}\Big(r_i^{3n+m-1}+\frac{2r_i^{n+m-1} r_e^{2n}}{\zeta_0}\Big)\Big)(a_{m-n} + a_{m+n})\\
  -r_e^{m-n-1}\big((r_i^{2(m+n)}+r_e^{2(m+n)})d_{m-n}+(r_i^{2m}r_e^{2n}+r_i^{2n}r_e^{2m})d_{m+n}\big)\Big],
\end{align*}
\begin{align*}
  M_{m,n}^{2}=  -\frac{6n \zeta_0}{r_e^{2m}}\Big[\Big(r_i^{3m+n-1} + r_e^{2(m-n)}\Big(r_i^{3n+m-1}+\frac{2r_i^{n+m-1} r_e^{2n}}{\zeta_0}\Big)\Big)(d_{m-n} + d_{m+n})\\
  -r_e^{m-n-1}\big((r_i^{2(m+n)}+r_e^{2(m+n)})h_{m-n}+(r_i^{2m}r_e^{2n}+r_i^{2n}r_e^{2m})h_{m+n}\big)\Big].
\end{align*}
Here $\zeta_0=$ satisfies \eqref{annulus-cloaking-zeta}, and $ M_{m,n}^{1}$, $ M_{m,n}^{2}$ are called first-order scattering coefficients.

If one needs $p^{(1)}=0$, then the following conditions should be satisfied:
{\small
\begin{align}\label{recursive equations-d-sin}
r_i^{m-n-1}&\big(r_i^{2(m+n)}+r_e^{2(m+n)}\big)(a_{m-n}+a_{m+n})-r_e^{m-n-1}\big(r_i^{2(m+n)}+r_e^{2(m+n)}\big)d_{m-n}\nonumber\\
&-r_e^{m+n-1}\big(r_i^{2m}+r_i^{2n}r_e^{2(m-n)}\big)d_{m+n}=0,
\end{align}}
and
{\small
\begin{align}\label{recursive equations-h-sin}
r_i^{m-n-1}&\big(r_i^{2(m+n)}+r_e^{2(m+n)}\big)(b_{m-n}+b_{m+n})-r_e^{m-n-1}\big(r_i^{2(m+n)}+r_e^{2(m+n)}\big)h_{m-n}\nonumber\\
&-r_e^{m+n-1}\big(r_i^{2m}+r_i^{2n}r_e^{2(m-n)}\big)h_{m+n}=0.
\end{align}}
The recursive equations (\ref{recursive equations-d-sin}) and (\ref{recursive equations-h-sin}) define the shape of the cloaking region relating it to the shape of the object.
Rearranging and changing the subscripts according to $m \rightarrow m + n$ yields
\begin{align}\label{recursive equations-d-1-sin}
 d_m=-\frac{r_i^{2(m+n)}r_e^{2n}+r_i^{2n}r_e^{2(m+n)}}{r_i^{2(m+2n)}+r_e^{2(m+2n)}}d_{m+2n}+\Big(\frac{r_i}{r_e}\Big)^{m-1}(a_{m}+a_{m+2n}), \\ h_m=-\frac{r_i^{2(m+n)}r_e^{2n}+r_i^{2n}r_e^{2(m+n)}}{r_i^{2(m+2n)}+r_e^{2(m+2n)}}h_{m+2n}+\Big(\frac{r_i}{r_e}\Big)^{m-1}(b_{m}+b_{m+2n}), \label{recursive equations-h-1-sin}
\end{align}
or
\begin{align}\label{recursive equations-d-2-sin}
 d_{m+2n}=-\frac{r_i^{2(m+2n)}+r_e^{2(m+2n)}}{r_i^{2(m+n)}r_e^{2n}+r_i^{2n}r_e^{2(m+n)}}\Big[d_{m}-\Big(\frac{r_i}{r_e}\Big)^{m-1}(a_{m}+a_{m+2n})\Big], \\ h_{m+2n}=-\frac{r_i^{2(m+2n)}+r_e^{2(m+2n)}}{r_i^{2(m+n)}r_e^{2n}+r_i^{2n}r_e^{2(m+n)}}\Big[h_{m}-\Big(\frac{r_i}{r_e}\Big)^{m-1}(b_{m}+b_{m+2n})\Big]. \label{recursive equations-h-2-sin}
\end{align}

Because of the recursive nature of (\ref{recursive equations-d-1-cos})--(\ref{recursive equations-h-2-cos}) (or (\ref{recursive equations-d-1-sin})--(\ref{recursive equations-h-2-sin}) ), there is freedom in the choice of $4n$ coefficients $d_{m+1}\cdots d_{m+2n}$,
and $h_{m+1} \cdots h_{m+2n}$, which affects the values and the number of non-vanishing coefficients. To obtain a finite number of
non-vanishing coefficients $d_m$ and $h_m$, we first define $m_{max}$ as the maximum subscript of the non-vanishing coefficients
$a_m$ and $b_m$. Thus, $a_m = b_m \equiv 0$ for $m > m_{max}$. Next, we set
\begin{align}\label{d-h-nmax}
  d_{m_{max}+1}=\cdots=  d_{m_{max}+2n}=  h_{m_{max}+1}= \cdots= h_{m_{max}+2n}=0.
\end{align}
From (\ref{recursive equations-d-cos})--(\ref{d-h-nmax}) and the definition of $n_{max}$, it follows that the coefficients with larger values of the subscript also
vanish,
\begin{align*}
  d_m=0, \  h_m=0,\quad  \mbox{for } m> m_{max}+2n.
\end{align*}
Hence a shape function $g$ can be constructed by $d_m =\{d_1,d_2,\ldots,d_{m_{max}} \}$ and  $h_m =\{h_1,h_2,\ldots,h_{m_{max}} \}$, where $d_m$ and $h_m$ are
determined from (\ref{recursive equations-d-1-cos}), (\ref{recursive equations-h-1-cos}) (or (\ref{recursive equations-d-1-sin}), (\ref{recursive equations-h-1-sin})) and (\ref{d-h-nmax}).

The proof is complete.
\end{proof}

\begin{rem}
Throughout this paper, we call these recursive formulas in proof the enhanced hydrodynamic near-cloaking conditions, which determines the existence of $g$.
\end{rem}

\begin{rem}\label{rem-special-circle}
From the recursive formulas, it follows that if $f= \frac{a_0}{2}$, then $g =\frac{r_e}{r_i} \frac{a_0}{2}$, and if $f= a_1 \cos(\theta)$ (or $a_1 \sin(\theta)$), then $ g =a_1 \cos(\theta)$ (or $a_1 \sin(\theta)$).
\end{rem}
Remark \ref{rem-special-circle} shows the shape function of the outer boundary is held constant when the shape function of the inner boundary is constant. Moreover, the factor is the ratio of the inner and outer radii, which does not depend on the background field. In fact, the structure is a perfect cloaking since the zeta potential satisfying perfect cloaking is a function of the ratio of the inner and outer radii, which corresponds to \eqref{annulus-cloaking-zeta-v2}. We also note that the shape functions of the inner and outer boundaries are the same when $f$ and $g$ are linear. Remark \ref{rem-special-circle} is verified numerically in Section \ref{sec-num-sim}.

\subsection{Enhanced near-cloaking on the deformed confocal ellipses}\label{subsec-ellipse}
In this subsection, we consider the enhanced microscale hydrodynamic near-cloaking when the boundaries of domains $D$
and $\Omega$ are confocal ellipses. We construct the shape function $g$ and derive the enhanced hydrodynamic near-cloaking conditions by calculating the explicit form of the solution. Throughout this subsection, we set
$$
\p D := \{ (\xi, \eta) : \xi = \xi_i\} \quad \mbox{and} \quad \p \Omega := \{ (\xi, \eta) : \xi = \xi_e\},
$$
where the numbers $\xi_i$ and $ \xi_e$ are called the elliptic radius of $\p D$ and $\p \Omega$, respectively.

Let $\p \Gamma = \{ (\xi, \eta) : \xi = \xi_a\}$ for $a=i, e$.
 The outward normal vector on $\p \Gamma$  is
\begin{align}\label{normal-vector}
  \nu_a =\gamma_a^{-1}(l\sinh \xi_a\cos \eta,l\cosh\xi_a\sin\eta),
\end{align}
for $a=i, e$, where
\begin{equation*}
\gamma_a =  \gamma (\xi_a, \eta) = l \sqrt{\sinh^2\xi_a+\sin^2\eta}.
\end{equation*}
By \eqref{normal-vector}, one can see easily that the length element $\d s$, the outward normal derivative $\frac{\p}{\p \nu}$ and tangent derivative $\frac{\p}{\p T}$ on $\p \Gamma$ are given in terms of the elliptic coordinates by
\begin{equation}\label{normal-tangent-derivatie}
  \d s = \gamma_a \d \eta, \quad \frac{\p}{\p \nu} = \gamma_a^{-1}\frac{\p}{\p \xi} \quad \mbox{and} \quad \frac{\p}{\p T} = \gamma_a^{-1}\frac{\p}{\p \eta}.
\end{equation}
To proceed, it is convenient to use the following notation: for $a=i, e$ and $m=1,2,\dots$,
\begin{align*}
  \beta_m^{c,a} := \gamma (\xi_a, \eta)^{-1} \cos (m\eta) \quad \mbox{and} \quad \beta_m^{s,a} := \gamma (\xi_a, \eta)^{-1} \sin (m\eta).
\end{align*}
From \eqref{pertur-D},\eqref{pertur-Omega} and \eqref{normal-vector}, $\p D_\epsilon$ and $\p \Omega_\epsilon$ in elliptic coordinates can be written as
\begin{align*}
\begin{cases}
  \tilde{x}_{i,1}=l (\cosh \xi_i  +\epsilon f \gamma_i^{-1}\sinh \xi_i)\cos\eta, \\
  \tilde{x}_{i,2}=l (\sinh \xi_i  + \epsilon f \gamma_i^{-1}\cosh\xi_i)\sin\eta,\quad
\end{cases}
\end{align*}
and
\begin{align*}
\begin{cases}
  \tilde{x}_{e,1}=l (\cosh \xi_e  +\epsilon g \gamma_e^{-1}\sinh\xi_e)\cos\eta, \\
  \tilde{x}_{e,2}=l (\sinh \xi_e  + \epsilon g \gamma_e^{-1}\cosh\xi_e)\sin\eta,\quad
\end{cases}
\end{align*}
where $\tilde{x}_i=(\tilde{x}_{i,1}, \tilde{x}_{i,2})\in \p D_\epsilon$ and  $\tilde{x}_e= (\tilde{x}_{e,1}, \tilde{x}_{e,2})\in \p \Omega_\epsilon$.

Before proceeding, we have a brief discussion. After many attempts, we find that it is difficult to find recursive equations by the explicit form of the solution to \eqref{first-order-equation}, which is different from the annulus case. To find recursive equations like that of annulus, we need to decompose the equation \eqref{first-order-equation} into two equations, in which the first equation is dominant. We can find recursive equations similar to the annulus case by the explicit form of the solution to this leading equation. Hence, from the principle of superposition we have the following decomposition:
\begin{align*}
\begin{cases}
  \varphi^{(1)} =    \varphi^{(1,0)}+  \varphi^{(1,1)},\\
  p^{(1)} =    p^{(1,0)}+  p^{(1,1)},\\
  E=E_0+E_1,\\
  A=A_0+A_1,\\
   B =B_0+B_1,\\
  C=C_0+C_1,
\end{cases}
\end{align*}
where $ \varphi^{(1,0)}$ and $  p^{(1,0)}$ satisfy
\begin{align}\label{first-order-equation-leader}
\begin{cases}
\ds \Delta \varphi ^{(1,0)}= 0  \quad \ &\mbox{in} \  \mathbb{R}^2\setminus\overline{D},\ms \\
\ds \frac{\partial \varphi^{(1,0)}}{\partial \nu_{D}} =E_0 \quad &\mbox{on } \partial D,\ms \\
\ds  \varphi^{(1,0)} = O(|x|^{-1})\ \ &as\ |x|\rightarrow +\infty,\ms \\
\ds \Delta p^{(1,0)}= 0 \quad \ &\mbox{in} \  \mathbb{R}^2\setminus\overline{D},\ms \\
\ds \frac{\partial p^{(1,0)}}{\partial \nu_{D}} = A_0 \quad &\mbox{on } \partial D,\ms \\
\ds p^{(1,0)}|_{+}-p^{(1,0)}|_{-}=B_0 \quad &\mbox{on} \ \partial \Omega,\ms \\
\ds \frac{\partial p^{(1,0)}}{\partial \nu_{\Omega}} \big|_{+} - \frac{\partial p^{(1,0)}}{\partial \nu_{\Omega}} \big|_{-} = C_0 \quad &\mbox{on } \partial \Omega, \ms \\
\ds  p^{(1,0)} = O(|x|^{-1})\ \ &as\ |x|\rightarrow +\infty,
\end{cases}
\end{align}
and  $ \varphi^{(1,1)}$ and $  p^{(1,1)}$ satisfy
\begin{align*}
%\label{first-order-equation-remainder}
\begin{cases}
\ds \Delta \varphi ^{(1,1)}= 0  \quad \ &\mbox{in} \  \mathbb{R}^2\setminus\overline{D},\ms \\
\ds \frac{\partial \varphi^{(1,1)}}{\partial \nu_{D}} =E_1 \quad &\mbox{on } \partial D,\ms \\
\ds  \varphi^{(1,1)} = O(|x|^{-1})\ \ &as\ |x|\rightarrow +\infty,\ms \\
\ds \Delta p^{(1,1)}= 0 \quad \ &\mbox{in} \  \mathbb{R}^2\setminus\overline{D},\ms \\
\ds \frac{\partial p^{(1,1)}}{\partial \nu_{D}} = A_1 \quad &\mbox{on } \partial D,\ms \\
\ds p^{(1,1)}|_{+}-p^{(1,1)}|_{-}=B_1 \quad &\mbox{on} \ \partial \Omega,\ms \\
\ds \frac{\partial p^{(1,1)}}{\partial \nu_{\Omega}} \big|_{+} - \frac{\partial p^{(1,1)}}{\partial \nu_{\Omega}} \big|_{-} = C_1 \quad &\mbox{on } \partial \Omega, \ms \\
\ds  p^{(1,1)} = O(|x|^{-1})\ \ &as\ |x|\rightarrow +\infty.
\end{cases}
\end{align*}

To decompose the boundary terms, we write $\gamma_a^{-1}$ as a Fourier series expansion
\begin{align*}
%\label{gamma-i-1-series}
\gamma_a^{-1}=\sum_{m=0}^\infty c_{a,2m}\cos(2m\eta),
\end{align*}
since $\gamma_a^{-1}$ is even function with respect to $\eta$, where the coefficients $d_m$  and  $h_m$ are defined as
\begin{align*}
 c_{a,0} = \frac{1}{2\pi}\int_{-\pi}^{\pi}\gamma_a^{-1} \d \eta, \quad  c_{a,2m} = \frac{1}{\pi}\int_{-\pi}^{\pi}\gamma_a^{-1} \cos(2m\eta) \d \eta, \quad a=i, e,\quad m= 1,  2,\ldots \ .
\end{align*}
For simplicity, we write
\begin{align}\label{gamma-a-inv}
\gamma_a^{-1}=\gamma_{a,0}^{-1}+ \gamma_{a,1}^{-1},
\end{align}
with
\begin{align*}
\gamma_{a,0}^{-1} =c_{a,0} \quad \mbox{and}\quad
\gamma_{a,1}^{-1} =\sum_{m=1}^\infty c_{a,2m}\cos(2m\eta).
\end{align*}
Then by \eqref{normal-tangent-derivatie} and \eqref{gamma-a-inv}, these boundary terms in \eqref{boundary-term} can be decomposed into a summation. We first have
{\small
\begin{align*}
  E=&\frac{\p f}{\p T_D}\frac{\partial  \varphi}{\partial T_D}-  f \frac{\partial^2  \varphi}{\partial \nu^2_{D}} =\gamma_i^{-2}\Big(\frac{\d f}{\d \eta} \frac{\p \varphi}{\p \eta}\Big|_{\xi=\xi_i}- f(\eta)\frac{\p^2 \varphi}{\p \xi^2}\Big|_{\xi=\xi_i} \Big)
  =\big(\gamma_{i,0}^{-1}+ \gamma_{i,1}^{-1}\big)\gamma_i^{-1}\Big(\frac{\d f}{\d \eta} \frac{\p \varphi}{\p \eta}\Big|_{\xi=\xi_i}- f(\eta)\frac{\p^2 \varphi}{\p \xi^2}\Big|_{\xi=\xi_i} \Big)\\
  =&\gamma_{i,0}^{-1}\gamma_i^{-1}\Big(\frac{\d f}{\d \eta} \frac{\p \varphi}{\p \eta}\Big|_{\xi=\xi_i}- f(\eta)\frac{\p^2 \varphi}{\p \xi^2}\Big|_{\xi=\xi_i} \Big)+ \gamma_{i,1}^{-1}\gamma_i^{-1}\Big(\frac{\d f}{\d \eta} \frac{\p \varphi}{\p \eta}\Big|_{\xi=\xi_i}- f(\eta)\frac{\p^2 \varphi}{\p \xi^2}\Big|_{\xi=\xi_i} \Big)\\
  =&E_0+E_1
\end{align*}}
on $\p D$. Next, we also obtain
{\small
\begin{align*}
  A  =&\frac{\p f}{\p T_D}\frac{\p p}{\p T_D} - f \frac{\p^2 p}{\p \nu_D^2}
      =\gamma_i^{-2}\Big(\frac{\d f}{\d \eta} \frac{\p p}{\p \eta}\Big|_{\xi=\xi_i}- f(\eta)\frac{\p^2 p}{\p \xi^2}\Big|_{\xi=\xi_i} \Big)
      = \big(\gamma_{i,0}^{-1}+ \gamma_{i,1}^{-1}\big)\gamma_i^{-1}\Big(\frac{\d f}{\d \eta} \frac{\p p}{\p \eta}\Big|_{\xi=\xi_i}- f(\eta)\frac{\p^2 p}{\p \xi^2}\Big|_{\xi=\xi_i} \Big)\\
      =&\gamma_{i,0}^{-1}\gamma_i^{-1}\Big(\frac{\d f}{\d \eta} \frac{\p p}{\p \eta}\Big|_{\xi=\xi_i}- f(\eta)\frac{\p^2 p}{\p \xi^2}\Big|_{\xi=\xi_i} \Big)+\gamma_{i,1}^{-1}\gamma_i^{-1}\Big(\frac{\d f}{\d \eta} \frac{\p p}{\p \eta}\Big|_{\xi=\xi_i}- f(\eta)\frac{\p^2 p}{\p \xi^2}\Big|_{\xi=\xi_i} \Big)\\
      =&A_0+A_1
\end{align*}}
on $\p D$, and
{\small
\begin{align*}
  B =& g\Big(\frac{\p p}{\p \nu_\Omega}\Big|_{-} - \frac{\p p}{\p \nu_\Omega}\Big|_{+} \Big)
     = -12\zeta_0 g\frac{\p \varphi}{\p \nu_\Omega}
     = -12\zeta_0 \gamma_e^{-1} g\frac{\p \varphi}{\p \xi}\Big|_{\xi=\xi_e}
     =-12\zeta_0 \big(\gamma_{e,0}^{-1}+\gamma_{e,1}^{-1}\big) g\frac{\p \varphi}{\p \xi}\Big|_{\xi=\xi_e}\\
        =&\Big(-12\zeta_0 \gamma_{e,0}^{-1} g\frac{\p \varphi}{\p \xi}\Big|_{\xi=\xi_e}\Big)+\Big(-12\zeta_0 \gamma_{e,1}^{-1} g\frac{\p \varphi}{\p \xi}\Big|_{\xi=\xi_e}\Big)\\
        =&B_0 +B_1
\end{align*}}
on $\p D$. We finally get
{\small
\begin{align*}
  C =& 12\zeta_0\Big(\frac{\partial \varphi^{(1)}}{\partial \nu_\Omega}+ g\frac{\partial^2 \varphi}{\partial \nu^2_\Omega}-g'\frac{\partial \varphi}{\partial T_\Omega}\Big)
  = 12\zeta_0 \gamma_e^{-1}\frac{\partial \varphi^{(1)}}{\p \xi}\Big|_{\xi=\xi_e}+12\zeta_0 \gamma_e^{-2}\Big( g\frac{\partial^2 \varphi}{\p \xi^2}\Big|_{\xi=\xi_e}- \frac{\d g}{\d \eta}\frac{\partial \varphi}{\p \eta}\Big|_{\xi=\xi_e}\Big)\\
    = &12\zeta_0 \gamma_e^{-1}\Big(\frac{\partial \varphi^{(1,0)}}{\p \xi}\Big|_{\xi=\xi_e}+\frac{\partial \varphi^{(1,1)}}{\p \xi}\Big|_{\xi=\xi_e}\Big)+12\zeta_0\big(\gamma_{e,0}^{-1}+ \gamma_{e,1}^{-1}\big) \gamma_e^{-1}\Big( g\frac{\partial^2 \varphi}{\p \xi^2}\Big|_{\xi=\xi_e}- \frac{\d g}{\d \eta}\frac{\partial \varphi}{\p \eta}\Big|_{\xi=\xi_e}\Big)\\
    =&12\zeta_0 \gamma_e^{-1}\Big(\frac{\partial \varphi^{(1,0)}}{\p \xi}\Big|_{\xi=\xi_e}+ \gamma_{e,0}^{-1}\Big( g\frac{\partial^2 \varphi}{\p \xi^2}\Big|_{\xi=\xi_e}- \frac{\d g}{\d \eta}\frac{\partial \varphi}{\p \eta}\Big|_{\xi=\xi_e}\Big)\Big)\\
    &+12\zeta_0 \gamma_e^{-1}\Big(\frac{\partial \varphi^{(1,1)}}{\p \xi}\Big|_{\xi=\xi_e}+\gamma_{e,1}^{-1}\Big( g\frac{\partial^2 \varphi}{\p \xi^2}\Big|_{\xi=\xi_e}- \frac{\d g}{\d \eta}\frac{\partial \varphi}{\p \eta}\Big|_{\xi=\xi_e}\Big)\Big)\\
    =&C_0+C_1,
\end{align*}}
where we used $\frac{\partial^2  p}{\partial \nu^2_{\Omega}}\big|_{-} = \frac{\partial^2  p}{\partial \nu^2_{\Omega}}\big|_{+}$ on $\p \Omega$, which is observed in the analytic solution \eqref{ellipse-p-dir-x}.

For a nonnegative integer $m$ and $a=i, e$, it is proven in \cite{Chung2014, Ando2016} that
\begin{align}\label{S-ellipse-cos}
\mathcal{S}_{\Gamma}[\beta_m^{c,a} ](x) = \begin{cases}
\ds -\frac{\cosh (m\xi)}{m\e^{m\xi_a}} \cos (m\eta),\quad & \xi  < \xi_a,\ms\\
\ds -\frac{\cosh (m\xi_a)}{m\e^{m\xi}}\cos (m\eta),\quad & \xi  > \xi_a,
\end{cases}
\end{align}
and
\begin{align}
\label{S-ellipse-sin}
\mathcal{S}_{\Gamma}[ \beta_m^{s,a}](x) = \begin{cases}
\ds -\frac{\sinh (m\xi)}{m\e^{m\xi_a}}\sin (m\eta),\quad & \xi  < \xi_a,\ms\\
\ds -\frac{\sinh (m\xi_a)}{m\e^{m\xi}}\sin (m\eta),\quad & \xi  > \xi_a.
\end{cases}
\end{align}
Moreover, we have
\begin{align}\label{K*-ellipse}
\mathcal{K}^*_\Gamma [\beta_m^{c,a}] = \frac{1}{2 e^{2m\xi_a}}\beta_m^{c,a} \quad \mbox{and} \quad \mathcal{K}^*_\Gamma [\beta_m^{s,a}] = -\frac{1}{2\e^{2m\xi_a}}\beta_m^{s,a}.
\end{align}

We also get
\begin{align}\label{D-ellipse-cos}
\mathcal{D}_{\Gamma}[\cos (m\eta) ](x) = \begin{cases}
\ds \frac{\sinh (m\xi)}{m\e^{m\xi_a}} \cos (m\eta),\quad & \xi  < \xi_a,\ms\\
\ds -\frac{\cosh (m\xi_a)}{m\e^{m\xi}}\cos (m\eta),\quad & \xi  > \xi_a,
\end{cases}
\end{align}
and
\begin{align}
\label{D-ellipse-sin}
\mathcal{D}_{\Gamma}[\sin (m\eta)](x) = \begin{cases}
\ds \frac{\cosh (m\xi)}{m\e^{m\xi_a}} \sin (m\eta),\quad & \xi  < \xi_a,\ms\\
\ds -\frac{\sinh (m\xi_a)}{m\e^{m\xi}} \sin (m\eta),\quad & \xi  > \xi_a.
\end{cases}
\end{align}
Moreover, we also have
\begin{align}\label{K-ellipse}
\mathcal{K}_\Gamma [ \cos (m\eta) ] = -\frac{1}{2 \e^{2m\xi_a}}\cos (m\eta) \quad \mbox{and} \quad \mathcal{K}_\Gamma [\sin (m\eta) ] = \frac{1}{2 \e^{2m\xi_a}}\sin (m\eta) .
\end{align}

With the above preparation, we now present the proof of Theorem \ref{thm-near-cloaking-ellipse} by calculating the explicit form of the solution to \eqref{first-order-equation-leader} using the representation formulas in Remark \ref{rem-representation-formula}.
\renewcommand{\proofname}{\indent Proof of  Theorem \ref{thm-near-cloaking-ellipse}}
\begin{proof}
Let $H(x)=\cosh (n\xi) \cos (n\eta)$ and $P = 12 \cosh (n\xi) \cos (n\eta)$  for $n\geq 1$.  We have already known in \cite{Liu2023}
\begin{equation}\label{ellipse-varphi}
  \varphi =\big(\cosh (n\xi) + \e^{n\xi_i} \sinh (n\xi_i) \;\e^{-n\xi}\big)\cos (n\eta),
\end{equation}
and
{\small
 \begin{align}\label{ellipse-p-dir-x}
 p=
  \begin{cases}
\ds -\frac{12}{\e^{n\xi_e}}\Big(\big(\sinh (n\xi_e) - \e^{n(\xi_i-\xi_e)}\sinh (n\xi_i)\big)\zeta_0 -\e^{n\xi_e}\Big)\Big(\cosh (n\xi) + \e^{n\xi_i} \sinh (n\xi_i) \, \e^{-n\xi}\Big)\cos (n\eta),\ \xi_i<\xi<\xi_e,
\vspace{1em}\\
\ds \Big( 12 \cosh (n\xi) - 12\Big(\zeta_0\big(\sinh (n\xi_e) - \e^{n(\xi_i-\xi_e)}\sinh (n\xi_i) \big)\cosh (n(\xi_e -\xi_i) )- \sinh (n\xi_i) \Big)\frac{\e^{n\xi_i}}{\e^{n\xi}}\Big)\cos (n\eta),\ \xi>\xi_e.
\end{cases}
 \end{align}}

We now solve the equations \eqref{first-order-equation-leader}. By the layer potential in Remark \ref{rem-representation-formula}, the solution $\varphi^{(1,0)}$ to \eqref{first-order-equation-leader} can be the following representation formula
\begin{equation}\label{ellipse-varphi1-leader}
\varphi^{(1,0)}=
\mathcal{S}_{D}[\phi_{1,0}](x),\quad x\in\mathbb{R}^2\setminus\overline{D},
\end{equation}
where
 \begin{align}\label{eq-density-ellipse-phi}
\Big(\frac{1}{2} I+\mathcal{K}^*_{D}\Big)[\phi_{1,0}]=E_0 \quad \mbox{on } \p D.
\end{align}
And the solution $p^{(1,0)}$ to \eqref{first-order-equation-leader} can be the following representation formula
\begin{equation}\label{ellipse-p1-leader}
p^{(1,0)}=
\mathcal{S}_{D}[\psi_{1,0}](x) + \mathcal{S}_{\Omega}[\psi_{2,0}](x) + \mathcal{D}_{\Omega}[\psi_{3,0}](x),\quad x\in\mathbb{R}^2\setminus\overline{D},
\end{equation}
where
\begin{align}
\begin{cases}\label{density-equation-v2-leading}
  \ds \Big(\frac{1}{2} I+\mathcal{K}^*_{D}\Big)[\psi_{1,0}] = -\frac{\p \mathcal{S}_{\Omega}[\psi_{2,0}] }{\p \nu_D} -\frac{\p \mathcal{D}_{\Omega}[\psi_{3,0}] }{\p \nu_D} + A_0 &\quad \mbox{on } \p D, \\
  \ds \psi_{2,0} = C_0  &\quad \mbox{on } \p \Omega, \\
  \ds \psi_{3,0} = - B_0  &\quad \mbox{on } \p \Omega.
\end{cases}
\end{align}

To determine  $\varphi^{(1,0)}$, we first write $f$ as a Fourier series expansion
\begin{align}\label{f-series-ellipse}
f(\eta)=\frac{a_0}{2}+\sum_{m=1}^\infty a_m\cos(m\eta) + b_m\sin(m\eta).
\end{align}
We then have
\begin{align*}
 E_0=&\gamma_{i,0}^{-1}\gamma_i^{-1}\Big(\frac{\d f}{\d \eta} \frac{\p \varphi}{\p \eta}\Big|_{\xi=\xi_i}- f(\eta)\frac{\p^2 \varphi}{\p \xi^2}\Big|_{\xi=\xi_i} \Big)\\
  =& -c_{i,0}\frac{n}{2}\e^{n\xi_i}\sum_{m=n}^{\infty}m\big((a_{m-n}-a_{m+n})\beta_m^{c, i} + (b_{m-n}-b_{m+n}) \beta_m^{s, i}\big).
\end{align*}
Solving \eqref{eq-density-ellipse-phi} by \eqref{K*-ellipse}, we get
\begin{align*}
  \phi_{1, 0} = -\frac{n}{2}\e^{n\xi_i}c_{i,0}\sum_{m=n}^{\infty}&(a_{m-n}-a_{m+n})\frac{m \e^{m\xi_i}}{\cosh(m\xi_i)}\beta_m^{c, i} +(b_{m-n}-b_{m+n}) \frac{m \e^{m\xi_i}}{\sinh(m\xi_i)}\beta_m^{s, i}.
\end{align*}
Hence using \eqref{S-ellipse-cos}, \eqref{S-ellipse-sin} and \eqref{ellipse-varphi1-leader} we obtain
\begin{align*}
  \varphi^{(1, 0)} = \frac{n}{2}c_{i,0} \sum_{m=n}^{\infty}\e^{(m+n)\xi_i}\e^{-m\xi}\big(&(a_{m-n}-a_{m+n})\cos(m\eta) + (b_{m-n}-b_{m+n})\sin(m\eta) \big).
\end{align*}

To proceed, we write $g$ as a Fourier series expansion
\begin{align}\label{g-series-ellipse}
g(\eta)=\frac{d_0}{2}+\sum_{m=0}^\infty d_m\cos(m\eta) + h_m\sin(m\eta),
\end{align}
where $d_m$ and $h_m$ are Fourier coefficients defined similarly to that of $f$. Then from  \eqref{ellipse-varphi}, \eqref{ellipse-p-dir-x}, \eqref{f-series-ellipse} and \eqref{g-series-ellipse}, $A_0$ and $B_0$ are given by the following expressions
{\small
\begin{align*}
  A_0  =&\gamma_{i,0}^{-1}\gamma_i^{-1}\Big(\frac{\d f}{\d \eta} \frac{\p p}{\p \eta}\Big|_{\xi=\xi_i}- f(\eta)\frac{\p^2 p}{\p \xi^2}\Big|_{\xi=\xi_i} \Big)\\
      = &\frac{6n\e^{n\xi_i}}{\e^{n\xi_e}}\Big(\big(\sinh (n\xi_e) - \e^{n(\xi_i-\xi_e)}\sinh (n\xi_i)\big)\zeta_0 -\e^{n\xi_e}\Big)c_{i,0}
       \sum_{m=n}^{\infty} m \big((a_{m-n}-a_{m+n})\beta_m^{c, i}+(b_{m-n}-b_{m+n})\beta_m^{s, i} \big),\\
  B_0 =& -12\zeta_0 \gamma_{e,0}^{-1} g\frac{\p \varphi}{\p \xi}\Big|_{\xi=\xi_e}\\
     = &   - 6 n \zeta_0\big(\sinh (n\xi_e) - \e^{n(\xi_i - \xi_e)}\sinh(n\xi_i)\big)c_{e,0}
      \sum_{m=n}^{\infty}(d_{m-n}+d_{m+n})\cos(m \eta) +(h_{m-n}+h_{m+n})\sin(m \eta),
\end{align*}}
and
\begin{align*}
  C_0 =& 12\zeta_0 \gamma_e^{-1}\Big(\frac{\partial \varphi^{(1,0)}}{\p \xi}\Big|_{\xi=\xi_e}+ \gamma_{e,0}^{-1}\Big( g(\eta)\frac{\partial^2 \varphi}{\p \xi^2}\Big|_{\xi=\xi_e}- \frac{\d g}{\d \eta}\frac{\partial \varphi}{\p \eta}\Big|_{\xi=\xi_e}\Big)\Big)\\
  =& - 6n\zeta_0 c_{i,0}\sum_{m=n}^{\infty}m \e^{(m+n)\xi_i-m\xi_e}\big((a_{m-n}-a_{m+n})\beta_m^{c, e}+(b_{m-n}-b_{m+n})\beta_m^{s, e} \big)\\
  & + 6n\zeta_0\big(\cosh (n\xi_e) + \e^{n(\xi_i-\xi_e)} \sinh (n\xi_i)\big)c_{e,0}\sum_{m=n}^{\infty} m \big((d_{m-n}-d_{m+n})\beta_m^{c, e}
  +(h_{m-n}-h_{m+n})\beta_m^{s, e}\big).
\end{align*}

To find the density $\psi_{1,0}$ by \eqref{density-equation-v2-leading}, we need to know the following layer potentials. For $x\in \Omega\setminus\overline{D}$, from \eqref{S-ellipse-cos}, \eqref{S-ellipse-sin}, \eqref{D-ellipse-cos} and \eqref{D-ellipse-sin}, we have
\begin{align*}
\mathcal{S}_{\Omega}[\psi_{2,0}](x)
  &=6n\zeta_0\e^{n\xi_i}c_{i,0}\\
  &\times\sum_{m=n}^{\infty} \e^{m(\xi_i- 2 \xi_e)}\big((a_{m-n}-a_{m+n})\cosh(m\xi)\cos(m \eta) +(b_{m-n}-b_{m+n})\sinh(m\xi)\sin(m \eta)\big)\\
  &- 6n\zeta_0\big(\cosh (n\xi_e) + \e^{n(\xi_i-\xi_e)} \sinh (n\xi_i)\big)c_{e,0} \\
  &\times\sum_{m=n}^{\infty} \e^{-m\xi_e} \big((d_{m-n}-d_{m+n})\cosh(m\xi)\cos(m \eta)+(h_{m-n}-h_{m+n})\sinh(m\xi)\sin(m \eta)\big),
\end{align*}
and
\begin{align*}
  \mathcal{D}_{\Omega}[\psi_{3,0}](x)
     = & 6n \zeta_0\big(\sinh (n\xi_e) - \e^{n(\xi_i - \xi_e)}\sinh(n\xi_i)\big)c_{e,0}\\
     &\times\sum_{m=n}^{\infty}  \e^{-m\xi_e} \big((d_{m-n}+d_{m+n})\sinh(m\xi) \cos(m \eta)+(h_{m-n}+h_{m+n})\cosh(m\xi)\sin(m \eta)\big).
\end{align*}
Their normal derivatives on the boundary $\p D$ are given by
\begin{align*}
    \frac{\p \mathcal{S}_{\Omega}[\psi_{2,0}]}{\p \nu_D}
  =&6n\zeta_0\e^{n\xi_i}c_{i,0}\\
  &\times\sum_{m=n}^{\infty} m \e^{m(\xi_i- 2 \xi_e)}\big((a_{m-n}-a_{m+n})\sinh(m\xi_i)\beta_m^{c, i}+(b_{m-n}-b_{m+n})\cosh(m\xi_i)\beta_m^{s, i}\big)\\
  &- 6n\zeta_0\big(\cosh (n\xi_e) + \e^{n(\xi_i-\xi_e)} \sinh (n\xi_i)\big)c_{e,0}\\
  &\times\sum_{m=n}^{\infty}m \e^{-m\xi_e} \big((d_{m-n}-d_{m+n})\sinh(m\xi_i)\beta_m^{c, i}+ (h_{m-n}-h_{m+n})\cosh(m\xi_i)\beta_m^{s, i}\big),
\end{align*}
and
\begin{align*}
 \frac{\mathcal{\p D}_{\Omega}[\psi_{3,0}]}{\p \nu_D}
     =&6 n \zeta_0\big(\sinh (n\xi_e) - \e^{n(\xi_i - \xi_e)}\sinh(n\xi_i)\big)c_{e,0}\\
     &\times\sum_{m=n}^{\infty} m \e^{-m\xi_e} \big((d_{m-n}+d_{m+n})\cosh(m\xi_i)\beta_m^{c, i}+(h_{m-n}+h_{m+n})\sinh(m\xi_i)\beta_m^{s, i}\big).
\end{align*}
Solving \eqref{density-equation-v2-leading} by \eqref{K*-ellipse} and \eqref{K-ellipse}, we obtain
\begin{align*}
  \psi_{1,0}
  %&= \frac{\e^{m\xi_i}}{\cosh(m\xi_i)}\Big(-\frac{\p \mathcal{S}_{\Omega}[\psi_2]}{\p \nu_D}-\frac{\mathcal{\p D}_{\Omega}[\psi_3]}{\p \nu_D}+A \Big)\\
  =& -6n\zeta_0\e^{n\xi_i} c_{i,0}\\
  &\times\sum_{m=n}^{\infty} m \e^{m(2\xi_i- 2 \xi_e)}\big((a_{m-n}-a_{m+n})\tanh(m\xi_i)\beta_m^{c, i} +(b_{m-n}-b_{m+n})\coth(m\xi_i)\beta_m^{s, i})\big)\\
  &+ 6n\zeta_0\big(\cosh (n\xi_e) + \e^{n(\xi_i-\xi_e)} \sinh (n\xi_i)\big)c_{e,0}\\
  &\times\sum_{m=n}^{\infty}m \e^{m(\xi_i-\xi_e)} \big((d_{m-n}-d_{m+n})\tanh(m\xi_i)\beta_m^{c, i}+(h_{m-n}-h_{m+n})\coth(m\xi_i)\beta_m^{s, i}\big) \\
  &-6 n \zeta_0\big(\sinh (n\xi_e) - \e^{n(\xi_i - \xi_e)}\sinh(n\xi_i)\big)c_{e,0}\\
  &\times\sum_{m=n}^{\infty} m \e^{m(\xi_i-\xi_e)} \big((d_{m-n}+d_{m+n})\tanh(m\xi_i)\beta_m^{c, i}+(h_{m-n}+h_{m+n})\coth(m\xi_i)\beta_m^{s, i}\big)\\
  &+\frac{6n}{\e^{n\xi_e}}\Big(\big(\sinh (n\xi_e) - \e^{n(\xi_i-\xi_e)}\sinh (n\xi_i)\big)\zeta_0 -\e^{n\xi_e}\Big)\e^{n\xi_i}c_{i,0}\\
  &\times\sum_{m=n}^{\infty}  \big((a_{m-n}-a_{m+n})\frac{m\e^{m\xi_i}}{\cosh(m\xi_i)}\beta_m^{c,i}+(b_{m-n}-b_{m+n})\frac{m\e^{m\xi_i}}{\sinh(m\xi_i)}\beta_m^{s, i}\big).
\end{align*}
To find the $p^{(1,0)}$ in the $\R^2\setminus \overline{\Omega}$, we need to compute the following layer potentials. For $x\in \R^2\setminus \overline{\Omega}$, from \eqref{S-ellipse-cos}, \eqref{S-ellipse-sin}, \eqref{D-ellipse-cos} and \eqref{D-ellipse-sin}, we have
{\small
\begin{align*}
  \mathcal{S}_{D}[\psi_{1,0}](x)
    &=  6n\zeta_0\e^{n\xi_i} c_{i,0}\\
    &\times\sum_{m=n}^{\infty}\e^{m(2\xi_i- 2 \xi_e-\xi)}\big((a_{m-n}-a_{m+n})\sinh(m\xi_i)\cos(m \eta)+(b_{m-n}-b_{m+n})\cosh(m\xi_i)\sin(m \eta)\big)\\
    &-6n\zeta_0\big(\cosh (n\xi_e) + \e^{n(\xi_i-\xi_e)} \sinh (n\xi_i)\big)c_{e,0}\\
  &\times\sum_{m=n}^{\infty} \e^{m(\xi_i-\xi_e-\xi)} \big((d_{m-n}-d_{m+n})\sinh(m\xi_i)\cos(m \eta) +(h_{m-n}-h_{m+n})\cosh(m\xi_i)\sin(m \eta)\big) \\
  &+ 6 n \zeta_0\big(\sinh (n\xi_e) - \e^{n(\xi_i - \xi_e)}\sinh(n\xi_i)\big)c_{e,0}\\
  &\times\sum_{m=n}^{\infty} \e^{m(\xi_i-\xi_e-\xi)} \big((d_{m-n}+d_{m+n})\cosh(m\xi_i) \cos(m \eta)+(h_{m-n}+h_{m+n})\sinh(m\xi_i)\sin(m \eta)\big)\\
  &-\frac{6n}{\e^{n\xi_e}}\Big(\big(\sinh (n\xi_e) - \e^{n(\xi_i-\xi_e)}\sinh (n\xi_i)\big)\zeta_0 -\e^{n\xi_e}\Big)\e^{n\xi_i}c_{i,0}\\
  &\times\sum_{m=n}^{\infty} \e^{m\xi_i-m\xi}  \big((a_{m-n}-a_{m+n})\cos(m \eta)+(b_{m-n}-b_{m+n})\sin(m \eta)\big),
\end{align*}}
\begin{align*}
\mathcal{S}_{\Omega}[\psi_{2,0}](x)
  &= 6n\zeta_0\e^{n\xi_i}c_{i,0}\\
  &\times\sum_{m=n}^{\infty} \e^{m(\xi_i-  \xi_e-\xi )}\big((a_{m-n}-a_{m+n})\cosh(m\xi_e)\cos(m \eta)+(b_{m-n}-b_{m+n})\sinh(m\xi_e)\sin(m \eta)\big)\\
  &-6n\zeta_0\big(\cosh (n\xi_e) + \e^{n(\xi_i-\xi_e)} \sinh (n\xi_i)\big)c_{e,0}\\
  &\times\sum_{m=n}^{\infty}\e^{-m\xi} \big((d_{m-n}-d_{m+n})\cosh(m\xi_e)\cos(m \eta) +(h_{m-n}-h_{m+n})\sinh(m\xi_e)\sin(m \eta)\big),
\end{align*}
%for $x\in \R^2\setminus \overline{\Omega}$.
\begin{align*}
  \mathcal{D}_{\Omega}[\psi_{3,0}](x)
     =&  -6 n \zeta_0\big(\sinh (n\xi_e) - \e^{n(\xi_i - \xi_e)}\sinh(n\xi_i)\big)c_{e,0}\\
    &\times\sum_{m=n}^{\infty}  \e^{-m\xi} \big((d_{m-n}+d_{m+n})\cosh(m\xi_e) \cos(m \eta)+(h_{m-n}+h_{m+n})\sinh(m\xi_e)\sin(m \eta)\big).
\end{align*}
%or $x\in \R^2\setminus \overline{\Omega}$.
Therefore by \eqref{ellipse-p1-leader} we have
\begin{align*}
p^{(1,0)} = \sum_{m=n}^{\infty} \e^{-m\xi}\big(M_{m,n}^{1}\cos(m\eta)+ M_{m,n}^{2}\sin(m\eta)\big),\quad \mbox{in }\R^2\setminus \overline{\Omega},
\end{align*}
where
\begin{align}
  M_{m,n}^{1} =&6n\zeta_0\Big(\e^{n\xi_i}\big(\e^{m(2\xi_i-2\xi_e)}\sinh(m\xi_i)+\e^{m(\xi_i-\xi_e)}\cosh(m\xi_e) \nonumber\\
   & -\e^{m\xi_i-n\xi_e}\big(\sinh(n\xi_e)-\e^{n(\xi_i-\xi_e)}
  \sinh(n\xi_i)-\zeta_0^{-1}\e^{n\xi_e}\big)\big) c_{i,0}(a_{m-n}-a_{m+n}) \nonumber\\
   &- \big(\cosh (n\xi_e) + \e^{n(\xi_i-\xi_e)} \sinh (n\xi_i)\big)\big(\e^{m\xi_i-m\xi_e}\sinh(m\xi_i)+\cosh(m\xi_e)\big) c_{e,0}(d_{m-n}-d_{m+n})\nonumber\\
   &+\big(\sinh (n\xi_e) - \e^{n(\xi_i - \xi_e)}\sinh(n\xi_i)\big)\big(\e^{m\xi_i-m\xi_e}\cosh(m\xi_i)-\cosh(m\xi_e)\big) c_{e,0}(d_{m-n}+d_{m+n})\Big),\label{ellipse-M-1}\\
   M_{m,n}^{2} =&6n\zeta_0\Big(\e^{n\xi_i}\big(\e^{m(2\xi_i-2\xi_e)}\cosh(m\xi_i)+\e^{m(\xi_i-\xi_e)}\sinh(m\xi_e)\nonumber\\
  &-\e^{m\xi_i-n\xi_e}\big(\sinh(n\xi_e)-\e^{n(\xi_i-\xi_e)}
  \sinh(n\xi_i)-\zeta_0^{-1}\e^{n\xi_e}\big)\big) c_{i,0}(b_{m-n}-b_{m+n})\nonumber\\
   &-\big(\cosh (n\xi_e) + \e^{n(\xi_i-\xi_e)} \sinh (n\xi_i)\big)\big(\e^{m\xi_i-m\xi_e}\cosh(m\xi_i)+\sinh(m\xi_e)\big) c_{e,0}(h_{m-n}-h_{m+n})\nonumber\\
   &+\big(\sinh (n\xi_e) - \e^{n(\xi_i - \xi_e)}\sinh(n\xi_i)\big)\big(\e^{m\xi_i-m\xi_e}\sinh(m\xi_i)-\sinh(m\xi_e)\big) c_{e,0}(h_{m-n}+h_{m+n})\Big), \label{ellipse-M-2}
\end{align}
where $\zeta_0$ satisfies \eqref{ellipse-cloaking-zeta-x}, and $ M_{m,n}^{1}$, $ M_{m,n}^{2}$ are called first-order scattering coefficients.

We are now in a position to find recursive formulas such that $M_{m,n}^{1}=0$ and $M_{m,n}^{2}=0$. The following properties are useful to simplify coefficients $M_{m,n}^{1}$ and $M_{m,n}^{2}$.
\begin{align}\label{properties-sinh-cosh-1}
  \e^{x}\cosh (x) + \e^{y} \sinh (y) = \frac{1}{2}\big(\e^{2x}+\e^{2y}\big),\\
    \e^{x}\cosh (x) - \e^{y} \cosh (y) = \frac{1}{2}\big(\e^{2x}-\e^{2y}\big),\label{properties-sinh-cosh-2}\\
      \e^{x}\sinh (x) - \e^{y} \sinh (y) = \frac{1}{2}\big(\e^{2x}-\e^{2y}\big).\label{properties-sinh-cosh-3}
\end{align}
If  we require $ M_{m,n}^{1}=0$ and  $ M_{m,n}^{2}=0$, then by \eqref{properties-sinh-cosh-1}--\eqref{properties-sinh-cosh-3} the following equality establishes.
\begin{align*}
 &\e^{(m+n)\xi_i+(m-n)\xi_e}\Big(\big(\e^{2n\xi_e} -\e^{2n\xi_i}\big)\frac{e^{2n\xi_e}+1}{e^{2n\xi_i}-1}+\e^{-2(m-n)\xi_e}\big(\e^{2m\xi_i}+\e^{2m\xi_e}\big)\Big)c_{i,0}(a_{m-n}-a_{m+n})\\
&- \big(\e^{2(m+n)\xi_i}+\e^{2(m+n)\xi_e}\big)c_{e,0}d_{m-n} + \big(\e^{2m\xi_i+2n\xi_e}+\e^{2m\xi_e+2n\xi_i}\big)c_{e,0}d_{m+n}=0,
\end{align*}
and
\begin{align*}
 &\e^{(m+n)\xi_i+(m-n)\xi_e}\Big(\big(\e^{2n\xi_e} -\e^{2n\xi_i}\big)\frac{e^{2n\xi_e}+1}{e^{2n\xi_i}-1}+\e^{-2(m-n)\xi_e}\big(\e^{2m\xi_i}+\e^{2m\xi_e}\big)\Big)c_{i,0}(b_{m-n}-b_{m+n})\\
&- \big(\e^{2(m+n)\xi_i}+\e^{2(m+n)\xi_e}\big)c_{e,0}h_{m-n} + \big(\e^{2m\xi_i+2n\xi_e}+\e^{2m\xi_e+2n\xi_i}\big)c_{e,0}h_{m+n}=0,
\end{align*}
where we substitute the expressions for $\cosh(\cdot)$ and $\sinh(\cdot)$ into \eqref{ellipse-M-1} and \eqref{ellipse-M-2}.

Rearranging and changing the subscripts according to $m \rightarrow m + n$ yields
\begin{align}\label{recursive equations-d-1-cos-ellipse}
d_{m}
 =&\frac{\e^{2(m+n)\xi_i+2n\xi_e}+\e^{2(m+n)\xi_e+2n\xi_i} }{\e^{2(m+2n)\xi_i}+\e^{2(m+2n)\xi_e}}d_{m+2n}
 +\frac{\e^{(m+2n)\xi_i+m\xi_e}}{\e^{2(m+2n)\xi_i} + \e^{2(m+2n)\xi_e}} \nonumber\\
 &\times\Big(\big(\e^{2n\xi_e} -\e^{2n\xi_i}\big)\frac{e^{2n\xi_e}+1}{e^{2n\xi_i}-1}+\e^{-2m\xi_e}\big(\e^{2(m+n)\xi_i}+\e^{2(m+n)\xi_e}\big)\Big)\frac{c_{i,0}}{c_{e,0}}(a_{m}-a_{m+2n}),\\
h_{m}
 =&\frac{\e^{2(m+n)\xi_i+2n\xi_e}+\e^{2(m+n)\xi_e+2n\xi_i} }{\e^{2(m+2n)\xi_i}+\e^{2(m+2n)\xi_e}}h_{m+2n}
 +\frac{\e^{(m+2n)\xi_i+m\xi_e}}{\e^{2(m+2n)\xi_i} + \e^{2(m+2n)\xi_e}}\nonumber \\
 &\times\Big(\big(\e^{2n\xi_e} -\e^{2n\xi_i}\big)\frac{e^{2n\xi_e}+1}{e^{2n\xi_i}-1}+\e^{-2m\xi_e}\big(\e^{2(m+n)\xi_i}+\e^{2(m+n)\xi_e}\big)\Big)\frac{c_{i,0}}{c_{e,0}}(b_{m}-b_{m+2n}),\label{recursive equations-h-1-cos-ellipse}
\end{align}
or
\begin{align}\label{recursive equations-d-2-cos-ellipse}
d_{m+2n}
 =&\frac{\e^{2(m+2n)\xi_i}+\e^{2(m+2n)\xi_e}}{\e^{2(m+n)\xi_i+2n\xi_e}+\e^{2(m+n)\xi_e+2n\xi_i}}d_{m}
 -\frac{\e^{(m+2n)\xi_i+m\xi_e}}{\e^{2(m+n)\xi_i+2n\xi_e}+\e^{2(m+n)\xi_e+2n\xi_i}}\nonumber\\
 &\times\Big(\big(\e^{2n\xi_e} -\e^{2n\xi_i}\big)\frac{e^{2n\xi_e}+1}{e^{2n\xi_i}-1}+\e^{-2m\xi_e}\big(\e^{2(m+n)\xi_i}+\e^{2(m+n)\xi_e}\big)\Big)\frac{c_{i,0}}{c_{e,0}}(a_{m}-a_{m+2n}),\\
h_{m+2n}
 =&\frac{\e^{2(m+2n)\xi_i}+\e^{2(m+2n)\xi_e}}{\e^{2(m+n)\xi_i+2n\xi_e}+\e^{2(m+n)\xi_e+2n\xi_i}}h_{m}
 -\frac{\e^{(m+2n)\xi_i+m\xi_e}}{\e^{2(m+n)\xi_i+2n\xi_e}+\e^{2(m+n)\xi_e+2n\xi_i}} \nonumber\\
 &\times\Big(\big(\e^{2n\xi_e} -\e^{2n\xi_i}\big)\frac{e^{2n\xi_e}+1}{e^{2n\xi_i}-1}+\e^{-2m\xi_e}\big(\e^{2(m+n)\xi_i}+\e^{2(m+n)\xi_e}\big)\Big)\frac{c_{i,0}}{c_{e,0}}(b_{m}-b_{m+2n}).\label{recursive equations-h-2-cos-ellipse}
\end{align}

Let $H(x) = \sinh(n\xi)\sin(n\eta)$ and $P(x)=12 \sinh(n\xi)\sin(n\eta)$ for $n\geq 1$. In a similar way, we have
\begin{align}\label{recursive equations-d-1-sin-ellipse}
d_{m}
 =&-\frac{\e^{2(m+n)\xi_i+2n\xi_e}+\e^{2(m+n)\xi_e+2n\xi_i} }{\e^{2(m+2n)\xi_i}+\e^{2(m+2n)\xi_e}}d_{m+2n}
+\frac{\e^{(m+2n)\xi_i+m\xi_e}}{\e^{2(m+2n)\xi_i} + \e^{2(m+2n)\xi_e}}\nonumber\\
&\times\Big(\big(\e^{2n\xi_e} -\e^{2n\xi_i}\big)\frac{e^{2n\xi_e}-1}{e^{2n\xi_i}+1}+\e^{-2m\xi_e}\big(\e^{2(m+n)\xi_i}+\e^{2(m+n)\xi_e}\big)\Big)\frac{c_{i,0}}{c_{e,0}}(a_{m}+a_{m+2n}),\\
h_{m}
 =&-\frac{\e^{2(m+n)\xi_i+2n\xi_e}+\e^{2(m+n)\xi_e+2n\xi_i} }{\e^{2(m+2n)\xi_i}+\e^{2(m+2n)\xi_e}}h_{m+2n}
+\frac{\e^{(m+2n)\xi_i+m\xi_e}}{\e^{2(m+2n)\xi_i} + \e^{2(m+2n)\xi_e}}\nonumber\\
&\times\Big(\big(\e^{2n\xi_e} -\e^{2n\xi_i}\big)\frac{e^{2n\xi_e}-1}{e^{2n\xi_i}+1}+\e^{-2m\xi_e}\big(\e^{2(m+n)\xi_i}+\e^{2(m+n)\xi_e}\big)\Big)\frac{c_{i,0}}{c_{e,0}}(b_{m}+b_{m+2n}),\label{recursive equations-h-1-sin-ellipse}
\end{align}
or
\begin{align}\label{recursive equations-d-2-sin-ellipse}
d_{m+2n}
 =&-\frac{\e^{2(m+2n)\xi_i}+\e^{2(m+2n)\xi_e}}{\e^{2(m+n)\xi_i+2n\xi_e}+\e^{2(m+n)\xi_e+2n\xi_i}}d_{m}
 +\frac{\e^{(m+2n)\xi_i+m\xi_e}}{\e^{2(m+n)\xi_i+2n\xi_e}+\e^{2(m+n)\xi_e+2n\xi_i}}\nonumber\\
 &\times\Big(\big(\e^{2n\xi_e} -\e^{2n\xi_i}\big)\frac{e^{2n\xi_e}-1}{e^{2n\xi_i}+1}+\e^{-2m\xi_e}\big(\e^{2(m+n)\xi_i}+\e^{2(m+n)\xi_e}\big)\Big)\frac{c_{i,0}}{c_{e,0}}(a_{m}+a_{m+2n}),\\
h_{m+2n}
 =&-\frac{\e^{2(m+2n)\xi_i}+\e^{2(m+2n)\xi_e}}{\e^{2(m+n)\xi_i+2n\xi_e}+\e^{2(m+n)\xi_e+2n\xi_i}}h_{m}
 +\frac{\e^{(m+2n)\xi_i+m\xi_e}}{\e^{2(m+n)\xi_i+2n\xi_e}+\e^{2(m+n)\xi_e+2n\xi_i}} \nonumber\\
 &\times\Big(\big(\e^{2n\xi_e} -\e^{2n\xi_i}\big)\frac{e^{2n\xi_e}-1}{e^{2n\xi_i}+1}+\e^{-2m\xi_e}\big(\e^{2(m+n)\xi_i}+\e^{2(m+n)\xi_e}\big)\Big)\frac{c_{i,0}}{c_{e,0}}(b_{m}+b_{m+2n}).\label{recursive equations-h-2-sin-ellipse}
\end{align}

The other setting about $d$ and $h$ is similar to the annulus case, we also set
\begin{align}\label{ellipse-d-h-nmax}
  d_{m_{max}+1}=\cdots=  d_{m_{max}+2n}=  h_{m_{max}+1}= \cdots= h_{m_{max}+2n}=0.
\end{align}
Hence a shape function $g$ can be constructed by $d_m =\{d_1,d_2,\ldots,d_{m_{max}} \}$ and  $h_m =\{h_1,h_2,\ldots,h_{m_{max}} \}$, where $d_m$ and $h_m$ are
determined from (\ref{recursive equations-d-1-cos-ellipse}), (\ref{recursive equations-h-1-cos-ellipse}) (or (\ref{recursive equations-d-1-sin-ellipse}) and (\ref{recursive equations-h-1-sin-ellipse})), (\ref{ellipse-d-h-nmax}).

The proof is complete.
\end{proof}

\section{Numerical simulations}\label{sec-num-sim}
In this section, we validate the theoretical results  by performing two-dimensional finite-element simulations, which shows good agreement. We perform finite-element numerical simulations using the commercial software COMSOL Multiphysics. In what follows, we assume that $\p D_\epsilon$ is a perturbation of a circle (or ellipse) with $r_i = 1$ (or $\xi_i=0.5$), as shown in Figure \ref{fig-perturbation-circle-ellipse}, and  $\p \Omega_\epsilon$ is a perturbation of a circle (or ellipse) with $r_e = 2$ (or $\xi_e=1$).

 \begin{figure}[H]
	\centering  %图片全局居中
	\subfigbottomskip=-10pt %两行子图之间的行间距
	\subfigcapskip=-10pt %设置子图与子标题之间的距离
	\subfigure[]{
		\includegraphics[width=0.32\linewidth]{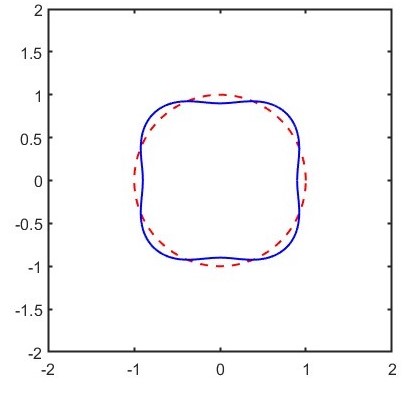}}\qquad
	%\quad
	\subfigure[]{
		\includegraphics[width=0.32\linewidth]{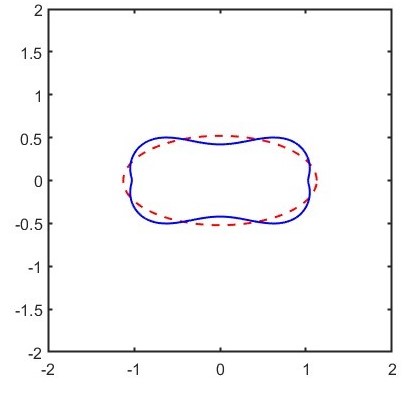}}
	\caption{The solid blue line represents the boundary of the object, which is a perturbation of a circle (or an ellipse), a dashed red line. The perturbation is given by $\epsilon f$, where $\epsilon = 0.1$ and $f=-\cos(4\theta)$. (a) rounded rectangle, (b) bone-shaped object.}\label{fig-perturbation-circle-ellipse}
\end{figure}

To quantify the enhanced effect, we define an evaluation function by
\begin{align}\label{Q}
  Q= \|p_\epsilon-P\|_{L^2(X)},
\end{align}
where $X$ denotes the computed domain. Form \eqref{Q}, we know that $Q=0$ when $\{(D, \Omega; \zeta_0)\}$ is perfect cloaking.  Throughout
this section, we have confirmed numerically the enhanced cloaking effect by comparing the perfect cloaking, $1$-order near-cloaking, and $2$-order near-cloaking. Here the $1$-order near-cloaking corresponds to a small inner boundary perturbation while the outer boundary is not perturbed, that is, $g=0$. The $2$-order near-cloaking corresponds that the inner and outer boundaries are simultaneously perturbed and Fourier coefficients of $g$ satisfy the recursive formulas in Section \ref{sec-hnc}.

We first consider the case of $D$ and $\Omega$ being concentric disks of radii $r_i=1$ and $r_e=2$ with $\zeta_0$ satisfying perfect cloaking. Figure \ref{fig-near-cloaking-n-1} presents a comparison of finite-element simulation results corresponding to perfect cloaking (a,d), $1$-order near-cloaking (b, e), and $2$-order near-cloaking (c, f) under a linear background field. Figures \ref{fig-near-cloaking-n-1}(a)--\ref{fig-near-cloaking-n-1}(c) present the resulting pressure distribution (colormap) and streamlines (white lines), showing excellent cloaking for all three cases. Under cloaking conditions, the streamlines outside of the control region are straight, unmodified relative to the uniform far field, and undisturbed by the object. In Figures \ref{fig-near-cloaking-n-1}(d)--\ref{fig-near-cloaking-n-1}(f) we compare the outer scattered field, showing that $2$-order near-cloaking has smaller scattering relative to $1$-order near-cloaking. This indicates that $2$-order near-cloaking has an enhanced cloaking effect.
To quantify this effect, we compute the evaluation function $Q$ using the equation \eqref{Q}, where $X$ denotes that square region minus $\Omega_\epsilon$. The computed results are summarized in Table \ref{tab-Q-circle}, which presents the comparison of $Q$ for all three cases, clearly indicating that $2$-order near-cloaking has smaller scattering. In addition, we also compare the scattered field on the circle of radius $3$, as shown in Figure \ref{fig-near-cloaking-circle-r-3}(a), showing that the scattering from $2$-order near-cloaking is smaller. The non-linear background field is also considered in Figure \ref{fig-near-cloaking-n-2}, Table \ref{tab-Q-circle}, and  Figure \ref{fig-near-cloaking-circle-r-3}(b), where $n=2$. In summary, these results clearly show that $2$-order near-cloaking has an enhanced cloaking effect relative to $1$-order near-cloaking and validate Theorem \ref{thm-near-cloaking-circle}. The performance of the proposed enhanced near-cloaking conditions has been numerically confirmed.
 \begin{figure}[H]
	\centering  %图片全局居中
	\subfigbottomskip=-10pt %两行子图之间的行间距
	\subfigcapskip=-10pt %设置子图与子标题之间的距离
	\subfigure[]{
		\includegraphics[width=0.32\linewidth]{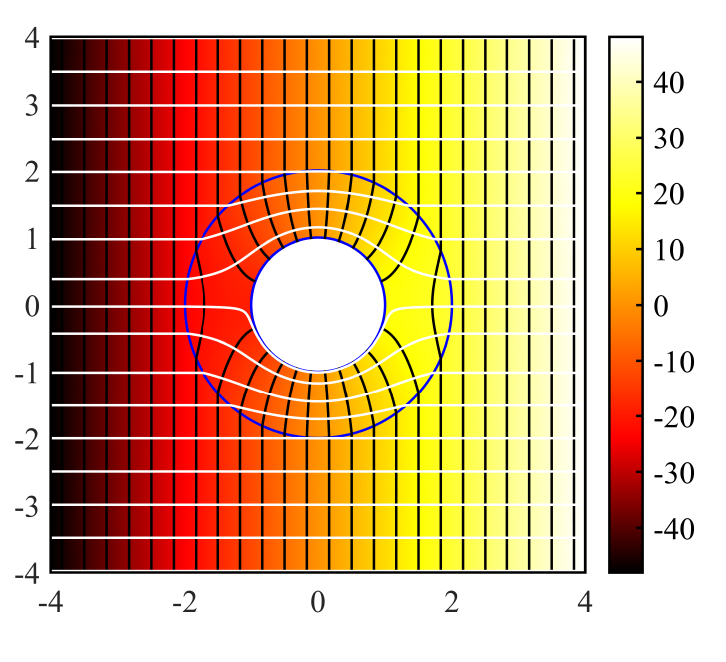}}
	\subfigure[]{
		\includegraphics[width=0.32\linewidth]{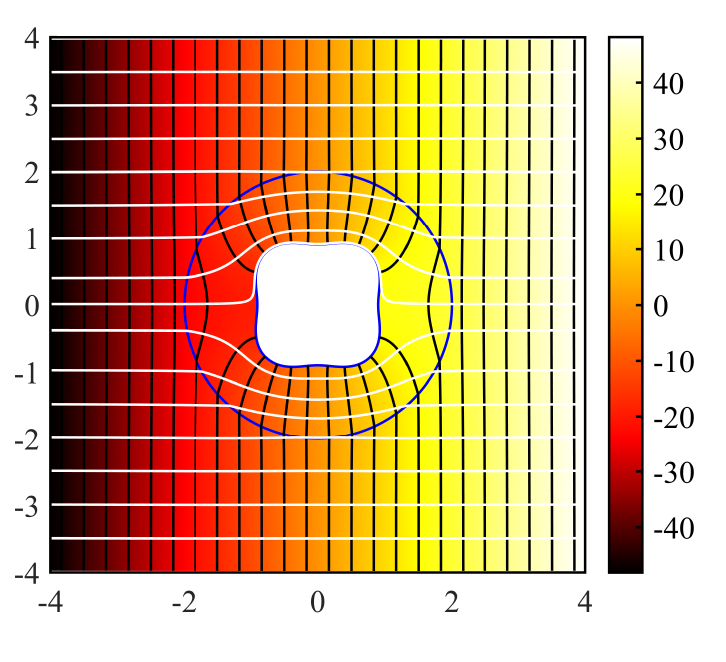}}
	\subfigure[]{
	\includegraphics[width=0.32\linewidth]{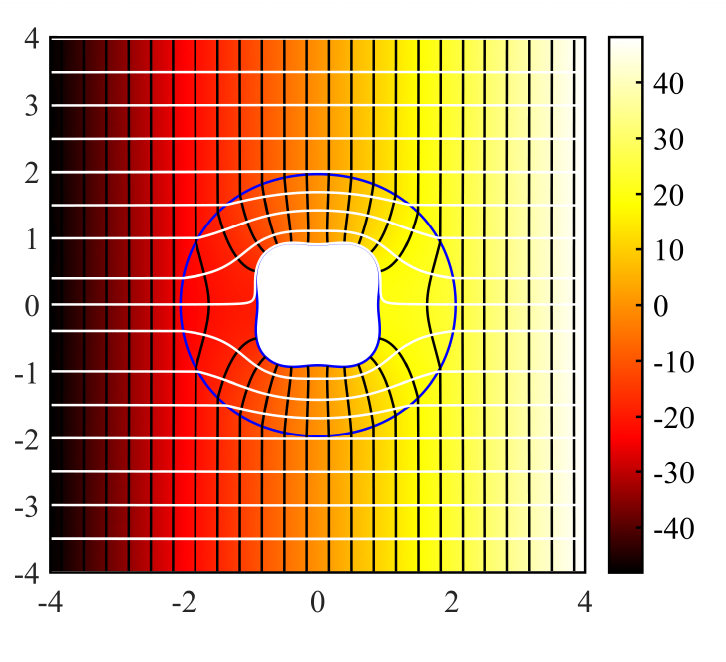}}\\
	\subfigure[]{
		\includegraphics[width=0.32\linewidth]{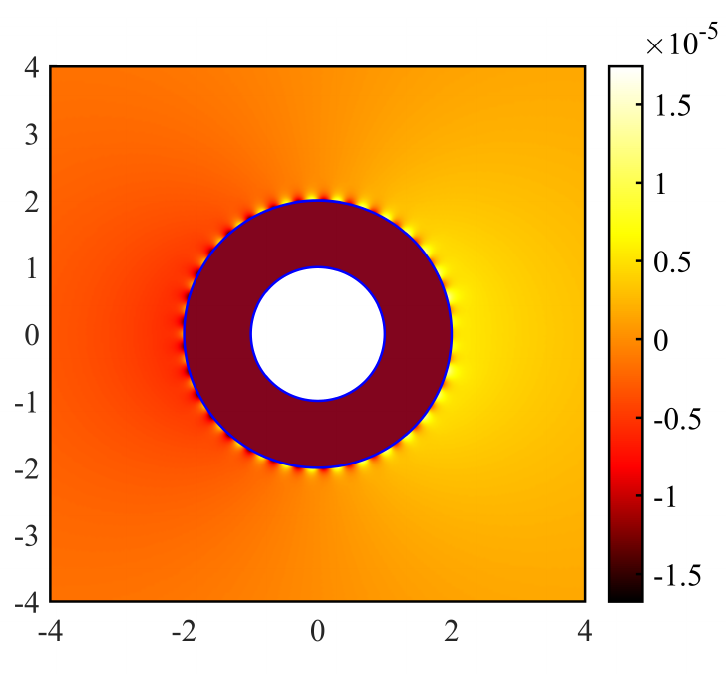}}
	\subfigure[]{
	\includegraphics[width=0.32\linewidth]{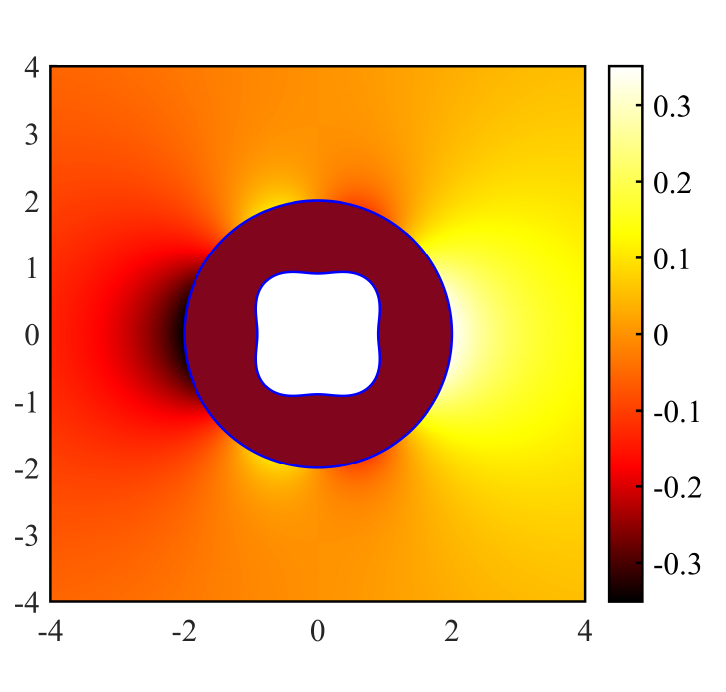}}
	\subfigure[]{
		\includegraphics[width=0.32\linewidth]{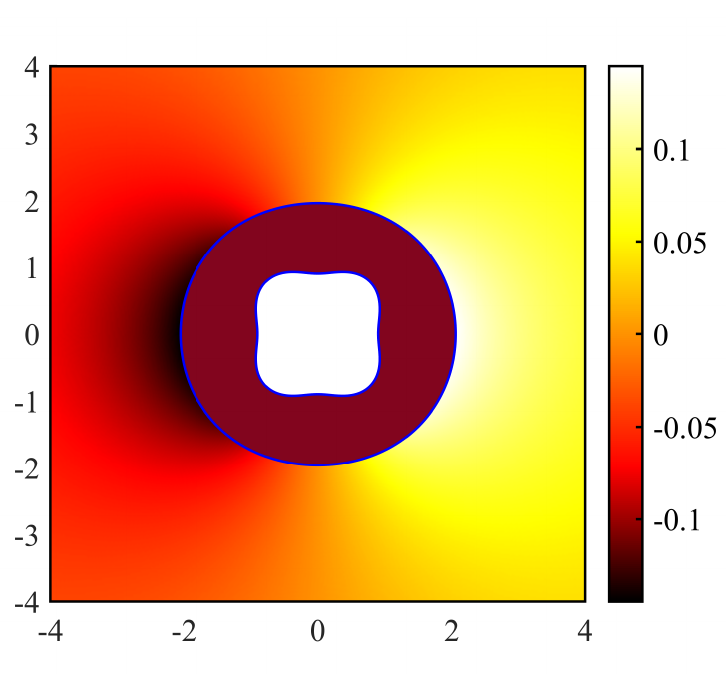}}
	\caption{Numerical results for the pressure distribution (colormap) and streamlines (white lines), corresponding to different order near-cloaking.
Top: outer total field $p_\epsilon$; bottom: outer scattered field $p_\epsilon-P$; left: perfect cloaking; middle: 1-order near-cloaking; right: 2-order near-cloaking. From \eqref{recursive equations-d-1-cos} and \eqref{recursive equations-h-1-cos}, the Fourier coefficients $d_m$ of $g$ are obtained as $d_0=0.2197$, $d_2=0.4669$, $d_4=-0.125$, where $n=1$.}\label{fig-near-cloaking-n-1}
\end{figure}

\begin{table}[!htbp]
  \caption{Evaluation function $Q$ with different cloaking and $n$}\label{tab-Q-circle}
  \centering
  \begin{tabular}{cccc}
    \toprule
    n   & perfect cloaking   &1-order near-cloaking &2-order near-cloaking \\
    \midrule
    1 & 0&  0.782010 & 0.480427\\
    2 & 0& 1.788966& 1.145818\\
    \bottomrule
  \end{tabular}
\end{table}

 \begin{figure}[H]
	\centering  %图片全局居中
	\subfigbottomskip=-10pt %两行子图之间的行间距
	\subfigcapskip=0pt %设置子图与子标题之间的距离
	\subfigure[]{
		\includegraphics[width=0.45\linewidth]{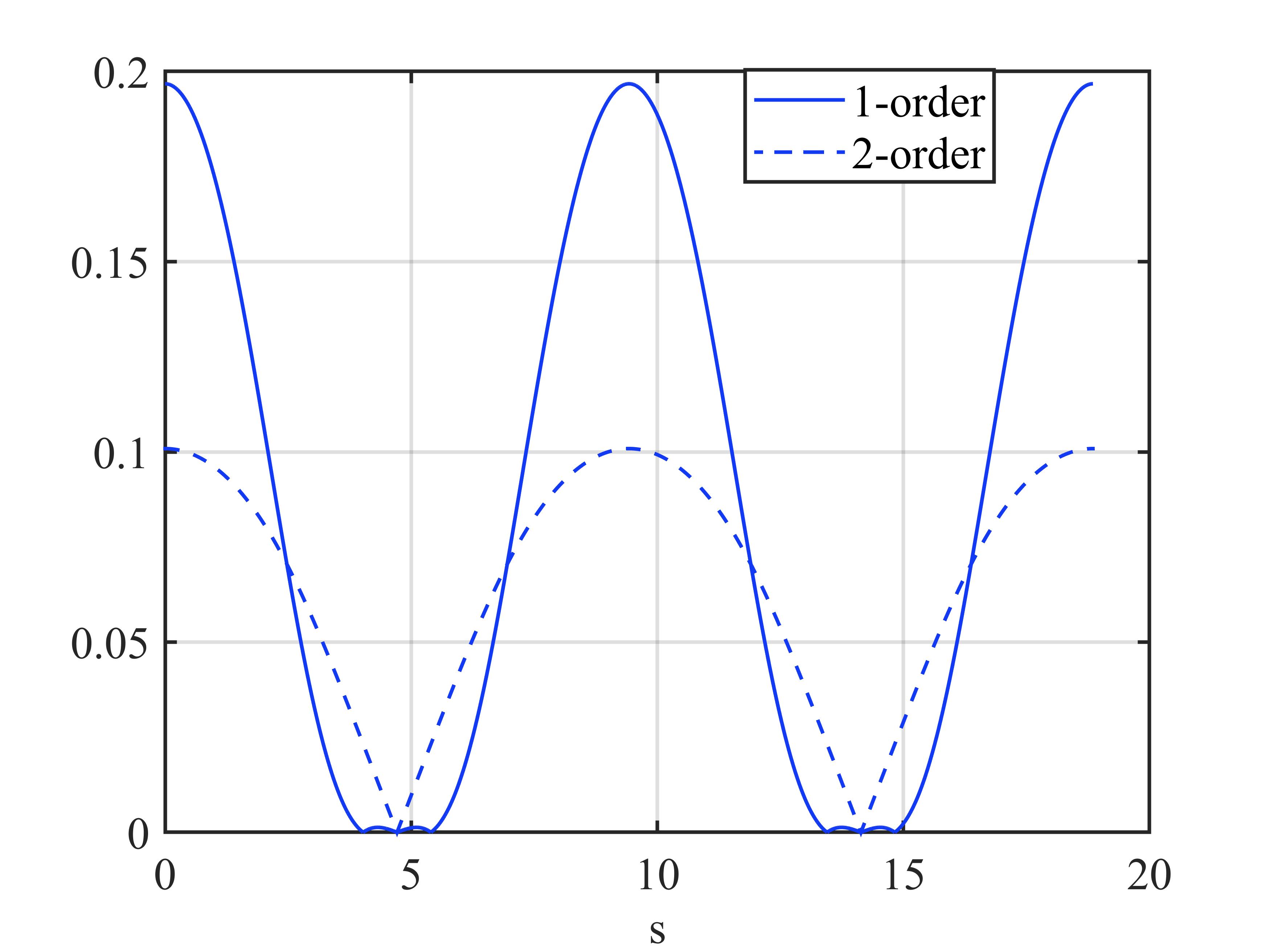}}
	\subfigure[]{
		\includegraphics[width=0.45\linewidth]{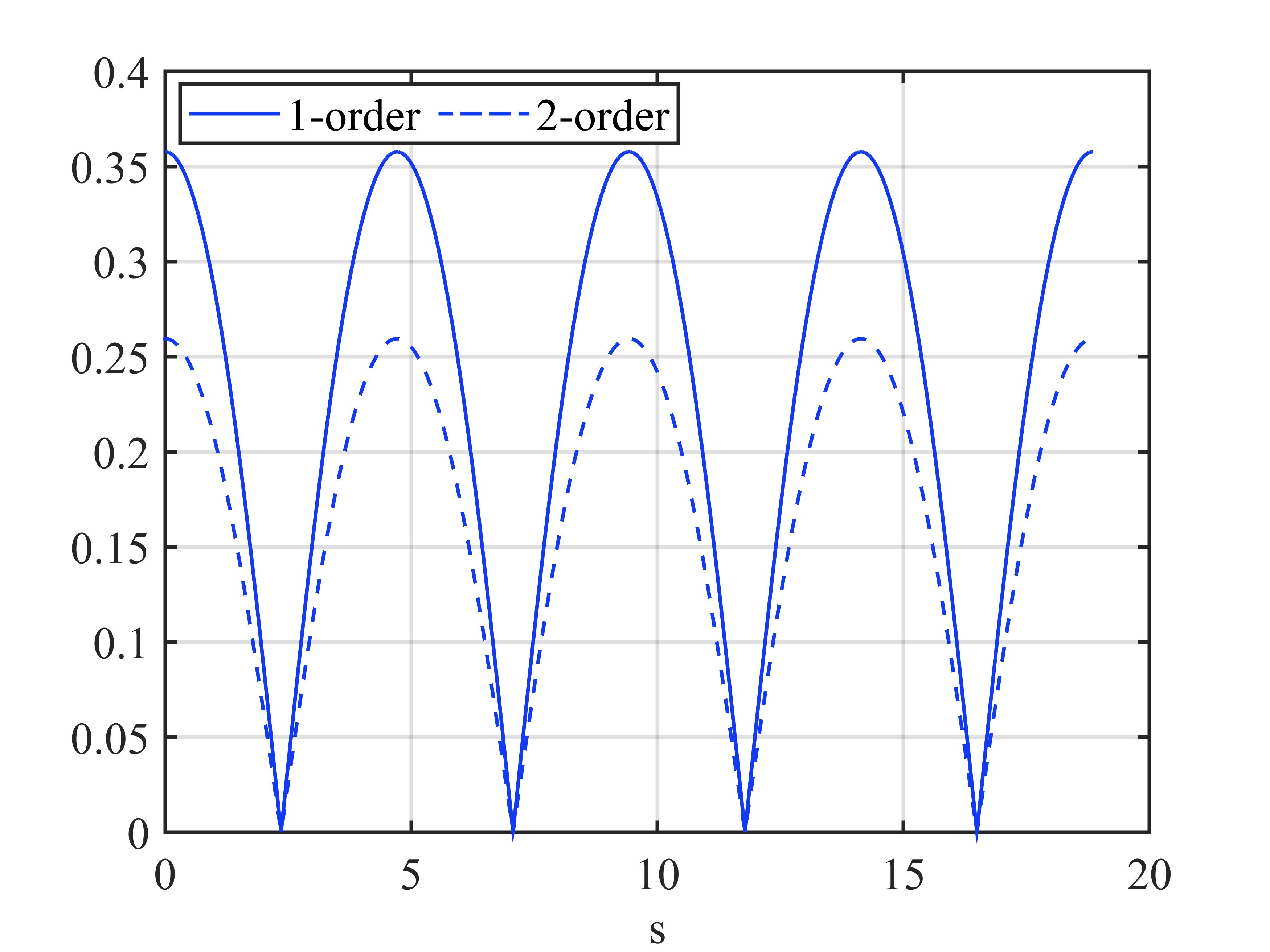}}
	\caption{Absolute value of outer scattered field on the circle of radius 3; left: background field $P=12r\cos(\theta)$; right: background field $P=12r^2\cos(2\theta)$. Here $s$ denotes arc length. }\label{fig-near-cloaking-circle-r-3}
\end{figure}

 \begin{figure}[H]
	\centering  %图片全局居中
	\subfigbottomskip=-10pt %两行子图之间的行间距
	\subfigcapskip=-10pt %设置子图与子标题之间的距离
	\subfigure[]{
		\includegraphics[width=0.32\linewidth]{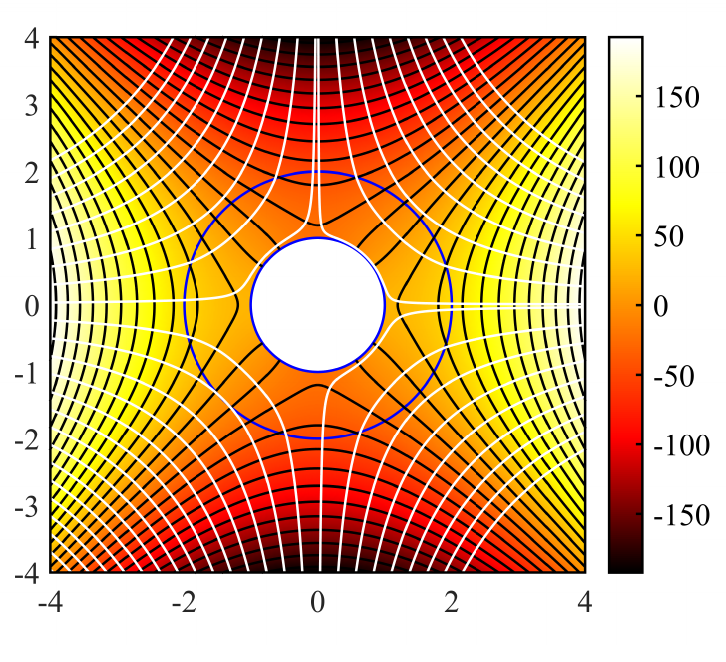}}
	\subfigure[]{
		\includegraphics[width=0.32\linewidth]{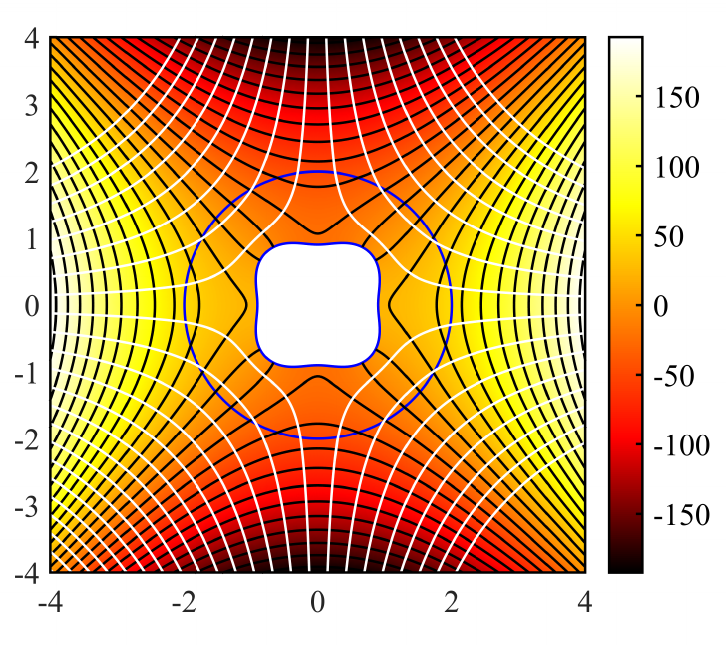}}
	\subfigure[]{
	\includegraphics[width=0.32\linewidth]{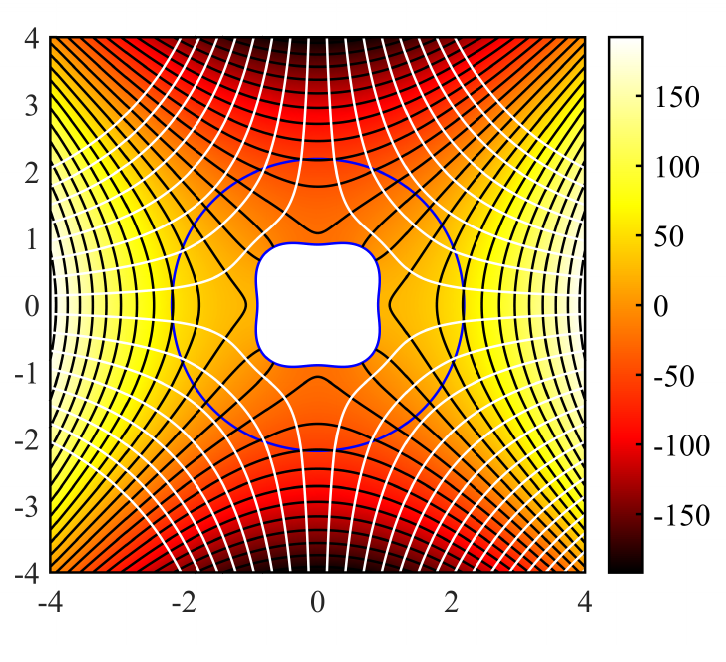}}\\
	\subfigure[]{
		\includegraphics[width=0.32\linewidth]{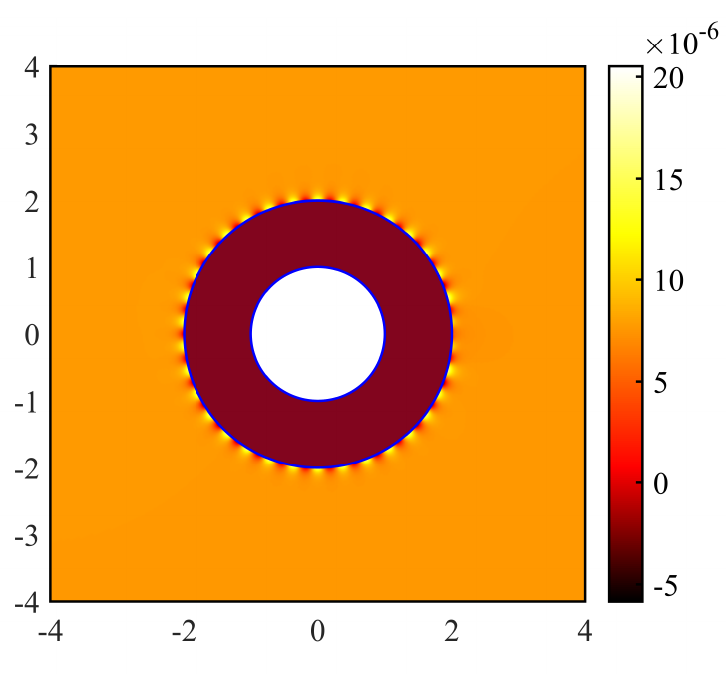}}
	\subfigure[]{
	\includegraphics[width=0.32\linewidth]{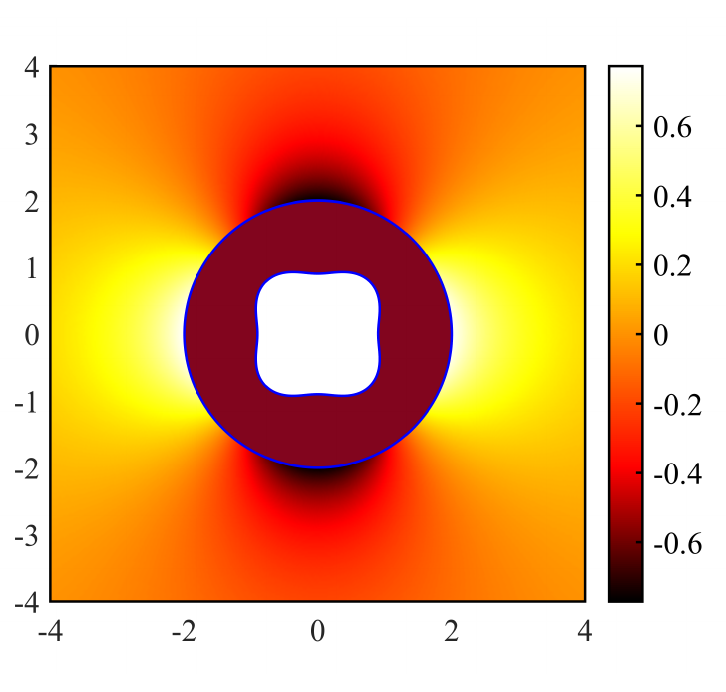}}
	\subfigure[]{
		\includegraphics[width=0.32\linewidth]{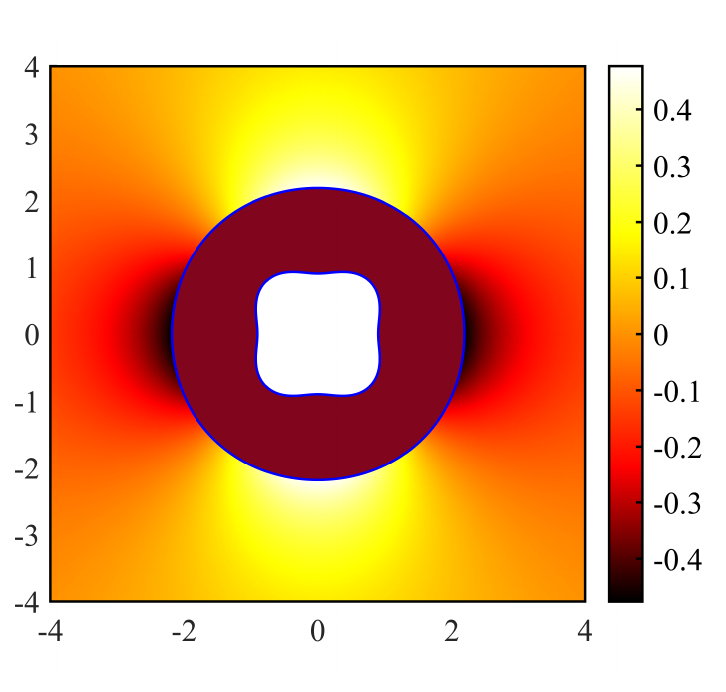}}
	\caption{Top: outer total field $p_\epsilon$; bottom: outer scattered field $p_\epsilon-P$; left: perfect cloaking; middle: 1-order near-cloaking; right: 2-order near-cloaking. From \eqref{recursive equations-d-1-cos} and \eqref{recursive equations-h-1-cos}, the Fourier coefficients $d_m$ of $g$ are obtained as $d_0=1.9922$, $d_4=-0.125$, where $n=2$.}\label{fig-near-cloaking-n-2}
\end{figure}

Extending our analysis to perturbed confocal ellipses, we next consider the case of $\p D$ and $\p \Omega$ being confocal ellipses of elliptic radii $\xi_i=0.5$ and $\xi_e=1$ with $\zeta_0$ satisfying perfect cloaking. From the theory in Section \ref{subsec-ellipse}, it follows that weak 2-order near-cloaking can be achieved. Before showing the numerical results, we compute the Fourier coefficients $c_{a,2m}$ of $\gamma_a^{-1}$, as shown in Table \ref{tab-c-fourier-seriers}. Observing the table, we can find the coefficient of the leading term is greater than that of the other term. This indicates that the method of the leading term-vanishing is reasonable. The following numerical results further demonstrate the method.

\begin{table}[!htbp]
  \caption{Fourier coefficients $c_{a,2m}$ of $\gamma_a^{-1}$ for $a=i, e$, where $\xi_i=0.5$, $\xi_e=1$.}\label{tab-c-fourier-seriers}
  \centering
  \begin{tabular}{ccccccc}
    \toprule
    m   &0   &1 &2 &3 &4 &5 \\
    \midrule
    $c_{i,2m}$& 1.257556&  0.471036& 0.130758& 0.040210&0.012967 & 0.004299\\
    $c_{e,2m}$& 0.739163& 0.100266& 0.010185& 0.001149&0.000136&0.000017 \\
    \bottomrule
  \end{tabular}
\end{table}

Figure \ref{fig-near-cloaking-n-1-ellipse} presents a comparison of finite-element simulation
results under a linear background field, i.e., n=1. Figures \ref{fig-near-cloaking-n-1-ellipse}(a)--\ref{fig-near-cloaking-n-1-ellipse}(c) present the resulting pressure distribution (colormap) and streamlines (white
lines), showing a good cloaking. Comparing the scattered field in Figure \ref{fig-near-cloaking-n-1-ellipse}(e) and Figure \ref{fig-near-cloaking-n-1-ellipse}(f), we can find that the magnitude of the scattered field has decreased drastically. This indicates the enhanced cloaking effect is achieved. Table \ref{tab-Q-ellipse} presents the evaluation function $Q$ for different cloaking, clearly showing that weak $2$-order near-cloaking has smaller scattering. In addition, we also compare the scattered field on the circle of radius $3$, as shown in Figure \ref{fig-near-cloaking-ellipse-r-3}(a), showing that the scattering from weak $2$-order near-cloaking is smaller. The case of a non-linear background field, i.e., $n=2$, is shown in Figure \ref{fig-near-cloaking-n-2-ellipse}, Table \ref{tab-Q-ellipse} and \ref{fig-near-cloaking-ellipse-r-3}(b). Analogously, the enhanced cloaking effect is also achieved.
These results present appreciable improvement that can be realized by controlling the relation of the shape functions at the inner and outer boundaries. Moreover, they also show excellent agreement like the perturbed circular cylinder case, and validate Theorem \ref{thm-near-cloaking-ellipse}. Hence the performance of the proposed enhanced near-cloaking conditions has been numerically confirmed.
 \begin{figure}[!htbp]
	\centering  %图片全局居中
	\subfigbottomskip=-10pt %两行子图之间的行间距
	\subfigcapskip=-10pt %设置子图与子标题之间的距离
	\subfigure[]{
		\includegraphics[width=0.32\linewidth]{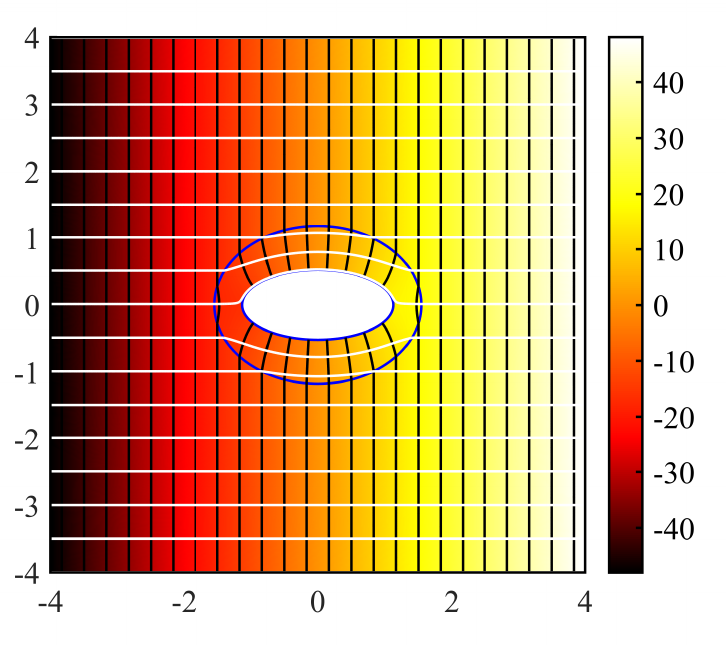}}
	\subfigure[]{
		\includegraphics[width=0.32\linewidth]{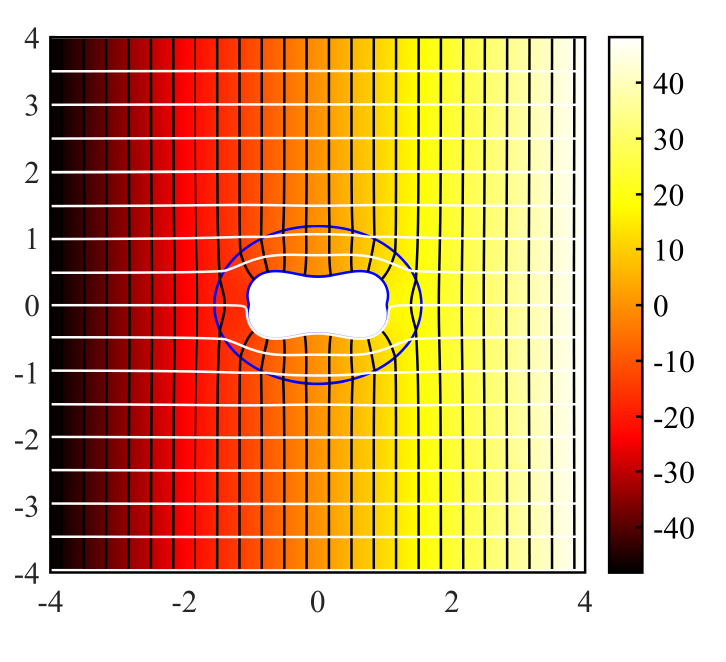}}
	\subfigure[]{
	\includegraphics[width=0.32\linewidth]{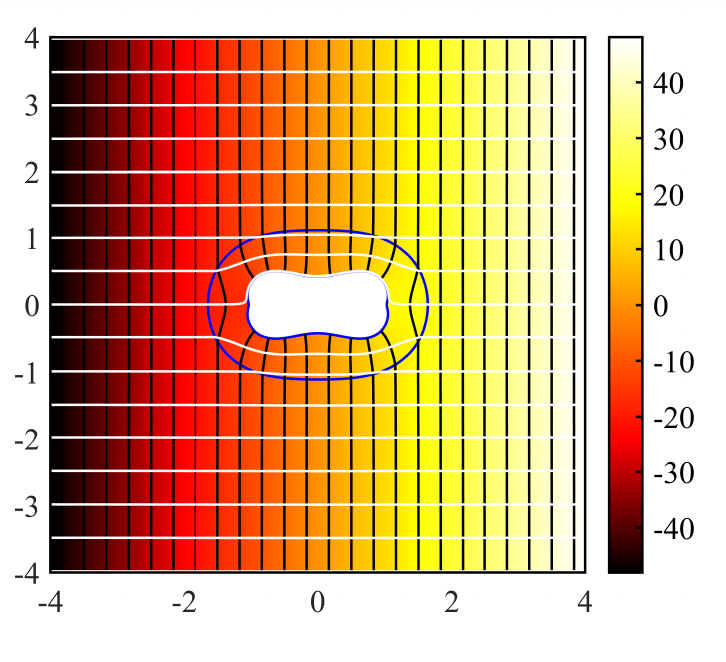}}\\
	\subfigure[]{
		\includegraphics[width=0.32\linewidth]{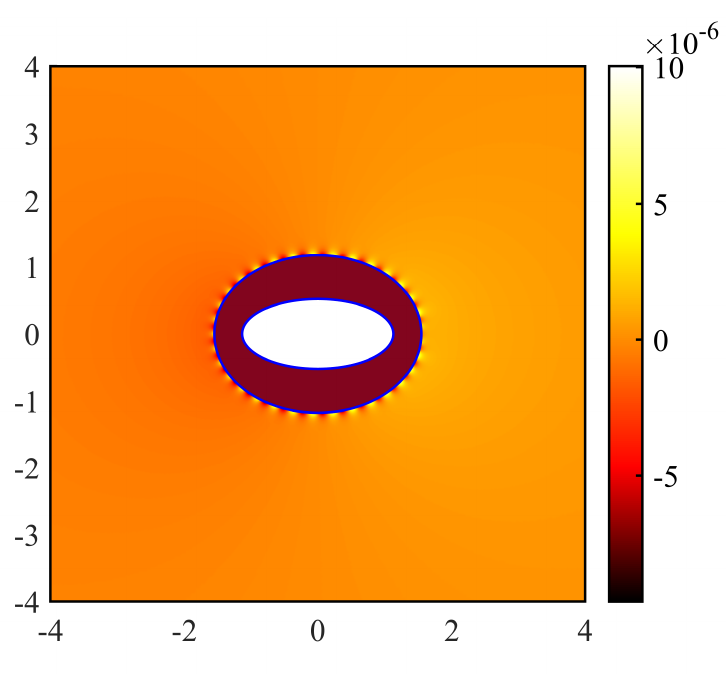}}
	\subfigure[]{
	\includegraphics[width=0.32\linewidth]{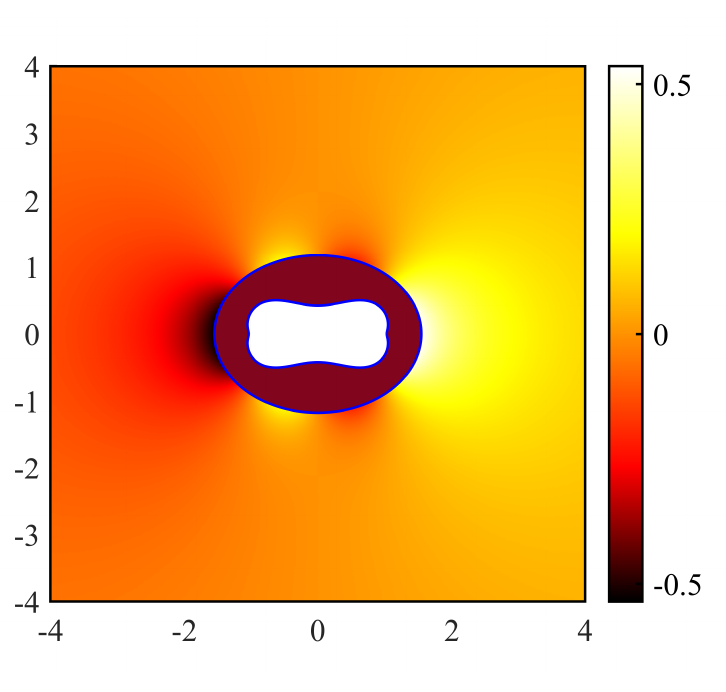}}
	\subfigure[]{
		\includegraphics[width=0.32\linewidth]{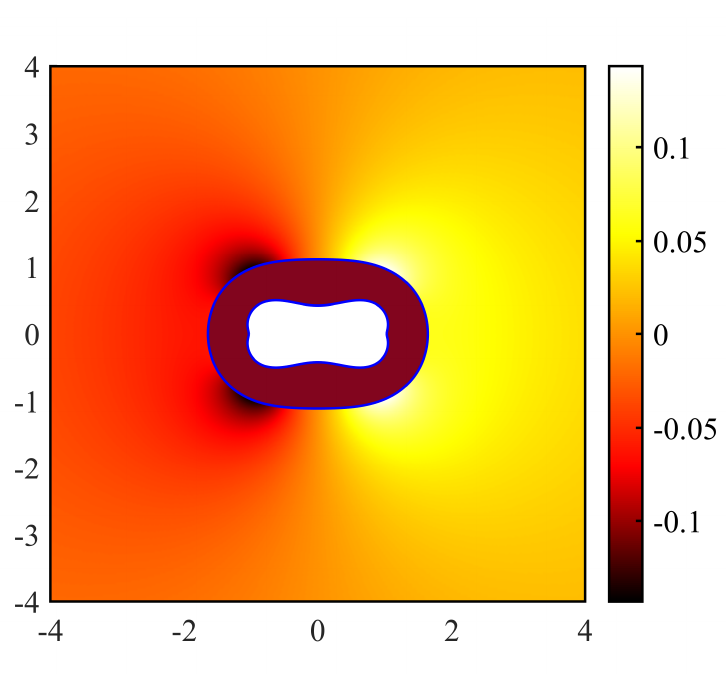}}
	\caption{Top: outer total field $p_\epsilon$; bottom: outer scattered field $p_\epsilon-P$; left: perfect cloaking; middle: 1-order near-cloaking; right: weak 2-order near-cloaking. From \eqref{recursive equations-d-1-cos-ellipse} and \eqref{recursive equations-h-1-cos-ellipse}, the Fourier coefficients $d_m$ of $g$ are obtained as $d_0=0.5141$, $d_2=0.7933$, $d_4=-0.3458$, where $n=1$.}\label{fig-near-cloaking-n-1-ellipse}
\end{figure}

\begin{table}[H]
  \caption{Evaluation function $Q$ with different cloaking and $n$}\label{tab-Q-ellipse}
  \centering
  \begin{tabular}{cccc}
    \toprule
    n   & perfect cloaking   &1-order near-cloaking & weak 2-order near-cloaking \\
    \midrule
    1 & 0&  0.969221 & 0.313614\\
    2 & 0& 1.634962& 1.213713\\
    \bottomrule
  \end{tabular}
\end{table}

 \begin{figure}[!htbp]
	\centering  %图片全局居中
	\subfigbottomskip=-10pt %两行子图之间的行间距
	\subfigcapskip=0pt %设置子图与子标题之间的距离
	\subfigure[]{
		\includegraphics[width=0.45\linewidth]{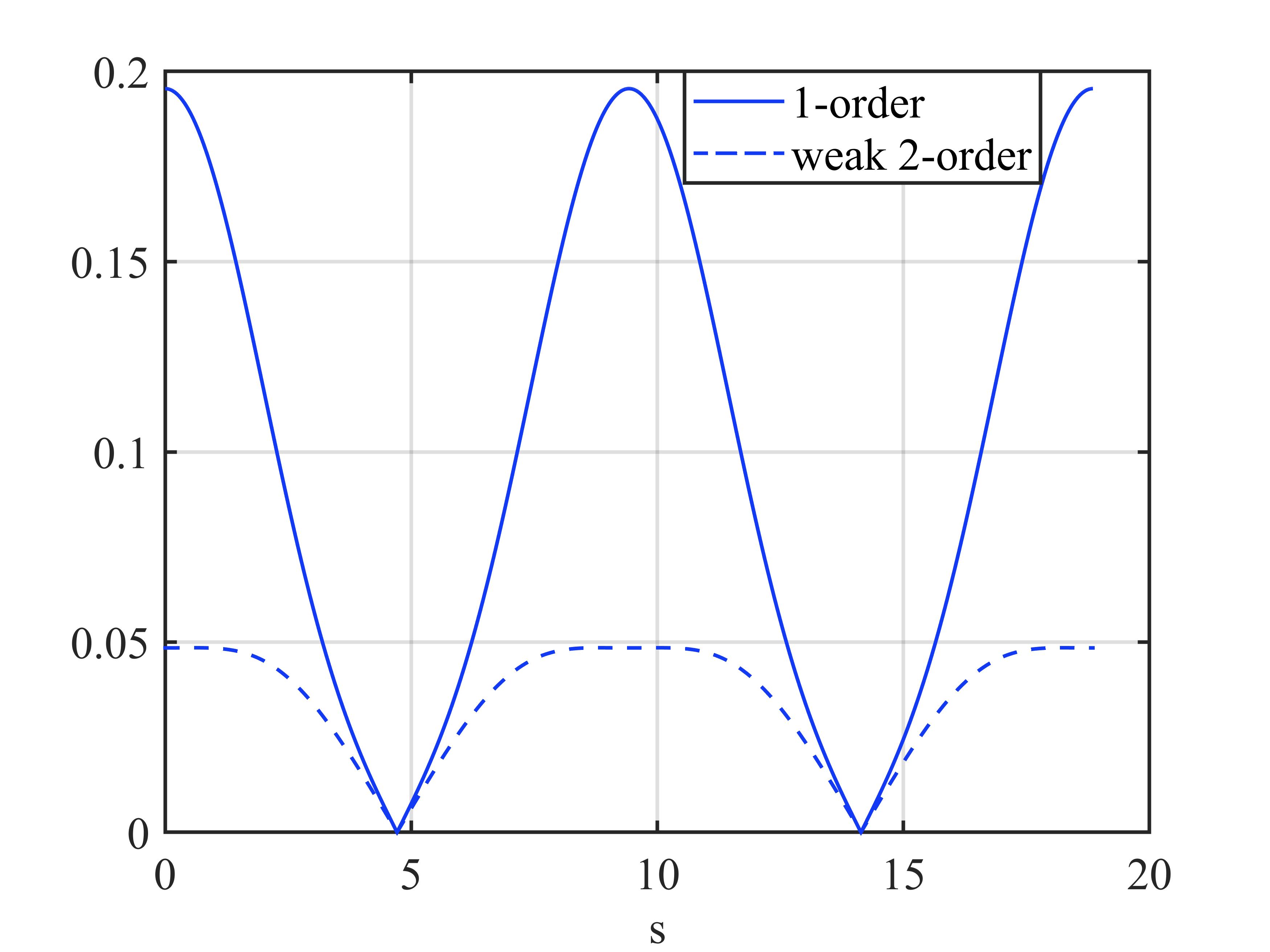}}
	\subfigure[]{
		\includegraphics[width=0.45\linewidth]{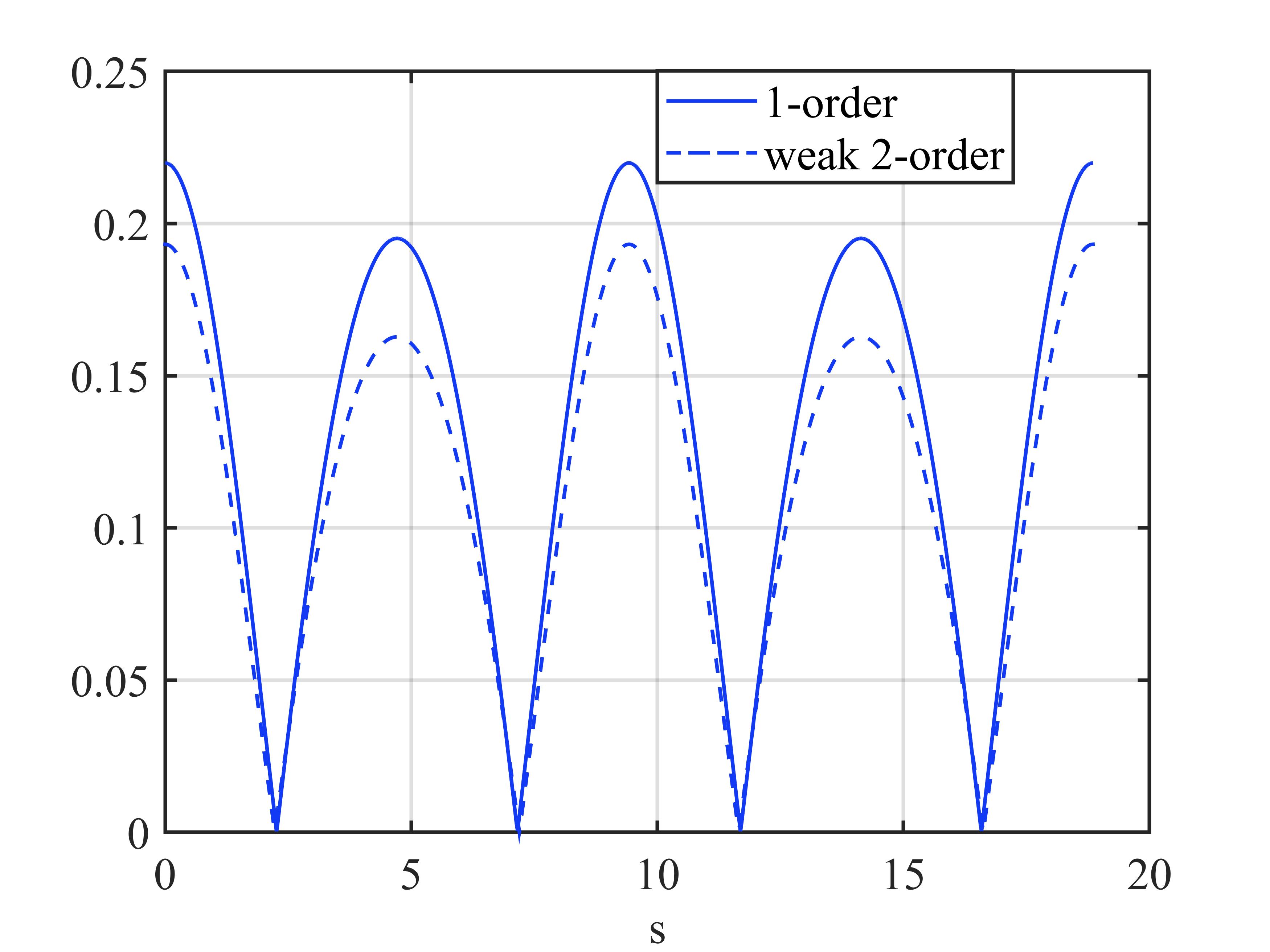}}
	\caption{Absolute value of outer scattered field on the circle of radius 3; left: background field $P=12r\cos(\theta)$; right: background field $P=12r^2\cos(2\theta)$. Here $s$ denotes arc length.}\label{fig-near-cloaking-ellipse-r-3}
\end{figure}

 \begin{figure}[H]
	\centering  %图片全局居中
	\subfigbottomskip=-10pt %两行子图之间的行间距
	\subfigcapskip=-10pt %设置子图与子标题之间的距离
	\subfigure[]{
		\includegraphics[width=0.32\linewidth]{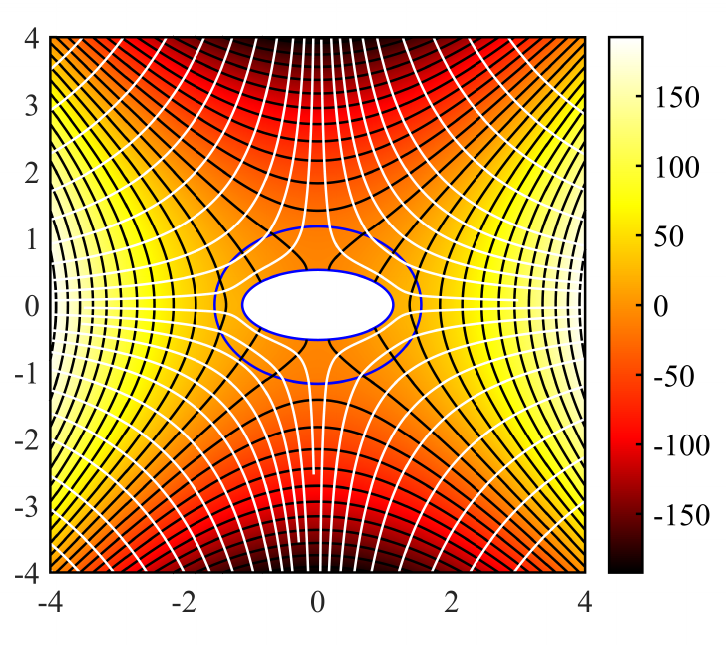}}
	\subfigure[]{
		\includegraphics[width=0.32\linewidth]{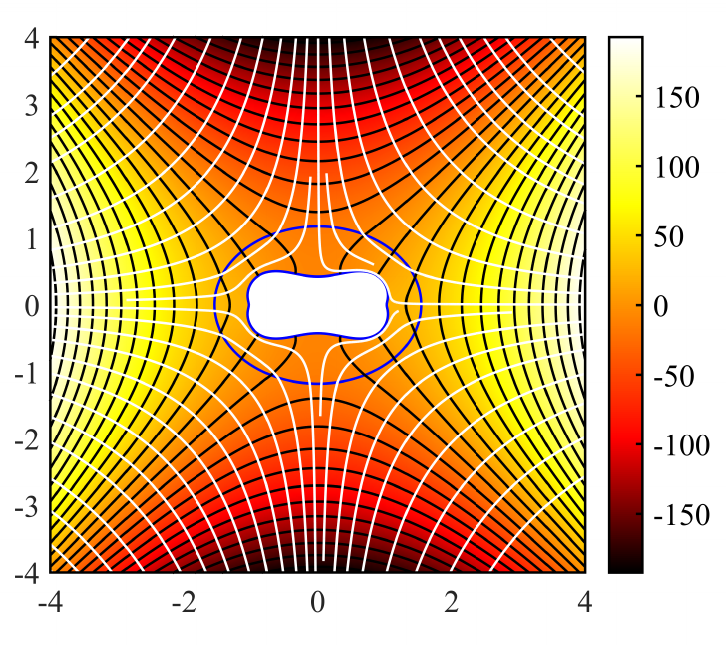}}
	\subfigure[]{
	\includegraphics[width=0.32\linewidth]{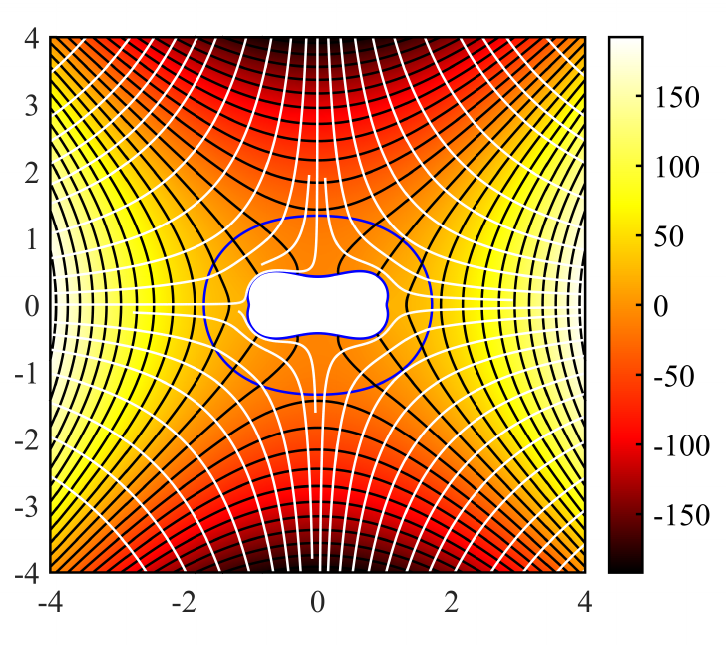}}\\
	\subfigure[]{
		\includegraphics[width=0.32\linewidth]{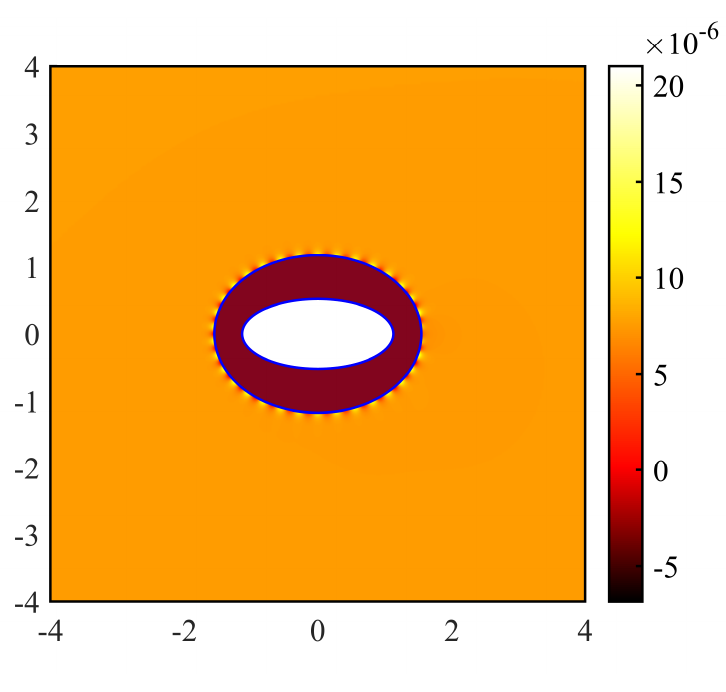}}
	\subfigure[]{
	\includegraphics[width=0.32\linewidth]{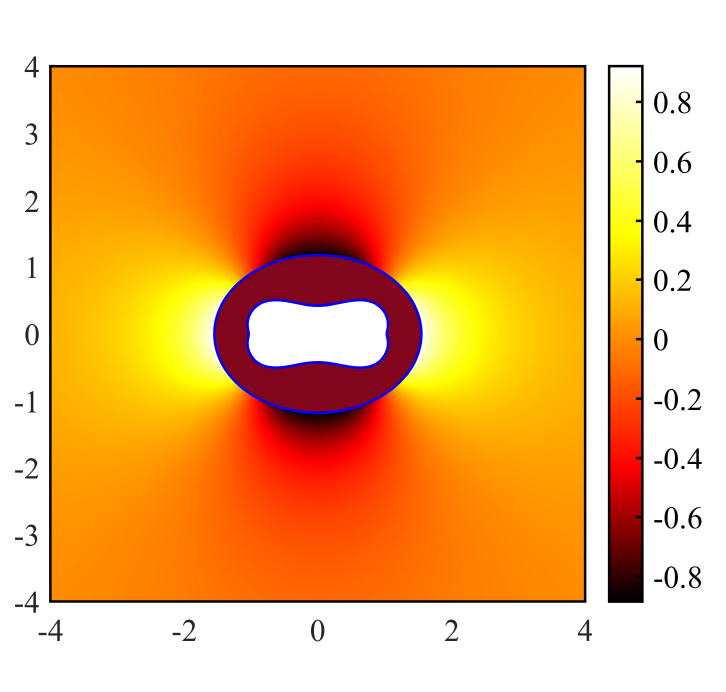}}
	\subfigure[]{
		\includegraphics[width=0.32\linewidth]{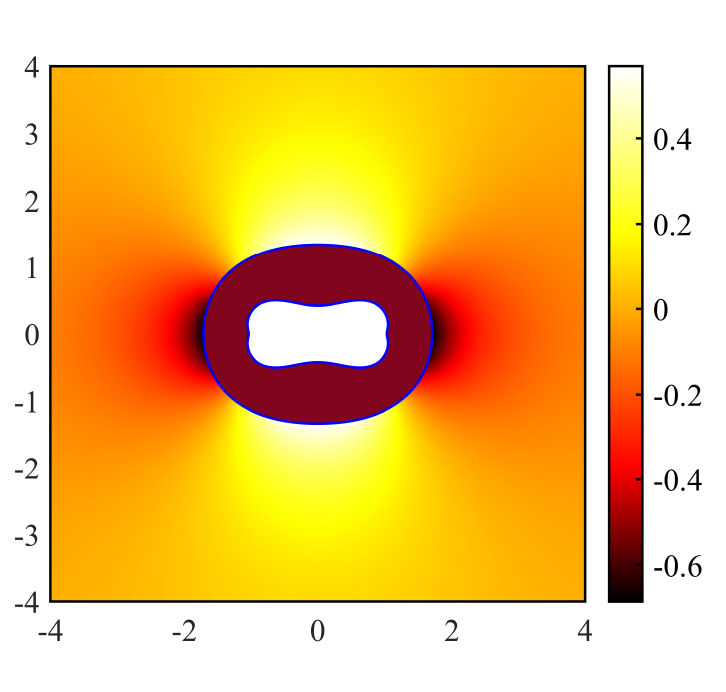}}
	\caption{Top: outer total field $p_\epsilon$; bottom: outer scattered field $p_\epsilon-P$; left: perfect cloaking; middle: 1-order near-cloaking; right: weak 2-order near-cloaking. From \eqref{recursive equations-d-1-cos-ellipse} and \eqref{recursive equations-h-1-cos-ellipse}, the Fourier coefficients $d_m$ of $g$ are obtained as $d_0=1.887$, $d_4=-0.2656$, where $n=2$.}\label{fig-near-cloaking-n-2-ellipse}
\end{figure}

\textcolor{blue}{We finally consider two special cases to verify Remark \ref{rem-special-circle} and to present the flexibility and ability of our proposed method for a larger perturbation, which aims at describing the extent of shape perturbation. Here we also choose $r_i=1$, $r_e=2$, but $\epsilon =0.5$ which is relatively large.} One case is the shape function is a constant, i.e., $f=\frac{a_0}{2}$ where we set $a_0=-1$. Then $\epsilon f = -0.25$ and the inner circle is compressed to a smaller circle of radius $0.75$, which leads to a slightly larger disturbance occurring at the inner boundary. Figure \ref{fig-near-cloaking-circle-m-0-n-1} presents a process of the change for three different cloaking. The perfect cloaking is destroyed due to the perturbation of the inner boundary. There is some scattering, as shown in Figures \ref{fig-near-cloaking-circle-m-0-n-1}(b) and \ref{fig-near-cloaking-circle-m-0-n-1}(e). However, a new perfect cloaking occurs when the outer boundary is also compressed according to the recursive formulas, as shown in Figures \ref{fig-near-cloaking-circle-m-0-n-1}(c) and \ref{fig-near-cloaking-circle-m-0-n-1}(f). The nonlinear background field is considered in Figure \ref{fig-near-cloaking-circle-m-0-n-2}. The other case is the shape function is linear, i.e., $f=a_1\cos(\theta)$. It is special since the shape functions of the inner and outer boundaries are the same. In Figure \ref{fig-near-cloaking-circle-m-1-n-1} we can see the shape of the inner circle is changed and the location is shifted left due to the inner boundary perturbation, which leads to the structure being changed to an eccentric ring from a concentric annulus. Comparing Figure \ref{fig-near-cloaking-circle-m-1-n-1}(e) and Figure \ref{fig-near-cloaking-circle-m-1-n-1}(f), we can find that the scattering due to the perturbation of the inner boundary is reduced when the outer boundary also shifts left such that the structure is concentric.  More specifically, we compute the evaluation function and compare the outer scattered field on the circle of radius $3$, as shown in Table \ref{tab-Q-eccentric-circle} and Figure \ref{fig-near-cloaking-ellipse-m-1-r-3}(a), clearly showing that the scattering from weak $2$-order near-cloaking is smaller. \textcolor{blue}{It is worth noting that the values of $Q$ in Table \ref{tab-Q-eccentric-circle} are larger. However, it is reasonable since the perturbation $\epsilon$ is relatively large.}
Hence a good near-cloaking should keep the structure concentric. The nonlinear background field is considered in Figure \ref{fig-near-cloaking-circle-m-1-n-2}, Table \ref{tab-Q-eccentric-circle}, and Figure \ref{fig-near-cloaking-ellipse-m-1-r-3}(b). A similar enhanced cloaking effect is also achieved. In summary, the performance of the proposed enhanced near-cloaking conditions
has been numerically confirmed.
\begin{figure}[htbp]
	\centering  %图片全局居中
	\subfigbottomskip=-10pt %两行子图之间的行间距
	\subfigcapskip=-10pt %设置子图与子标题之间的距离
	\subfigure[]{
		\includegraphics[width=0.3\linewidth]{total-field-pc-n-1-circle.pdf}}
	\subfigure[]{
		\includegraphics[width=0.3\linewidth]{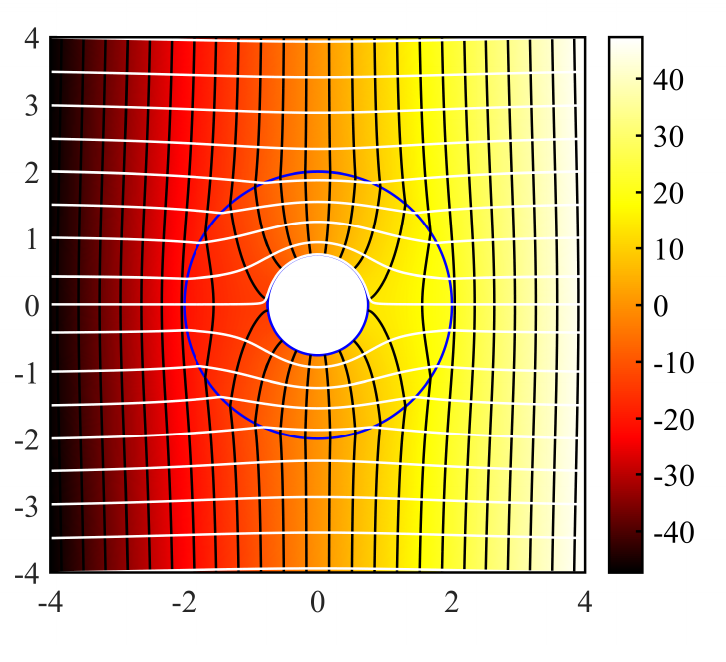}}
	\subfigure[]{
	\includegraphics[width=0.3\linewidth]{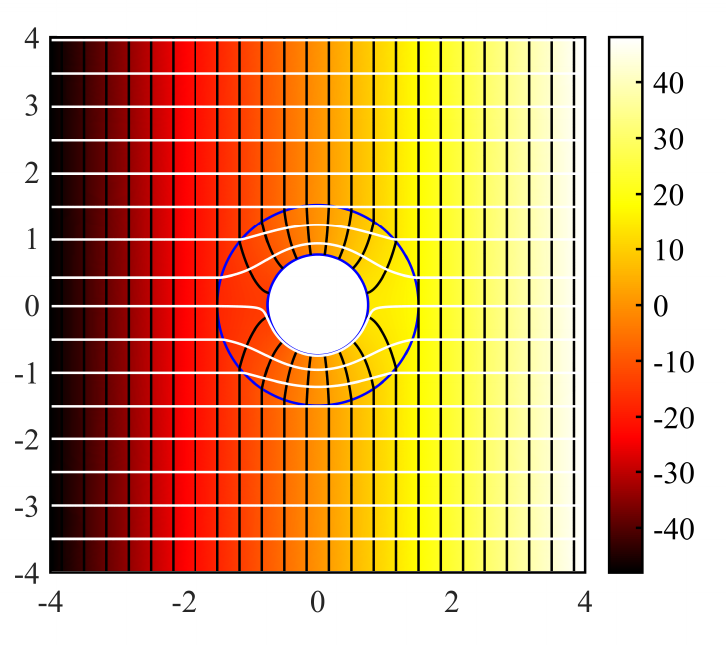}}\\
	\subfigure[]{
		\includegraphics[width=0.3\linewidth]{scattered-field-pc-n-1-circle.pdf}}
	\subfigure[]{
	\includegraphics[width=0.3\linewidth]{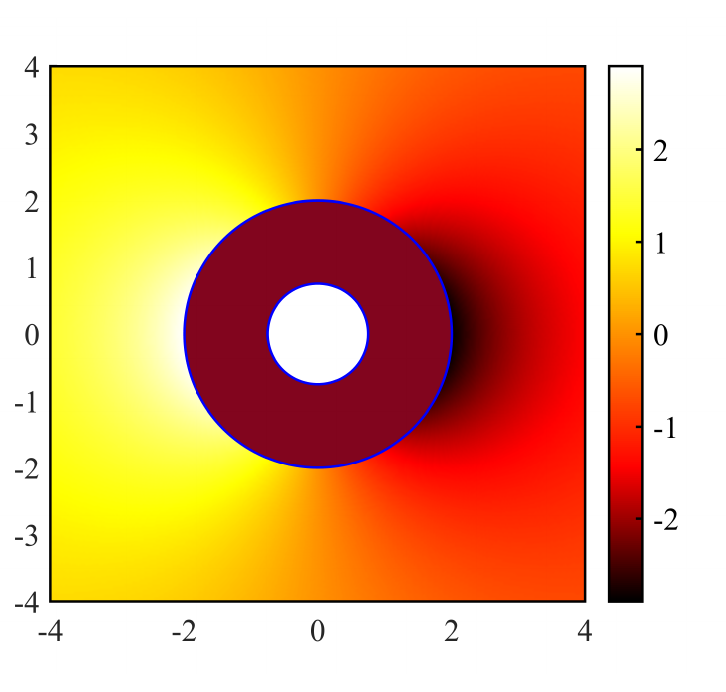}}
	\subfigure[]{
		\includegraphics[width=0.3\linewidth]{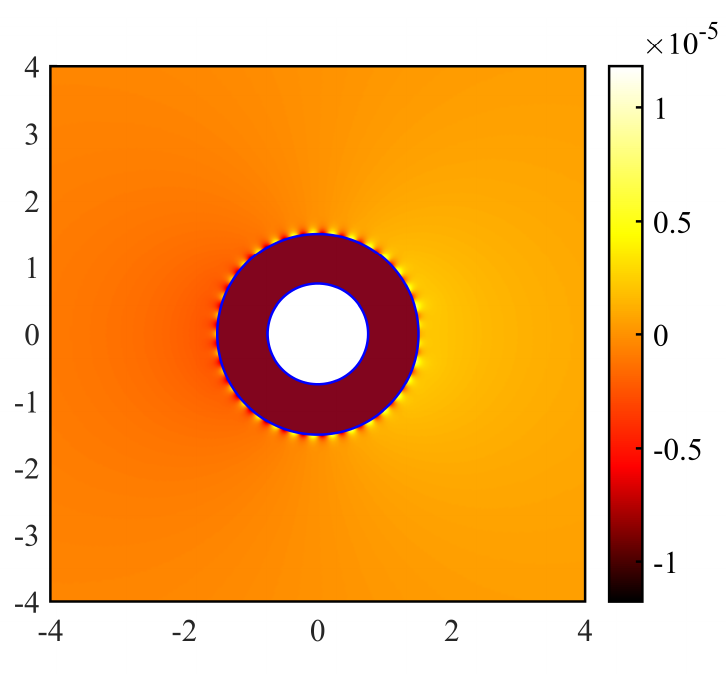}}
	\caption{Top: outer total field $p_\epsilon$; bottom: outer scattered field $p_\epsilon-P$; left: perfect cloaking; middle: 1-order near-cloaking; right: 2-order near-cloaking. From \eqref{recursive equations-d-1-cos} and \eqref{recursive equations-h-1-cos}, the Fourier coefficient $d_m$ of $g$ is obtained as $d_0=-2$, where $n=1$.}\label{fig-near-cloaking-circle-m-0-n-1}
\end{figure}

\begin{figure}[htbp]
	\centering  %图片全局居中
	\subfigbottomskip=-10pt %两行子图之间的行间距
	\subfigcapskip=-10pt %设置子图与子标题之间的距离
	\subfigure[]{
		\includegraphics[width=0.3\linewidth]{total-field-pc-n-2-circle.pdf}}
	\subfigure[]{
		\includegraphics[width=0.3\linewidth]{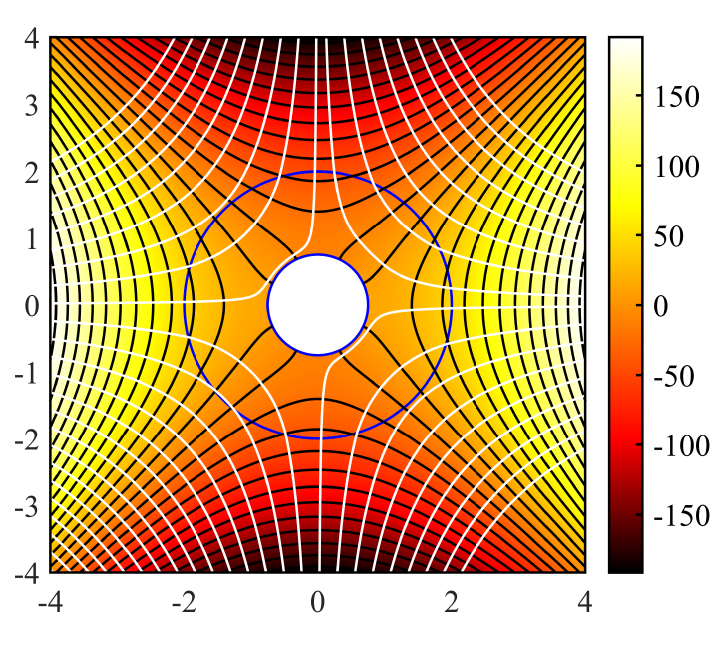}}
	\subfigure[]{
	\includegraphics[width=0.3\linewidth]{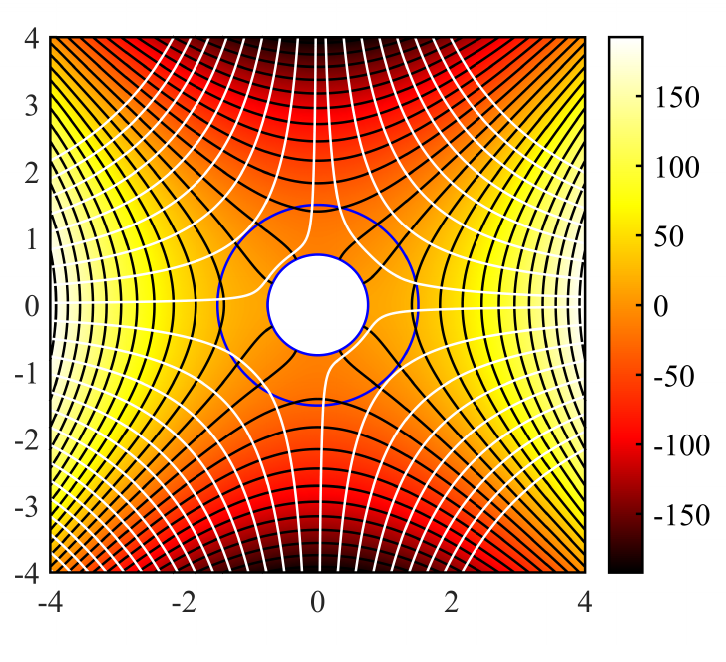}}\\
	\subfigure[]{
		\includegraphics[width=0.3\linewidth]{scattered-field-pc-n-2-circle.pdf}}
	\subfigure[]{
	\includegraphics[width=0.3\linewidth]{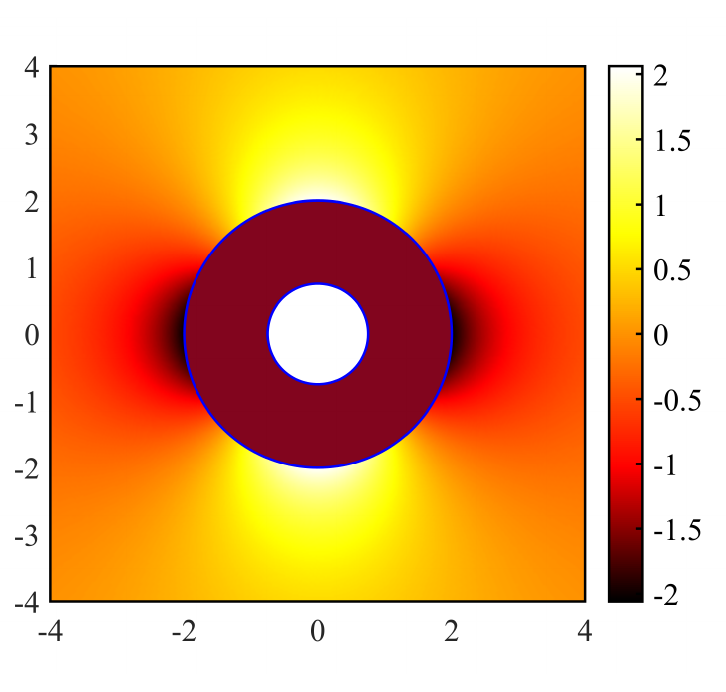}}
	\subfigure[]{
		\includegraphics[width=0.3\linewidth]{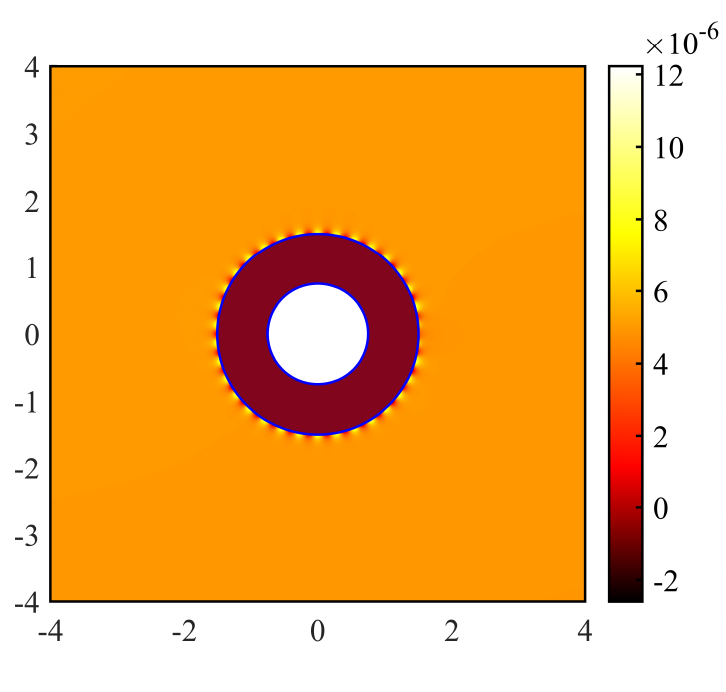}}
	\caption{Top: outer total field $p_\epsilon$; bottom: outer scattered field $p_\epsilon-P$; left: perfect cloaking; middle: 1-order near-cloaking; right: 2-order near-cloaking. From \eqref{recursive equations-d-1-cos} and \eqref{recursive equations-h-1-cos}, the Fourier coefficient $d_m$ of $g$ is obtained as $d_0=-2$, where $n=2$.}\label{fig-near-cloaking-circle-m-0-n-2}
\end{figure}

% \begin{figure}[!htbp]
%	\centering  %图片全局居中
%	\subfigbottomskip=-2pt %两行子图之间的行间距
%	\subfigcapskip=-10pt %设置子图与子标题之间的距离
%	\subfigure[]{
%		\includegraphics[width=0.32\linewidth]{total-field-pc-n-2-circle.pdf}}
%	\subfigure[]{
%		\includegraphics[width=0.32\linewidth]{total-field-nc-1st-n-2-m-0-circle.pdf}}
%	\subfigure[]{
%	\includegraphics[width=0.32\linewidth]{total-field-nc-2ed-n-2-m-0-circle.pdf}}\\
%	\subfigure[]{
%		\includegraphics[width=0.32\linewidth]{scattered-field-pc-n-2-circle.pdf}}
%	\subfigure[]{
%	\includegraphics[width=0.32\linewidth]{scattered-field-nc-1st-n-2-m-0-circle.pdf}}
%	\subfigure[]{
%		\includegraphics[width=0.32\linewidth]{scattered-field-nc-2ed-n-2-m-0-circle.pdf}}
%	\caption{Top: outer total field; bottom: outer scattered field; left: perfect cloaking; middle: first-order near-cloaking; right: second-order near-cloaking. From \eqref{recursive equations-d-1-cos} and \eqref{recursive equations-h-1-cos}, the Fourier coefficients $d_m$ of $g$ is obtained as $d_0=-10$, where $n=2$.}\label{fig-near-cloaking-circle-m-0-n-2}
%\end{figure}

 \begin{figure}[H]
	\centering  %图片全局居中
	\subfigbottomskip=-10pt %两行子图之间的行间距
	\subfigcapskip=-10pt %设置子图与子标题之间的距离
	\subfigure[]{
		\includegraphics[width=0.32\linewidth]{total-field-pc-n-1-circle.pdf}}
	\subfigure[]{
		\includegraphics[width=0.32\linewidth]{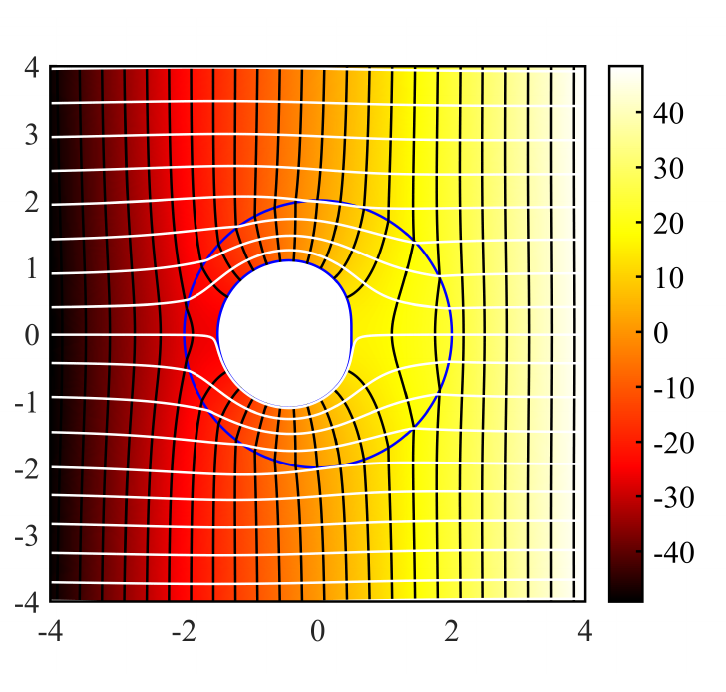}}
	\subfigure[]{
	\includegraphics[width=0.32\linewidth]{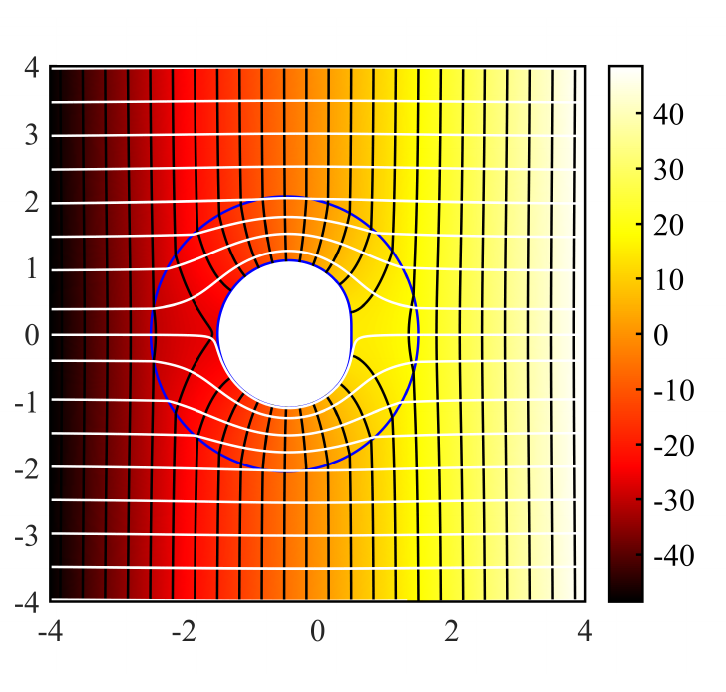}}\\
	\subfigure[]{
		\includegraphics[width=0.32\linewidth]{scattered-field-pc-n-1-circle.pdf}}
	\subfigure[]{
	\includegraphics[width=0.32\linewidth]{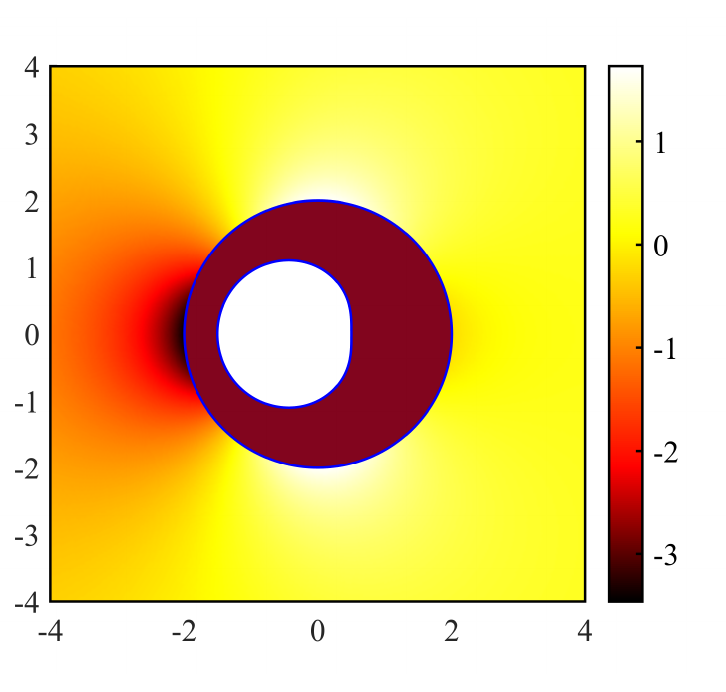}}
	\subfigure[]{
		\includegraphics[width=0.32\linewidth]{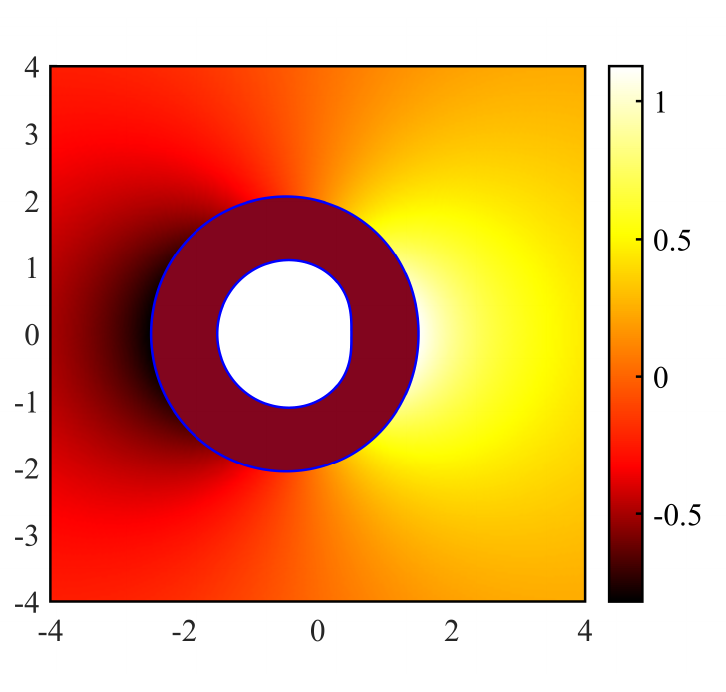}}
	\caption{Top: outer total field $p_\epsilon$; bottom: outer scattered field $p_\epsilon-P$; left: perfect cloaking; middle: 1-order near-cloaking; right: 2-order near-cloaking. From \eqref{recursive equations-d-1-cos} and \eqref{recursive equations-h-1-cos}, the Fourier coefficient $d_m$ of $g$ is obtained as $d_1=-1$, where $a_1=-1$,  $n=1$. }\label{fig-near-cloaking-circle-m-1-n-1}
\end{figure}

 \begin{figure}[H]
	\centering  %图片全局居中
	\subfigbottomskip=-10pt %两行子图之间的行间距
	\subfigcapskip=0pt %设置子图与子标题之间的距离
	\subfigure[]{
		\includegraphics[width=0.45\linewidth]{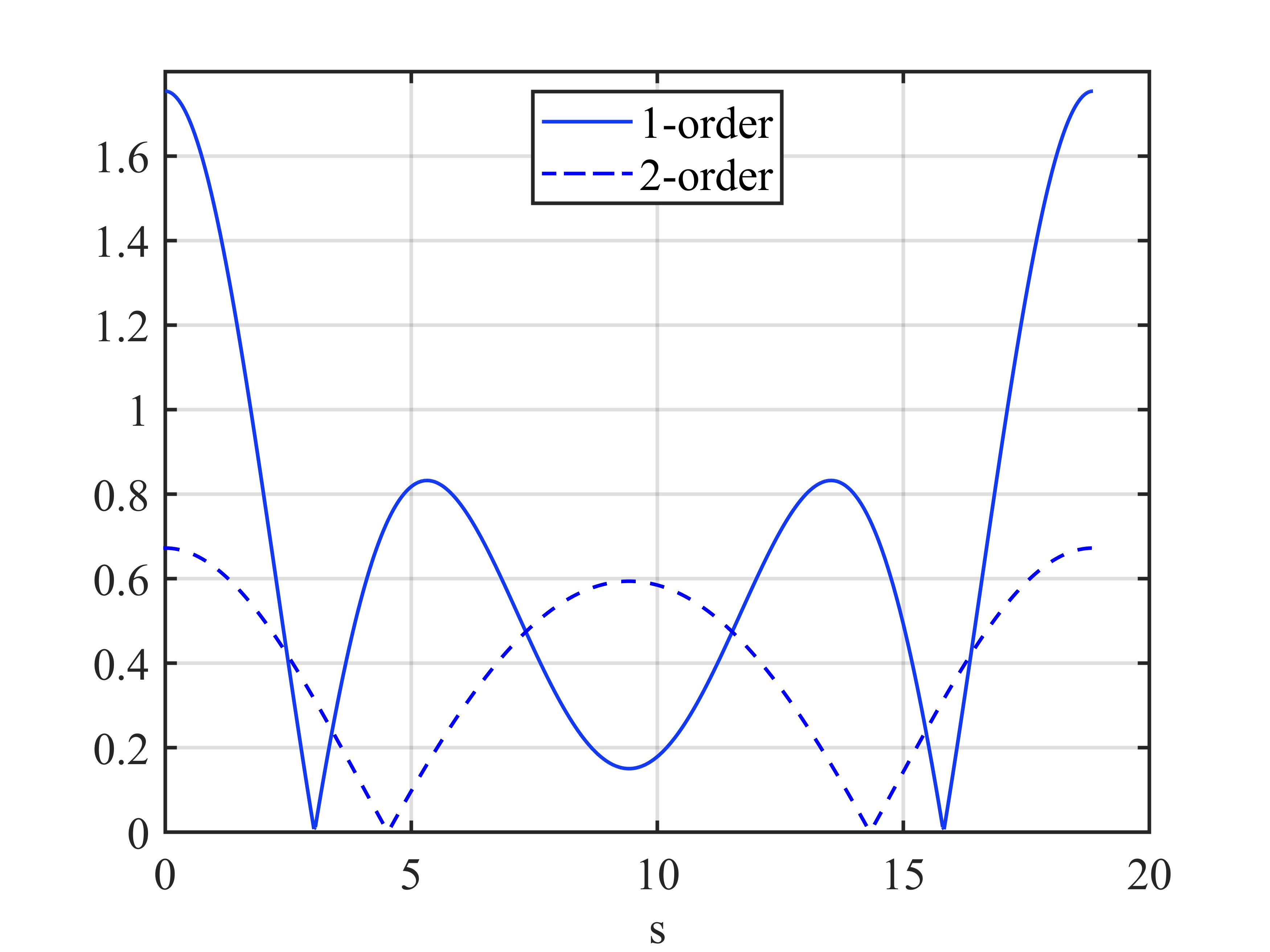}}
	\subfigure[]{
		\includegraphics[width=0.45\linewidth]{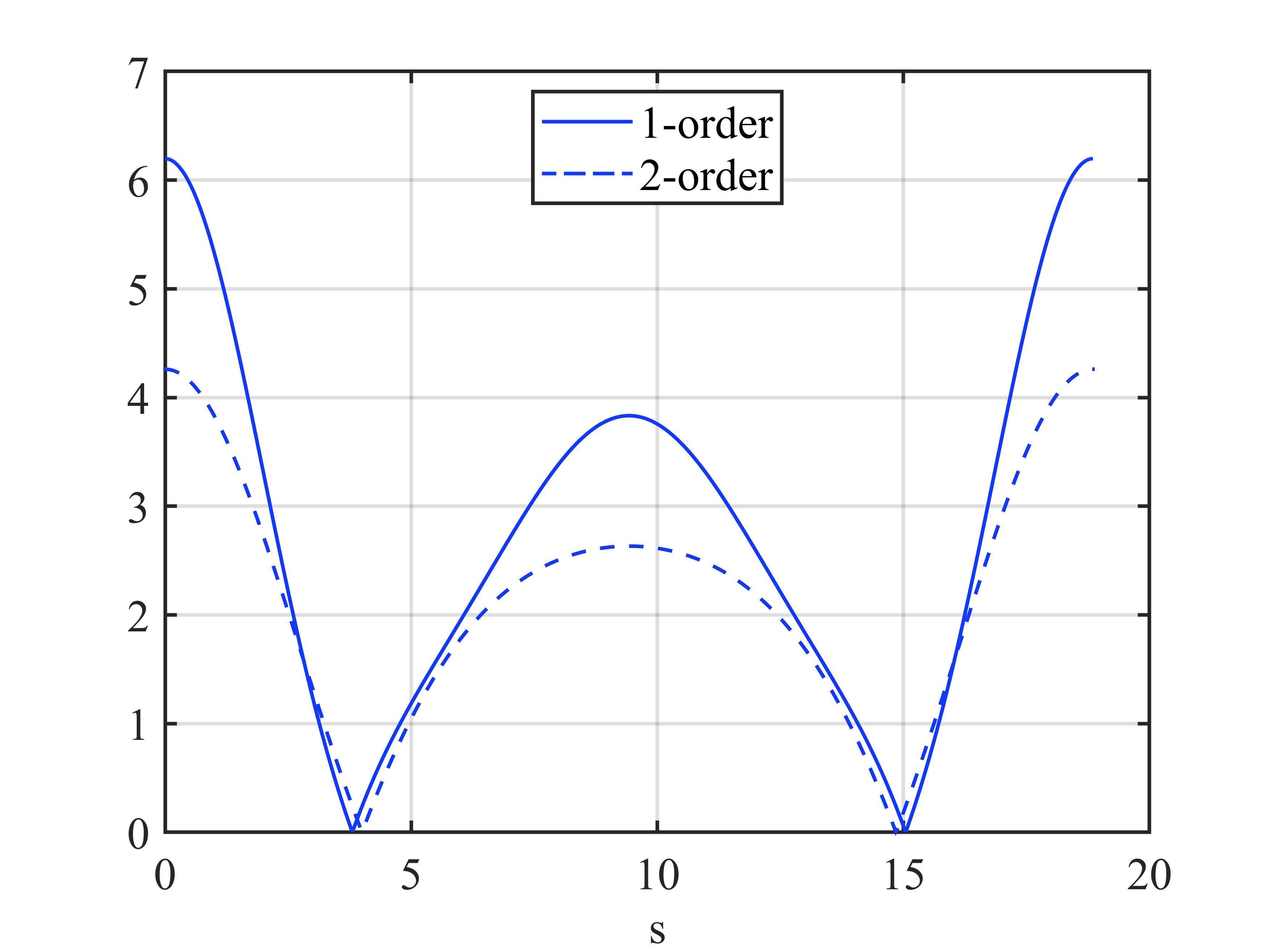}}
	\caption{Absolute value of outer scattered field on the circle of radius 3; left: background field $P=12r\cos(\theta)$; right: background field $P=12r^2\cos(2\theta)$. Here $s$ denotes arc length.}\label{fig-near-cloaking-ellipse-m-1-r-3}
\end{figure}

\begin{table}[!htbp]
  \caption{Evaluation function $Q$ with different cloaking and $n$}\label{tab-Q-eccentric-circle}
  \centering
  \begin{tabular}{cccc}
    \toprule
    n   & perfect cloaking   &1-order near-cloaking & weak 2-order near-cloaking \\
    \midrule
    1 & 0&  5.645441 & 3.008405\\
    2 & 0& 21.711518& 15.596474\\
    \bottomrule
  \end{tabular}
\end{table}

 \begin{figure}[H]
	\centering  %图片全局居中
	\subfigbottomskip=-10pt %两行子图之间的行间距
	\subfigcapskip=-10pt %设置子图与子标题之间的距离
	\subfigure[]{
		\includegraphics[width=0.32\linewidth]{total-field-pc-n-2-circle.pdf}}
	\subfigure[]{
		\includegraphics[width=0.32\linewidth]{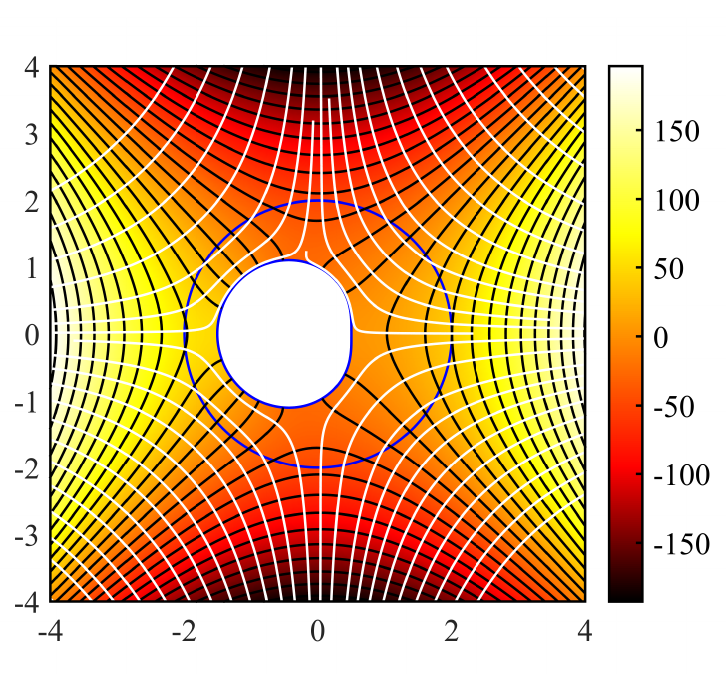}}
	\subfigure[]{
	\includegraphics[width=0.32\linewidth]{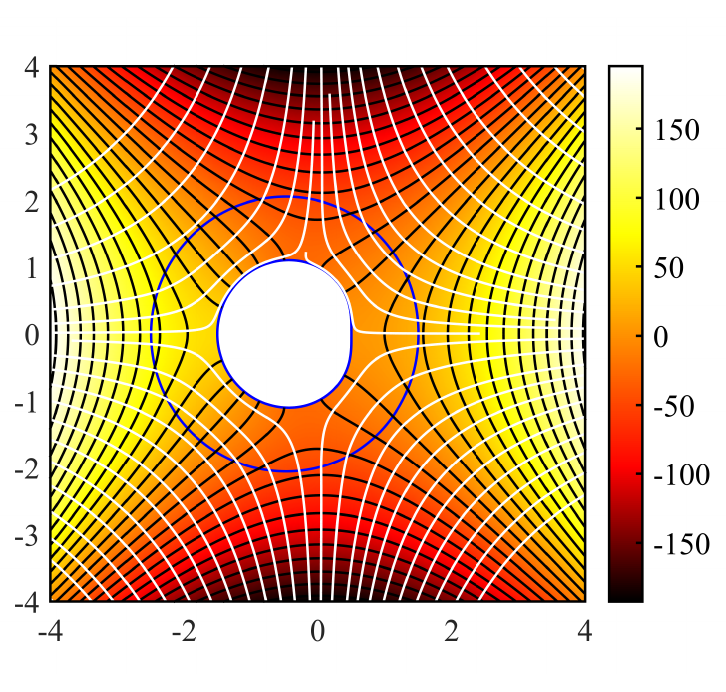}}\\
	\subfigure[]{
		\includegraphics[width=0.32\linewidth]{scattered-field-pc-n-2-circle.pdf}}
	\subfigure[]{
	\includegraphics[width=0.32\linewidth]{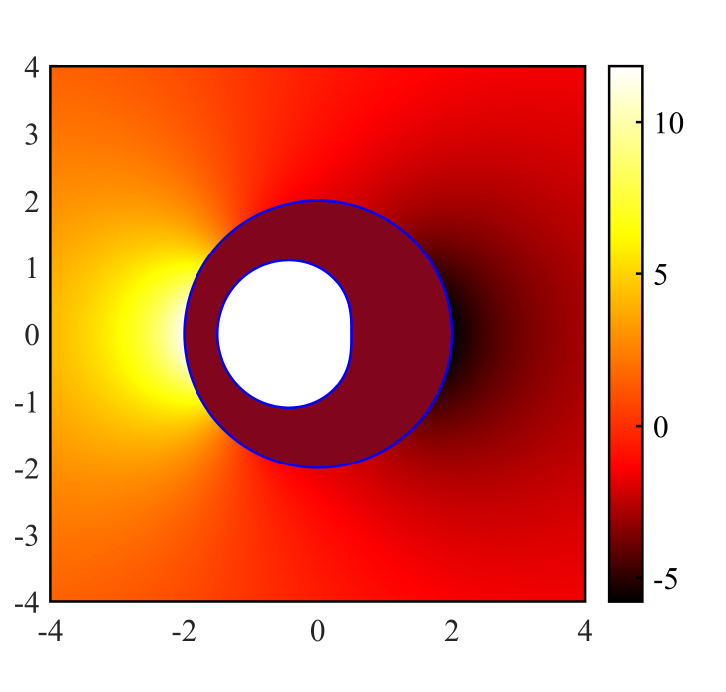}}
	\subfigure[]{
		\includegraphics[width=0.32\linewidth]{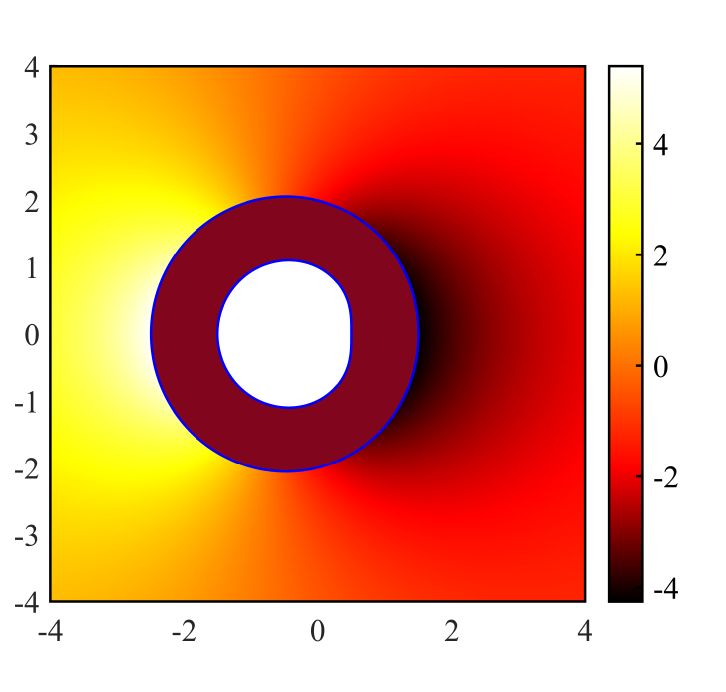}}
	\caption{Top: outer total field $p_\epsilon$; bottom: outer scattered field $p_\epsilon-P$; left: perfect cloaking; middle: 1-order near-cloaking; right: 2-order near-cloaking. From \eqref{recursive equations-d-1-cos} and \eqref{recursive equations-h-1-cos}, the Fourier coefficient $d_m$ of $g$ is obtained as $d_1=-1$, where $a_1=-1$, $n=2$.  }\label{fig-near-cloaking-circle-m-1-n-2}
\end{figure}

\section{Conclusions}\label{sec-conclusion}
In this paper, we presented a new method for the design of a near-cloaking structure that enhanced the invisibility effect based on the perfect hydrodynamic cloaking using the boundary perturbation theory.
We established a complete mathematical framework that allows us to compute the enhanced near-cloaking conditions for complex geometries and achieved an enhanced cloaking effect for this complex object inside the cloaked region with approximately zero scattering. Such a cloaking device is obtained by simultaneously perturbing the inner and outer boundaries of the perfect cloaking structure. The cloaking effect for the electro-osmosis system is significantly enhanced by the proposed near-cloaking structures. In addition to the theoretical results, extensive numerical experiments were conducted to corroborate the theoretical findings. Finally, we would like to emphasize that the proposed near-cloaking structures are metamaterial-less, which eliminates the dependence on complex metamaterial structures.

\section*{Acknowledgement}
The research of H. Liu was supported by NSFC/RGC Joint Research Scheme, N CityU101/21, ANR/RGC
Joint Research Scheme, A-CityU203/19, and the Hong Kong RGC General Research Funds (projects 12302919,
12301420 and 11300821). The research of G. Zheng was supported by the NSF of China (12271151), NSF of
Hunan (2020JJ4166) and NSF Innovation Platform Open Fund project of Hunan Province (20K030).
\\
\\
\textbf{\large{Declarations}}
\\
\\
\textbf{Conflict of interest}
The authors have not disclosed any competing interests.

%\appendix
%\section{Appendix} \label{Appendix}

\begin{thebibliography}{99}
\bibitem{Abbas2017}
{\sc T. Abbas, H. Ammari, G. Hu, A. Wahab and J.C. Ye}, {\em Two-dimensional elastic scattering coefficients and enhancement of nearly elastic cloaking}, {\sl J. Elast.} 128(2), (2017), 203--243.

\bibitem{Ammari2013-1}
{\sc H. Ammari, G. Ciraolo, H. Kang, H. Lee and G. Milton},
{\em Spectral theory of a Neumann-Poincar\'{e}-type operator and analysis of cloaking due to anomalous localized resonance}, {\sl Arch. Rational Mech. Anal.} 208 (2013), 667--692.

\bibitem{Alu2005}
{\sc A. Al\`{u} and N. Engheta},
{\em Achieving transparency with plasmonic and metamaterial coatings},
{\sl Phys. Rev. E} 72 (2005), 016623.

\bibitem{Alu2007}
{\sc A. Al\`{u} and N. Engheta},
{\em Cloaking and transparency for collections of particles with metama-terial and plasmonic covers}, {\sl Opt. Express}, 15 (2007), pp. 7578--7590.

\bibitem{Ammari2012-3}
{\sc H. Ammari, J. Garnier, V. Jugnon, H. Kang, H. Lee, and M. Lim},
{\em Enhancement of near-cloaking. Part III: Numerical simulations, statistical stability, and related questions}, {\sl Contemp. Math.}, 577 (2012), 1--24.

\bibitem{Ando2016}
{\sc K. Ando and H. Kang},
{\em Analysis of plasmon resonance on smooth domains using spectral properties of the Neumann–Poincar\'{e} operator}, {\sl Jour. Math. Anal. Appl.} 435 (2016), 162--178.

\bibitem{Ammari2012-1}
{\sc H. Ammari, H. Kang, H. Lee and M. Lim},
{\em Enhancement of near cloaking using generalizedpolarization tensors vanishing structures. Part I: The conductivity problem},
{\sl Comm. Math.Phys.}, 317 (2012), 253--266.

\bibitem{Ammari2012-2}
{\sc H. Ammari, H. Kang, H. Lee and M. Lim},
{\em Enhancement  of  near-cloaking.  Part II: The Helmholtz  equation},
{\sl Comm.  Math.  Phys.},  317 (2012), 485--502.

\bibitem{Ammari2013}
{\sc H. Ammari, H. Kang, H. Lee, M. Lim and S. Yu},
{\sl Enhancement of near cloaking for the full Maxwell equations}, {\sl SIAM J. Appl. Math.} 73(6) (2013), 2055--2076.

\bibitem{M. Lim2012}
{\sc H. Ammari, H. Kang, M. Lim and H. Zribi},
{\em  The generalized polarization tensors for resolved imaging. part I: Shape reconstruction of a conductivity inclusion}, {\sl Math. Comp.}, 81 (2012) 367--386.

\bibitem{Boyko2021}
{\sc E. Boyko, V. Bacheva, M. Eigenbrod, F. Paratore, A. Gat, S. Hardt and M. Bercovici},
{\em Microscale hydrodynamic cloaking and shielding via electro-osmosis}, {\sl Phys. Rev. Lett}. 126 (2021), 184502.


\bibitem{Coifman1999}
{\sc R. Coifman, M. Goldberg, T. Hrycak, M. Israeli and V. Rokhlin}, {\em An improved operator expansion algorithm for direct and inverse scattering computations}, {\sl Waves Random Media} 1999;  9:441--457

\bibitem{Chung2014}
{\sc D. Chung, H. Kang, K. Kim and H. Lee},
{\em Cloaking due to anomalous localized resonance in plasmonic structures of confocal ellipses}, {\sl SIAM J. Appl. Math.} 74 (2014), 1691--1707.

\bibitem{Chen2012}
{\sc P. Chen, J. Soric and A. Al\`{u}},
{\em Invisibility and cloaking based on scattering cancellation}, {\sl Adv. Mater.} 24 (2012), OP281COP304.

\bibitem{Deng2017-1}
{\sc Y. Deng, H. Liu and G. Uhlmann}, {\em On regularized full- and partial-cloaks in acoustic scattering}, {\sl Commun. Part. Differ. Equ.} 42 (2017), 821--851.

\bibitem{Deng2017-2}
{\sc Y. Deng, H. Liu and G. Uhlmann},
{\em Full and partial cloaking in electromagnetic scattering}, {\sl Arch. Ration. Mech. Anal.} 223 (2017), 265--299.

\bibitem{Greenleaf2008}
{\sc A. Greenleaf, Y. Kurylev, M. Lassas and G. Uhlmann}, {\em Isotropic transformation optics: Approximate acoustic and quantum cloaking}, {\sl New J. Phys.}, 10 (2008), 115024.

\bibitem{Greenleaf2009-1}
{\sc A. Greenleaf, Y. Kurylev, M. Lassas and G. Uhlmann},
{\em Cloaking devices, electromagnetic wormholes and transformation optics}, {\sl SIAM Review} 51 (2009), 3--33.

\bibitem{Greenleaf2009-2}
{\sc A. Greenleaf, Y. Kurylev, M. Lassas and G. Uhlmann}, {\em Invisibility and inverse problems}, {\sl Bull. Amer. Math. Soc. (N.S.)}, 46 (2009), 55--97.

\bibitem{Greenleaf2003}
{\sc A. Greenleaf, M. Lassas and G. Uhlmann},
{\em On nonuniqueness for Calder´on’s inverse problem},
{\sl Math. Res. Lett.} 10 (2003), no. 5-6, 685--693.

\bibitem{Hele1898}
{\sc H. Hele-Shaw},
{\em The flow of water}, {\sl Nature} 58, 34  (1898).
%\bibitem{Kato1976}
%{\sc Kato T}, {\em Perturbation theory for linear operators}, New York (NY): Springer-Verlag; 1976.

\bibitem{Kocyigit2013}
{\sc I. Kocyigit, H. Liu, and H. Sun}, {\em Regular scattering patterns from near-cloaking devices and their implications for invisibility cloaking}, {\sl Inverse Problems}, 29 (2013), 045005.

\bibitem{Kohn2010}
{\sc R. Kohn, D. Onofrei, M. Vogelius and M. Weinstein},
{\em Cloaking via change of variables for the Helmholtz equation}, {\em Commu. Pure Appl. Math.}, 63 (2010), 0973--1016.

\bibitem{Kohn2008}
{\sc R. Kohn, H. Shen, M. Vogelius and M. Weinstein}, {\em Cloaking  via  change  of  variables  inelectric  impedance  tomography}, {\sl Inverse Problems}, 24 (2008), 015016.

\bibitem{Liu2009}
{\sc H. Liu}, {\em Virtual reshaping and invisibility in obstacle scattering}, {\sl Inverse Probl.} 25(4), 045006 (2009).

\bibitem{Liu2013-1}
{\sc H. Liu}, {\em On near-cloak in acoustic scattering}, {\sl J. Differ. Equ.} 254, 1230--1246 (2013).

\bibitem{Li2013}
{\sc J. Li, H. Liu and S. Mao}, {\em Approximate acoustic cloaking in inhomogeneous isotropic space}, {\sl Sci. China Math.} 56(12), 2631--2644 (2013).

\bibitem{Li2015}
{\sc J. Li, H. Liu, L. Rondi and G. Uhlmann},
{\em Regularized transformation-optics cloaking for the Helmholtz equation: from partial cloak to full cloak}, {\sl Commun. Math. Phys.} 335(2), 671--712 (2015).

\bibitem{Li2012}
{\sc J. Li, H. Liu and H. Sun}, {\em Enhanced approximate cloaking by SH and FSH lining}, {\sl Inverse Probl.} 28, 075011 (2012).

\bibitem{Liu2023}
{\sc H. Liu and Z-Q. Miao and G-H. Zheng},
{\em A mathematical theory of microscale hydrodynamic cloaking and shielding by electro-osmosis}, {\sl arXiv.2302.07495}.

\bibitem{Liu2013-2}
{\sc H. Liu and H. Sun}, {\em Enhanced near-cloak by FSH lining}, {\sl J. Math. Pures Appl.} 99(1), 17--42 (2013)

\bibitem{Liu2021}
{\sc H. Liu, W-Y. Tsui, A. Wahab and and X. Wang}
{\em Three-dimensional elastic scattering coefficients and enhancement of the elastic near cloaking}, {\sl J. Elast.} 143 (2021), 111--146.

\bibitem{Liu2011}
{\sc  H. Liu and T. Zhou}, {\em On approximate electromagnetic cloaking by transformation media}, {\sl SIAM J. Appl. Math.} 71, 218--241 (2011).

\bibitem{Zribi2016}
{\sc J. Lagha, F. Triki and H. Zribi}
{\em Small perturbations of an interface for elastostatic problems}
{sl Mathematical Methods in the Applied Sciences} 40(10): (2016) 3608--3636.

\bibitem{Pendry2006}
{\sc J. Pendry, D. Schurig and D. Smith},
{\em Controlling electromagnetic fields}, {\sl Science}, 12(5781): (2006), 1780--1782.

\bibitem{Park2019a}
{\sc J. Park, J. Youn and Y. Song},
{\em Hydrodynamic metamaterial cloak for drag-free flow}, {\sl Phys. Rev. Lett.}, 123 (2019), 074502.

\bibitem{Park2019b}
{\sc J. Park, J. Youn and Y. Song},
{\em Fluid-flow rotator based on hydrodynamic metamaterial}, {\sl Phys. Rev. Appl.}, 12 (2019), 061002.

\bibitem{Park2021}
{\sc J. Park, J. Youn and Y. Song},
{\em Metamaterial hydrodynamic flow concentrator}, {\sl Extreme Mech. Lett.}, 42 (2021), 101061.
\end {thebibliography}
\end{document}